\documentclass[a4paper]{amsart}
\usepackage{amsmath}
\usepackage{amsfonts}
\usepackage{amssymb}
\usepackage{amsthm}
\usepackage{amsbsy}
\usepackage{enumerate}
\usepackage[shortlabels]{enumitem}
\usepackage[hidelinks]{hyperref}
\usepackage{graphicx}

\newcommand{\R}{\mathbb{R}}
\newcommand{\C}{\mathbb{C}}
\newcommand{\deriv}[2][]{\frac{d #1}{d #2}}
\newcommand{\pderiv}[2][]{\frac{\partial #1}{\partial #2}}

\DeclareMathOperator{\id}{id}
\newcommand{\ii}{\sqrt{-1}}
\newcommand{\cob}[1]{C^{#1}_{\mathrm{COB}}}
\newcommand{\vmv}{\overline{C^\infty}}
\newcommand{\laggrass}{\mathcal{LG}}
\newcommand{\zeromas}{m_0}
\newcommand{\geo}{\mathcal{G}}
\newcommand{\lagcones}{\mathcal{LCO}}
\newcommand{\orientedprojective}{\mathbb{P}^+}
\DeclareMathOperator{\real}{Re}
\DeclareMathOperator{\imaginary}{Im}
\DeclareMathOperator{\vol}{vol}
\DeclareMathOperator{\crit}{Crit}

\DeclareMathOperator{\Gr}{Gr}
\DeclareMathOperator{\phase}{phase}

\DeclareMathOperator{\diff}{Diff}
\DeclareMathOperator{\im}{im}
\DeclareMathOperator{\tr}{trace}

\theoremstyle{plain}
\newtheorem{thm}{Theorem}[section]
\newtheorem{lemma}[thm]{Lemma}
\newtheorem{prop}[thm]{Proposition}
\newtheorem{cor}[thm]{Corollary}
\newtheorem{qu}[thm]{Question}
\newtheorem{conj}[thm]{Conjecture}

\theoremstyle{definition}
\newtheorem{dfn}[thm]{Definition}
\newtheorem{set}[thm]{Setting}
\newtheorem{ntn}[thm]{Notation}

\theoremstyle{remark}
\newtheorem{rem}[thm]{Remark}

\usepackage{xcolor}

\begin{document}
	\title{Special Lagrangian webbing}
	\author[J. Solomon]{Jake P. Solomon}
	\address{Institute of Mathematics\\ Hebrew University, Givat Ram\\Jerusalem, 91904, Israel}
	\email{jake@math.huji.ac.il}
	\author[A. Yuval]{Amitai M. Yuval}
	\address{Institute of Mathematics\\ Hebrew University, Givat Ram\\Jerusalem, 91904, Israel}
	\email{amitai.yuval@mail.huji.ac.il}

	\date{May 2026}
\begin{abstract}
We construct families of imaginary special Lagrangian cylinders near transverse Maslov index $0$ or $n$ intersection points of positive Lagrangian submanifolds in a general Calabi-Yau manifold. Hence, we obtain geodesics of open positive Lagrangian submanifolds near such intersection points. Moreover, this result is a first step toward the non-perturbative construction of geodesics of closed positive Lagrangian submanifolds. Also, we introduce a method for proving $C^{1,1}$ regularity of geodesics of positive Lagrangians at the non-smooth locus. This method is used to show that $C^{1,1}$ geodesics of positive Lagrangian spheres persist under small perturbations of endpoints, improving the regularity of a previous result of the authors. In particular, we obtain the first examples of $C^{1,1}$ solutions to the positive Lagrangian geodesic equation in arbitrary dimension that are not invariant under isometries. Along the way, we study geodesics of positive Lagrangian linear subspaces in a complex vector space, and prove an a priori existence result in the case of Maslov index $0$ or $n.$ Throughout the paper, the cylindrical transform introduced in previous work of the authors plays a key role.
\end{abstract}
\maketitle

	\tableofcontents
	
	\section{Introduction}
\subsection{Setting}
Let $(X,\omega,J,\Omega)$ be a \emph{Calabi-Yau manifold}. Namely, $X$ is a K\"ahler manifold with symplectic form $\omega$ and complex structure $J$, and $\Omega$ is a non-vanishing holomorphic volume form on $X$. We denote by $g$ the K\"ahler metric and by $n$ the complex dimension. Such a manifold is sometimes called almost Calabi-Yau as $g$ need not be Ricci flat, but we omit the `almost' for brevity.
	
	An oriented Lagrangian submanifold $\Lambda\subset X$, possibly immersed, is said to be \emph{positive} if $\real \Omega|_\Lambda$ is a positive volume form.
This condition is also known as being almost-calibrated~\cite{Wang}.
A positive Lagrangian submanifold is \emph{special} if $\imaginary \Omega|_\Lambda = 0.$ An oriented Lagrangian submanifold is called \emph{imaginary special} if $\real \Omega|_\Lambda = 0$ and $\imaginary \Omega|_\Lambda$ is a positive volume form. Imaginary special Lagrangians are also called~\cite{Joy07} special Lagrangians of phase $\ii$.
	
	Let $\mathcal{O}$ be a Hamiltonian isotopy class of closed smoothly embedded positive Lagrangians diffeomorphic to a given manifold $L.$ Then $\mathcal{O}$ is naturally a smooth Fr\'echet manifold. A Riemannian metric $G$ on $\mathcal{O}$ is defined in ~\cite{solomon}. It is shown in~\cite{solomon-curv} that the metric $G$ has a Levi-Civita connection and the associated sectional curvature is non-positive. The Levi-Civita connection, which we describe in detail in Section~\ref{subsection:positive Lagrangians and geodesics}, gives rise to the notion of geodesics. The geodesic equation is a fully non-linear degenerate elliptic PDE~\cite{rubinstein-solomon}. A satisfactory existence theory for these geodesics would have far-reaching consequences for the uniqueness and existence of special Lagrangian submanifolds in $\mathcal{O}$~\cite{solomon,solomon-curv} as well as rigidity of Lagrangian intersections~\cite{rubinstein-solomon}.

In~\cite{cylinders}, the authors introduce the cylindrical transform for geodesics of positive Lagrangians. A cylinder is a manifold of the form $N \times [0,1].$ The cylindrical transform of a geodesic of positive Lagrangians $(\Lambda_t)_{t \in [0,1]}$ is a one-parameter family of imaginary special Lagrangian cylinders. These cylinders satisfy the elliptic boundary condition that the boundary component corresponding to $N\times\{i\}$ is contained in $\Lambda_i$ for $i = 0,1.$ Thus, the degenerate elliptic geodesic equation is transformed to a family of elliptic equations. When the endpoints of the geodesic $\Lambda_0,\Lambda_1,$ are Lagrangian spheres intersecting transversally in two points, necessary and sufficient conditions are given for a family of imaginary special Lagrangian cylinders to arise as the cylindrical transform of a geodesic. As a consequence, it is shown that geodesics of such positive Lagrangian spheres persist under small perturbations of the endpoints. The cylindrical transform is naturally defined for geodesics of immersed Lagrangians that are smooth away from a finite number of points, where a certain type of conical singularity is allowed. We call such geodesics \emph{cone-smooth}. See Section~\ref{subsection:Lagrangians with cone points} for the precise definition. Conversely, geodesics constructed from the inverse cylindrical transform are in general cone-smooth. Below, unless otherwise mentioned, all geodesics are cone-smooth.

\subsection{Statement of results}
The present paper proves the a priori existence of families of imaginary special Lagrangian cylinders near intersection points of positive Lagrangian submanifolds of Maslov index $0$ or $n.$ As explained in Section~\ref{ssec:future}, this result is a first step toward the non-perturbative construction of geodesics of positive Lagrangian submanifolds in a general Calabi-Yau manifold. Moreover, we introduce a method for proving $C^{1,1}$ regularity of geodesics at the non-smooth locus. A geodesic is called $C^{1,1}$ if it admits a parameterization by a $C^{1,1}$ family of positive Lagrangian immersions.

Let $\Lambda_0,\Lambda_1 \subset X$ be positive Lagrangian submanifolds and let $N$ be a manifold of dimension $n-1$. We denote by $\mathcal{SLC}(N;\Lambda_0,\Lambda_1)$ the space of imaginary special Lagrangian submanifolds of $X,$ perhaps immersed, diffeomorphic to $N\times[0,1],$ such that the boundary corresponding to $N \times \{i\}$ is embedded in $\Lambda_i$ for $i = 0,1.$ Positive Lagrangian submanifolds are naturally graded in the sense of~\cite{Kon95,Sei00}, so the Maslov index of intersection points is defined absolutely. See Definition~\ref{definition:Maslov index}. In Definition~\ref{definition: interior regularity}, we recall the notion of regular convergence of a family \[
(Z_s)_{s\in(0,\epsilon)}\subset\mathcal{SLC}\left(S^{n-1};\Lambda_0,\Lambda_1\right)
\]
to an intersection point $q \in \Lambda_0\cap \Lambda_1.$ Roughly speaking, regular convergence is a necessary and sufficient condition for $(Z_s)_s$ to arise from the cylindrical transform of a geodesic of open positive Lagrangians in a neighborhood of $q.$ Here, open means not compact and without boundary. Although the metric $G$ is defined only on a Hamiltonian isotopy class of closed positive Lagrangians, the associated Levi-Civita connection and geodesic equation continue to be well-defined in the open setting as explained in Definition~\ref{definition:geodesic}.
	\begin{thm}
		\label{theorem:spider web}
Let $\Lambda_0,\Lambda_1\subset X$ be smoothly embedded positive Lagrangians intersecting transversally at a point $q$ with Maslov index $0.$
		\begin{enumerate}[label=(\alph*)]
			\item\label{regular family of cylinders} There exists a one-parameter family $(Z_s)_{s\in(0,\epsilon)}\subset\mathcal{SLC}\left(S^{n-1};\Lambda_0,\Lambda_1\right)$ converging regularly to $q.$
			\item\label{geodesic of open C1 Lags} There exist open neighborhoods, $q\in U_i\subset\Lambda_i,\;i=0,1,$ which are connected by a $C^{1,1}$ geodesic $(U_t)_{t\in[0,1]}$ of open positive Lagrangians.
		\end{enumerate}
	\end{thm}
\begin{rem}\label{rem:tri}
For $\Lambda_0,\Lambda_1 \subset X$ positive Lagrangians and $q \in \Lambda_0 \cap\Lambda_1,$ the Maslov index $m(q;\Lambda_0,\Lambda_1)$ satisfies
\[
m(q;\Lambda_0,\Lambda_1) = n - m(q;\Lambda_1,\Lambda_0).
\]
Also, the time parameter of a geodesic from $\Lambda_0$ to $\Lambda_1$ can be reversed to obtain a geodesic from $\Lambda_1$ to $\Lambda_0.$
So, Theorem~\ref{theorem:spider web} holds for $q$ of Maslov index $n$ as well as $0.$
\end{rem}
The family of imaginary special Lagrangian cylinders $(Z_s)_{s}$ of Theorem~\ref{theorem:spider web} is depicted in Figure~\ref{fig:lw}. We call such a family special Lagrangian webbing.

\begin{figure}[ht]
\centering
\includegraphics[width=12cm]{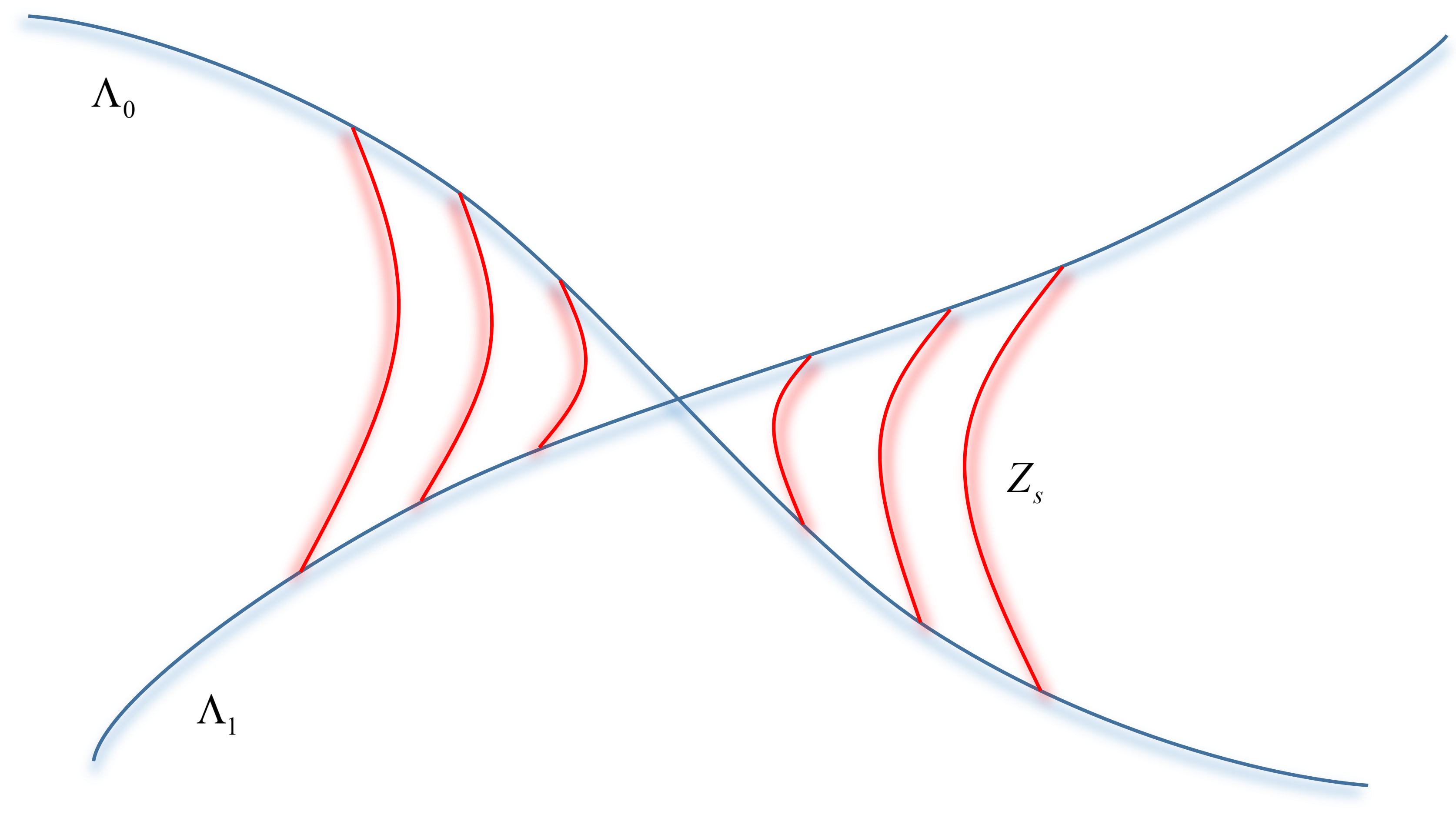}
\caption{A family of imaginary special Lagrangian cylinders $(Z_s)_s$ emanating from a Maslov zero intersection point of Lagrangian submanifolds $\Lambda_0,\Lambda_1,$ as in Theorem~\ref{theorem:spider web}.}
\label{fig:lw}
\end{figure}

	In~\cite[Theorem~1.6]{cylinders} it is shown that a cone-smooth geodesic $(\Lambda_t)_{t \in [0,1]}$ connecting positive Lagrangian spheres $\Lambda_0,\Lambda_1,$ that intersect transversally at two points persists under small perturbations of $\Lambda_0,\Lambda_1.$ The following theorem shows that if the geodesic $(\Lambda_t)_t$ is $C^{1}$ regular at the non-smooth locus, then it is in fact $C^{1,1},$ and moreover, this regularity persists under perturbations. Let $\mathcal{O}$ be a Hamiltonian isotopy class of smoothly embedded positive Lagrangian spheres and let $\mathfrak{G}_\mathcal{O}$ denote the space of geodesics $(\Lambda_t)_{t \in [0,1]}$ with $\Lambda_0,\Lambda_1 \in \mathcal{O}$ intersecting transversally at two points. We refer the reader to Definition~\ref{dfn:topology on geodesics} for the strong and weak $C^{1,\alpha}$ topologies on $\mathfrak{G}_\mathcal{O}.$

	\begin{thm}
		\label{theorem:perturbed C1 geodesic}
		Let $\Lambda_0,\Lambda_1\in\mathcal{O}$ intersect transversally at exactly two points. Suppose there exists a $C^1$ geodesic $(\Lambda_t)_{t\in[0,1]}$ between $\Lambda_0$ and $\Lambda_1.$ Let $\alpha\in(0,1).$ Then, there exists a $C^{2,\alpha}$-open neighborhood $\mathcal{Y}$ of $\Lambda_1$ in $\mathcal{O}$ and a weak $C^{1,\alpha}$-open neighborhood $\mathcal X$ of $(\Lambda_t)_{t \in [0,1]}$ in $\mathfrak{G}_\mathcal{O}$ such that for every $\Lambda\in\mathcal{Y}$ there exists a unique geodesic between $\Lambda_0$ and $\Lambda$ in $\mathcal{X}.$ This geodesic  has regularity $C^{1,1}$ and depends continuously on $\Lambda$ with respect to the $C^{2,\alpha}$ topology on $\mathcal{Y}$ and the strong $C^{1,\alpha}$ topology on $\mathcal{X}.$
	\end{thm}
In~\cite{solomon-yuval}, there are examples of geodesics of positive Lagrangians of arbitrary dimension, many of which satisfy the hypothesis of Theorem~\ref{theorem:perturbed C1 geodesic}. However, they are all preserved by an isometric action of $O(n)$ on the ambient manifold $X.$ From Theorem~\ref{theorem:perturbed C1 geodesic} we obtain the following.
\begin{cor}
There exist $C^{1,1}$ geodesics of closed positive Lagrangians in arbitrary dimension that are not invariant under any isometries of the ambient manifold.
\end{cor}
To prove Theorems~\ref{theorem:spider web} and~\ref{theorem:perturbed C1 geodesic}, we use a blow-up argument, which can be summarized as follows.
The critical locus of a geodesic of positive Lagrangians $(\Lambda_t)_t$ is defined by
\[
\crit((\Lambda_t)_t) = \bigcap_t \Lambda_t.
\]
To each point $p$ of each Lagrangian $\Lambda_t,$ using the definition of cone-smoothness, we associate a unique tangent cone $TC_p \Lambda_t.$ If $\Lambda_t$ is $C^1$ regular at $p,$ then $TC_p\Lambda_t = T_p\Lambda_t.$
For $q \in \crit((\Lambda_t)_t),$ Lemma~\ref{lemma:geodesic of tangent cones} shows that the family of tangent cones $TC_q\Lambda_t\subset T_qX$ is a geodesic from $T_q\Lambda_0$ to~$T_q\Lambda_1.$ Conversely, given a geodesic of positive Lagrangian cones in $T_qX$ from $T_q\Lambda_0$ to $T_q\Lambda_1,$ Lemma~\ref{lemma:integration of tangent cone} shows how to construct open neighborhoods, $q\in U_i\subset\Lambda_i,\;i=0,1,$ and a geodesic $(U_t)_{t\in[0,1]}$ of open positive Lagrangians connecting them. Non-smooth points of a cone-smooth geodesic are always critical points. Lemma~\ref{lemma:C11 regularity from blowup} shows that if the tangent cones at a critical point $q$ of a geodesic are all linear, then the geodesic is of regularity $C^{1,1}$ at $q.$ The final ingredient in Theorem~\ref{theorem:spider web} is the following a priori existence result for geodesics of positive Lagrangian linear subspaces. Let $\laggrass^+(n)$ denote the Grassmannian of positive Lagrangian linear subspaces in $\C^n$ with the standard Calabi-Yau structure. For $\Lambda \in \laggrass^+(n),$ the tangent space $T_\Lambda \laggrass^+(n)$ is canonically isomorphic to the space of quadratic forms on $\Lambda.$ For $\Lambda_0, \Lambda_1 \in \laggrass^+(n),$ we abbreviate $m(\Lambda_0,\Lambda_1)$ for the Maslov index of $0 \in \Lambda_0 \cap \Lambda_1.$
	\begin{thm}
		\label{theorem:linear geodesic}
		Let $\Lambda_0,\Lambda_1\in\laggrass^+(n)$ with $m(\Lambda_0,\Lambda_1)=0.$ Then there exists a geodesic in $\laggrass^+(n)$ between $\Lambda_0$ and $\Lambda_1$ with negative semi-definite derivative. If, in addition, $\Lambda_0$ and $\Lambda_1$ intersect transversally, the geodesic has negative definite derivative.
	\end{thm}
By Remark~\ref{rem:tri}, Theorem~\ref{theorem:linear geodesic} applies also when $m(\Lambda_0,\Lambda_1) = n,$ only the derivative of the geodesic is positive instead of negative. The final ingredient in the proof of Theorem~\ref{theorem:perturbed C1 geodesic} is Proposition~\ref{proposition:geodesic is in fact linear}, which asserts that geodesics of positive Lagrangian linear subspaces are stable under deformations in the space of geodesics of positive Lagrangian cones. The cylindrical transform plays a key role in several steps of the preceding summary. In particular, the ellipticity of the family of special Lagrangian boundary value problems produced by the cylindrical transform is a crucial ingredient in inverse function theorem arguments.

\subsection{Directions for future research}\label{ssec:future}
\subsubsection{From local to global}
Theorem~\ref{theorem:spider web} can be viewed as a special Lagrangian analog of Bishop's result~\cite{Bis65} on the existence of a family of holomorphic disks near a point of an $n$-dimensional real submanifold of $\C^n$ with tangent space containing a complex line. The extension of Bishop's local families of holomorphic disks to global families is the basis for foundational results in symplectic and contact geometry~\cite{Eli90,Hof93}.
A fundamental question for future research is the following.
\begin{qu}\label{qu:nl}
Under what conditions can one extend the family of imaginary special Lagrangian cylinders $(Z_s)_s$ of Theorem~\ref{theorem:spider web} to obtain the cylindrical transform of a geodesic of positive Lagrangians between $\Lambda_0$ and $\Lambda_1$?
\end{qu}
We outline briefly what is needed to address Question~\ref{qu:nl}. Theorem~1.2 of~\cite{cylinders} asserts that the space of imaginary special Lagrangian cylinders $\mathcal{SLC}(N;\Lambda_0,\Lambda_1)$ is a $1$-dimensional smooth manifold for $N$ an arbitrary closed connected $(n-1)$-manifold. So, the family of imaginary special Lagrangian cylinders $(Z_s)_s$ of Theorem~\ref{theorem:spider web} extends until a singularity develops. For example, a singularity can develop as the cylinders $Z_s$ approach an intersection point $q \in  \Lambda_0 \cap \Lambda_1.$ More precisely, a submanifold $Y_s \subset Z_s$ diffeomorphic to $S^{m-1} \times [0,1],$ where $m$ is the Maslov index $m(q;\Lambda_0,\Lambda_1),$ can contract to a point. When $m < n,$ a gluing argument should show that the family $(Z_s)_s$ can be continued through the singularity after performing a surgery that replaces a tubular neighborhood of $Y_s$ in $Z_s$ diffeomorphic to $S^{m-1} \times [0,1] \times D^{n-m}$ with a copy of $S^{n-m-1} \times [0,1] \times D^{m}.$ When $m = n,$ the family of cylinders $(Z_s)_s$ contracts to the point $q.$ If one could show that no other singularities develop, so that after a series of surgeries the family $(Z_s)_s$ must eventually contract to a point of Maslov index $n,$ it should follow that the family $(Z_s)_s$ arises as the cylindrical transform of a geodesic from $\Lambda_0$ to $\Lambda_1.$

We plan to treat this program in detail in future work.
In dimension $n = 2,$ special Lagrangian cylinders are equivalent to holomorphic annuli by hyperk\"ahler rotation, and Gromov's work~\cite{Gro85} gives a great deal of control over which singularites can appear. In higher dimensions, much less is known. However, there is an extensive literature on special Lagrangian singularities~\cite{BeR12,Has04a,Has04b,HaK07,HaK13,Joy02,Joy03,Joy04b,Joy04c,Joy05,Joy08}.

Beyond the existence problem for geodesics of special Lagrangians, progress on Question~\ref{qu:nl} could potentially lead to progress on the nearby Lagrangian conjecture of Arnol'd as we now explain. Let $M$ be a smooth closed manifold and let $\alpha$ denote the Liouville $1$-form on the cotangent bundle $T^*M.$ A Lagrangian submanifold $\Lambda \subset T^*M$ is \emph{exact} if $\alpha|_\Lambda$ is exact. The nearby Lagrangian conjecture is the following.
\begin{conj}[Arnol'd]\label{conj:nl}
An exact closed Lagrangian submanifold $\Lambda \subset T^*M$ is Hamiltonian isotopic to the zero section.
\end{conj}

Recently, there has been dramatic progress on this conjecture~\cite{Abo12,AbK18,FSS08,Kra13,Nad09,NZ09}, culminating in the result that an exact closed Lagrangian $\Lambda \subset T^*M$ is simply homotopy equivalent to $M.$ The full conjecture has been proved for $M = S^2$ in~\cite{Hin12}. Nonetheless, in general, we are still far from being able to construct a Hamiltonian isotopy from an exact closed Lagrangian $\Lambda \subset T^*M$ to the zero section.

Suppose we are given a compatible complex structure and a holomorphic volume form on a neighborhood $U$ of the zero section of $T^*M$ such that the zero section is positive Lagrangian. These always exist by the Grauert tube construction~\cite{Gra58,GuS91,LeS91}.
Let $\Lambda_0$ be the zero section of $T^*M$ and let $\Lambda_1$ be an exact closed positive Lagrangian in $U \subset T^*M$ intersecting $\Lambda_0$ transversally. It follows from the above cited results on the nearby Lagrangian conjecture that there exists at least one intersection point of $\Lambda_0$ and $\Lambda_1$ of Maslov index zero. So, Theorem~\ref{theorem:spider web} applies. A geodesic of positive Lagrangians is in particular a Hamiltonian isotopy, so progress on Question~\ref{qu:nl} gives progress on Conjecture~\ref{conj:nl}.

\subsubsection{Local questions}
In our proof of Theorem~\ref{theorem:linear geodesic}, we use in an essential way the assumption that $m(\Lambda_0,\Lambda_1) = 0,n.$ Thus, the following seems fundamental.
\begin{qu}\label{qu:laggrass}
For general $\Lambda_0,\Lambda_1 \in \laggrass^+(n),$ does there exist a geodesic in $\laggrass^+(n)$ between $\Lambda_0$ and $\Lambda_1$?
\end{qu}
A positive answer to Question~\ref{qu:laggrass} would provide a model at the level of tangent cones for geodesics of positive Lagrangians near critical points of index strictly between $0$ and $n,$ completing the picture given by Theorem~\ref{theorem:linear geodesic}. However, even when $m(\Lambda_0,\Lambda_1) = 0,n,$ as in Theorem~\ref{theorem:linear geodesic}, it is not clear whether there are more complicated models for tangents cones at critical points.

\begin{qu}\label{qu:lagcones}
For $\Lambda_0,\Lambda_1 \in \laggrass^+(n),$ is it possible that there exists a geodesic from $\Lambda_0$ to $\Lambda_1$ of positive Lagrangian cones that are not linear subspaces?
\end{qu}

If the answer to Question~\ref{qu:lagcones} is the affirmative for $m(\Lambda_0,\Lambda_1) = 0,$ Lemma~\ref{lemma:integration of tangent cone} gives geodesics of open positive Lagrangians in an arbitrary Calabi-Yau manifold that are not $C^1.$

\subsubsection{Regularity limitations}
The limited regularity of geodesics of positive Lagrangian submanifolds is reminiscent of the limited regularity of geodesics in the space of K\"ahler metrics~\cite{chen-kaehlermetrics,chen-tian,CTW17,DaL12} and in the space of positive $(1,1)$-forms~\cite{chu2020space,collins2018moment}.
However, the above results on $C^{1,1}$ regularity are of a different nature than $C^{1,1}$ regularity in the K\"ahler metric or positive $(1,1)$-form context. On the one hand, as explained in~\cite{rubinstein-solomon}, a geodesic of positive Lagrangian submanifolds can be viewed locally as the graph of the gradient of a potential function analogous to the potential functions for K\"ahler metrics and positive $(1,1)$-forms. The $C^{1,1}$ regularity of Theorems~\ref{theorem:spider web} and~\ref{theorem:perturbed C1 geodesic} translates to $C^{2,1}$ regularity for the local potential function unlike the $C^{1,1}$ optimal result for potential functions of geodesics of K\"ahler metrics and positive $(1,1)$-forms. In fact, it may well be possible to obtain even higher regularity in the settings of Theorem~\ref{theorem:spider web} and~\ref{theorem:perturbed C1 geodesic}. On the other hand, we consider only the case that all critical points of geodesics are non-degenerate with Maslov index either $0$ or $n.$ In greater generality, it is likely that geodesics of positive Lagrangian submanifolds exhibit regularity limits similar to geodesics of K\"ahler metrics or positive $(1,1)$-forms.

\subsection{Outline}
In Section~\ref{sec:background} we recall relevant definitions and results from~\cite{cylinders}. In Section~\ref{section:the positive Lagrangian Grassmannian} we study geodesics of linear positive Lagrangians and prove Theorem~\ref{theorem:linear geodesic}. Section~\ref{section:geodesics of Lagrangian cones} shows that geodesics of positive Lagrangian linear subspaces are stable under deformations in the space of geodesics of positive Lagrangian cones. The precise statement is given in Proposition~\ref{proposition:geodesic is in fact linear}. In Section~\ref{sec:blowup} we develop a blowup procedure that relates the behavior of a geodesic of positive Lagrangians near a critical point and the associated geodesic of tangent cones. In particular, if the tangent cones are linear, the geodesic is $C^{1,1}$ regular. The section concludes with the proof of Theorem~\ref{theorem:perturbed C1 geodesic}. In Section~\ref{section:geodesics of small open Lagrangians} we show how to construct a special Lagrangian webbing from a geodesic of special Lagrangian cones and we prove Theorem~\ref{theorem:spider web}.

\subsection{Acknowledgements}
The authors would like to thank T. Collins, Y. Eliashberg, Y. Rubinstein, P. Seidel, and G. Tian, for helpful conversations.
The authors were partially supported by ERC starting grant 337560 and BSF grant 2016173. The first author was partially supported by BSF Grant 2024293 and the Miriam and Julius Vinik Chair in Mathematics. The second author was partially supported by the Adams Fellowship Program of the Israel Academy of Sciences and Humanities.

\section{Background}\label{sec:background}
	
	This article relies on the notation and results of~\cite{cylinders} summarized below.

\subsection{Immersed Lagrangians}
	\label{subsection:immersed lagrangians}

Let $N,M$ be smooth manifolds, $M$ perhaps with boundary. Denote by $\diff(M)$ the diffeomorphisms of
$M$	preserving each boundary component. That is, if $\varphi \in \diff(M)$ and $B \subset \partial M$ is a component, then $\varphi(B) = B.$
\begin{dfn}
		\label{definition: immersed submanifold}
		An \emph{immersed (resp. embedded) submanifold of $N$ of type $M$} is an equivalence class of immersions (resp. embeddings),
		\[
		K=[f:M\to N],
		\]
		where the equivalence is with respect to the $\diff(M)$-action: The immersions $f$ and $f'$ are equivalent if there exists $\varphi\in\diff(M)$ such that
		\[
		f'=f\circ\varphi.
		\]
We say that $K=[f]$ is \emph{free} if $f$ has trivial isotropy subgroup. We say that $K$ has boundary if $M$ does. In this case, to each boundary component of $M$ we associate a boundary component of $K,$ which is itself an immersed submanifold. A \emph{differential form} on $K$ is an equivalence class of pairs $\eta = [(f,\tau)]$ where $f$ is a representative of $K$ and $\tau \in \Omega^*(M).$ The pairs $(f,\tau)$ and $(f',\tau')$ are equivalent if there exists $\varphi \in \diff(M)$ such that $f' = f \circ \varphi$ and $\tau' = f^* \tau.$
	\end{dfn}	

	Let $(X,\omega)$ be a symplectic manifold of dimension $2n$ and let $L$ be a smooth manifold of dimension $n.$ An immersion $f:L\to X$ is \emph{Lagrangian} if it satisfies $f^*\omega=0.$ The diffeomorphism group $\diff(L)$ acts on Lagrangian immersions $L\to X$ by composition. An \emph{immersed Lagrangian submanifold} in $X$ of type $L$ is a submanifold of type $L$ that can be represented by a Lagrangian immersion and thus all representatives are Lagrangian.
	
	Suppose now that $L$ is closed. We let $\mathcal{L}(X,L)$ denote the space of free immersed Lagrangian submanifolds in $X$ of type $L.$ The space $\mathcal{L}(X,L)$ is a smooth Fr\'echet manifold. See~\cite{akveld-salamon} for a discussion in the case of embedded Lagrangians; the generalization to free immersed Lagrangians is performed in~\cite[Theorem 2.12]{cylinders} using ideas from~\cite[Theorem~1.5]{cervera-mascaro-michor}. For $\Lambda\in\mathcal{L}(X,L),$ contraction with $\omega$ yields the isomorphism
	\[
	T_\Lambda\mathcal{L}(X,L)\cong\left\{\left.\tau\in\Omega^1(\Lambda)\;\right|\;d\tau=0\right\}.
	\]
This is proved in~\cite{akveld-salamon} for embedded Lagrangians. The proof for the immersed case is essentially the same. A path $(\Lambda_t)_{t\in[0,1]}$ in $\mathcal{L}(X,L)$ is said to be \emph{exact} if its derivative, $\deriv{t}\Lambda_t,$ is exact for $t\in[0,1].$

	\subsection{Positive Lagrangians and geodesics}
	\label{subsection:positive Lagrangians and geodesics}
	
	In the literature, there are different, non-equivalent definitions of Calabi-Yau manifolds. Here is the definition used throughout this work.
	
	\begin{dfn}
		\label{definition:Calabi-Yau}
		A \emph{Calabi-Yau manifold} is a quadruple $(X,\omega,J,\Omega),$ where $(X,\omega)$ is a symplectic manifold, $J$ is an $\omega$-compatible integrable complex structure, and $\Omega$ is a non-vanishing holomorphic volume form on $X.$
	\end{dfn}

	Let $(X,\omega,J,\Omega)$ be Calabi-Yau and let $\Lambda\subset X$ be an oriented smooth Lagrangian submanifold perhaps immersed. Then $\Omega|_\Lambda$ is non-vanishing~\cite{harvey-lawson}. In fact, we have
	\[
	\Omega|_\Lambda=\rho e^{\ii\theta_\Lambda}\vol_\Lambda,
	\]
	where $\vol_\Lambda$ denotes the Riemannian volume form with respect to the K\"ahler metric, the positive function $\rho:X\to\mathbb{R}$ is determined by
	\begin{equation}
	\label{equation:rho}
	\frac{\rho^2\omega^n}{n!}=(-1)^{\frac{n(n-1)}{2}}\left(\frac{\ii}{2}\right)^n\Omega\wedge\overline{\Omega},
	\end{equation}
	and $\theta_\Lambda:\Lambda\to S^1$ is called the \emph{phase function} of $\Lambda.$ We say $\Lambda$ is \emph{positive} if $\real\Omega|_\Lambda$ is positive. In this case, the phase may be regarded as a real-valued function admitting values in the interval $\left(-\frac{\pi}{2},\frac{\pi}{2}\right).$ We say $\Lambda$ is \emph{special-Lagrangian} if $\theta_\Lambda\equiv0.$ If we have $\theta_\Lambda\equiv\frac{\pi}{2},$ we say $\Lambda$ is \emph{imaginary special-Lagrangian}.
	
Let $L$ be a closed manifold, let $\mathcal{O}$ be an exact isotopy class of positive Lagrangians in $X$ of type $L,$ and let $\Lambda\in\mathcal{O}.$ Positivity gives rise to the isomorphism
	\[
	T_\Lambda\mathcal{O}\cong\vmv(\Lambda):=\left\{h\in C^\infty(\Lambda)\;\left|\;\int_\Lambda h\real\Omega=0\right.\right\}.
	\]
Hence, we think of vectors tangent to $\mathcal{O}$ as functions.
	
	\begin{dfn}
		\label{definition:liftings and horizontal liftings}
		Let $\mathcal{O}$ be as above and let $(\Lambda_t)_{t\in[0,1]}$ be a smooth path in $\mathcal{O}.$ A \emph{lifting} of $(\Lambda_t)_t$ is a smooth family of Lagrangian immersions, $(\Psi_t:L\to X)_{t\in[0,1]},$ such that $\Psi_t$ represents $\Lambda_t$ for $t\in[0,1].$ A lifting $(\Psi_t)_t$ is \emph{horizontal} if it satisfies
		\[
		i_{\deriv{t}\Psi_t}\real\Omega=0,\quad t\in[0,1].
		\]
	\end{dfn}

It is shown in~\cite{solomon} that, given a smooth path $(\Lambda_t)_{t\in[0,1]}$ in $\mathcal{O},$ every representative $\Psi_0:L\to X$ of $\Lambda_0$ extends uniquely to a horizontal lifting of $(\Lambda_t)_t.$ We thus define a connection on $\mathcal{O}$ as follows. Let $(h_t)_{t\in[0,1]}$ be a  vector field along a path $(\Lambda_t)_t.$ Choose a horizontal lifting $(\Psi_t)_t.$ Then the covariant derivative of $(h_t)_t$ is given by
	\[
	\frac{D}{dt}h_t:=\left(\Psi_t\right)_*\left(\deriv{t}h_t\circ\Psi_t\right).
	\]
The definition is independent of the choice of $(\Psi_t)_t.$
	This connection, to which we refer as \emph{the positive Lagrangian connection}, is studied thoroughly in~\cite{solomon-curv}, where it is shown to be the Levi-Civita connection of the Riemannian metric on $\mathcal{O}$ defined by
	\begin{equation}
	\label{equation:Jake's Riemannian metric}
	(h,k):=\int_\Lambda hk\real\Omega, \qquad h,k \in  	T_\Lambda\mathcal{O}\cong\vmv(\Lambda).
	\end{equation}
The main objects of the present work are geodesics with respect to the positive Lagrangian connection. Below, we extend the notion of geodesics to isotopy classes of Lagrangians that may be open and/or non-smooth. See Definition~\ref{definition:geodesic}.

\subsection{Cone-smooth maps}\label{subsection:cone smooth maps}
Below we summarize definitions and results concerning cone-smooth maps from~\cite{cylinders} . Intuitively a cone-smooth map is smooth everywhere but a finite number of points, where it may not be differentiable, but it still admits a $1$-homogeneous first order approximation that is smooth away from zero. The cone-smooth generalization of an immersed submanifold has a finite number of points where it may not have a tangent space but only a tangent cone that is smoothly immersed away from its vertex. See Definition~\ref{dfn:cone-immersed submanifold terminology} below. The terminology ``cone-smooth" was chosen because cone-smooth maps can have conical singularities while otherwise behaving much like smooth maps. The following notations and definitions are taken from Section~3 of~\cite{cylinders}.

\begin{ntn}\label{ntn:blowup}
Let $M$ be a smooth manifold and let $S \subset M$ be a finite subset. We denote by $\pi : \widetilde M_S \to M$ the \emph{oriented blowup} of $M$ at $S.$ For $p \in S,$ we denote by $E_p=\pi^{-1}(p)$ the \emph{exceptional sphere over $p$.} For a detailed account of the definition of the oriented blowup, see~\cite[Definition 3.1]{cylinders}.
\end{ntn}

	\begin{dfn}
		\label{definition: cone differentiability}
		Let $M$ and $N$ be smooth manifolds, let $p\in M,$ and let $\Psi:M\to N$ be continuous. Let $\widetilde{M}_p,E_p$ and $\pi:\widetilde{M}_p\to M$ be as in Notation~\ref{ntn:blowup}.
		\begin{enumerate}
			\item The map $\Psi$ is said to be \emph{cone-smooth} at $p$ if there exists an open $E_p\subset\widetilde{U}\subset\widetilde{M}_p$ such that the composition $\Psi\circ\pi|_{\widetilde{U}}:\widetilde{U}\to N$ is smooth.
			\item Suppose $\Psi$ is cone-smooth at $p.$ The \emph{cone-derivative} of $\Psi$ at $p$ is the unique map
			\[
			d\Psi_p:T_pM\to T_{\Psi(p)}N
			\]
			satisfying the equality
			\[
			d(\Psi\circ\pi)_{\widetilde{p}}=d\Psi_p\circ d\pi_{\widetilde{p}}
			\]
			for $\widetilde{p}\in E_p.$ One verifies that the cone-derivative is well-defined and homogeneous of degree 1. Also, the restricted map $d\Psi_p|_{T_pM\setminus\{0\}}$ is smooth. Nevertheless, the cone-derivative is not linear in general.
			\item Suppose $\Psi$ is cone-smooth at $p.$ We say $\Psi$ is \emph{cone-immersive} at $p$ if the restricted map
			\[
			d\Psi_p|_{T_pM\setminus\{0\}}:T_pM\setminus\{0\}\to T_{\Psi(p)}N
			\]
			is a smooth immersion. In particular, in this case we have $d\Psi_p(v)\ne0$ for $0\ne v\in T_pM.$
		\end{enumerate}
	\end{dfn}

In the following definition and throughout the present work, we use the definition of manifold with corners from~\cite{Joy09}.

	\begin{dfn}\label{definition:cone immersed submanifold}
		Let $M$ and $N$ be smooth manifolds. Let $S\subset M$ be a finite subset and $\Psi:M\to N$ a continuous map.
		\begin{enumerate}[label=(\alph*)]
        \item
    We say the pair $(\Psi,S)$ is \emph{cone-smooth} if $\Psi$ is smooth away from $S$ and cone-smooth at every element of $S.$ That is, the composition $\Psi\circ \pi : \widetilde M_S \to N$ is smooth. For $\Theta$ a manifold with corners, a family of maps $((\Psi_t,S))_{t \in \Theta}$ is cone-smooth if the composition $\Psi_t \circ \pi : \widetilde M_S \to N$ gives a smooth map $\widetilde M_S \times \Theta \to N.$
			\item We say the pair $(\Psi, S)$ is a \emph{cone-immersion} from $(M,S)$ to $N$ if $\Psi$ is a smooth immersion away from $S$ and cone-immersive at every element of $S.$
    \item
    Suppose $N = M$ and $\Psi(p) = p$ for all $p \in S.$ The pair $(\Psi,S)$ is a \emph{cone-diffeomorphism} of $(M,S)$ if $\Psi$ is a smooth diffeomorphism away from $S$ and for $p \in S,$ the cone derivative $d\Psi_p|_{T_pM\setminus \{0\}} : T_pM\setminus \{0\} \to T_pM$ is a diffeomorphism onto $T_pM \setminus \{0\}.$ We let $\diff(M,S)$ denote the group of cone-diffeomorphisms of $(M,S)$ that act trivially on the set of connected components.
			\item Let the diffeomorphism group $\diff(M,S)$ act on cone-immersions from $(M,S)$ to $N$ by composition. A \emph{cone-immersed submanifold of $N$ of type $(M,S)$} is an orbit of the $\diff(M,S)$-action.
    \item\label{item:orientation} Suppose $M$ is orientable and let $\diff^+(M,S)\triangleleft \diff(M,S)$ denote the normal subgroup of orientation preserving cone-smooth diffeomorphisms. An \emph{orientation} on a cone-immersed submanifold $K$ of $N$ of type $(M,S)$ is an equivalence class of pairs $(O,C)$ where $O$ is an orientation of $M$ and $C$ is a $\diff^+(M,S)$ orbit inside the $\diff(M,S)$ orbit $K.$ There is a natural $\diff(M,S)/\diff^+(M,S)$ action on such pairs and this gives rise to the desired equivalence relation.
		\end{enumerate}
	\end{dfn}

\begin{dfn}\label{dfn:cone-immersed submanifold terminology}
	Let $K=[(\Psi:M\to N,S)]$ be a cone-immersed submanifold of type $(M,S).$ We let $\im(K)$ denote the image of $K$ in $M.$ That is,
\[
\im(K) = \Psi(M).
\]
A \emph{point} $p$ in $K$ is an equivalence class of pairs $((\chi,S),q),$ where $(\chi: M \to N,S)$ is a representative of $K$ and $q\in M.$ We let $\im(p)$ denote the image of $p$ in $N.$ That is, $\im(p) = \chi(q)$ for $((\chi,S),q)$ a representative of $p.$ The cone-immersed submanifold $K$ thus has a well-defined \emph{cone locus}
	\[
	K^C:=\{[((\Psi,S),c)]\;|\;c\in S\}.
	\]
A \emph{cone point} is an element of the cone locus. We define the \emph{tangent cone} of $K$ at a point $p = [((\Psi,S),q)]$ to be the cone-immersed submanifold
\[
TC_pK := [(d\Psi_c : T_cM \to T_{\im(p)}N,\{0\})]
\]
of $T_{\im(p)}N.$ The tangent cone $TC_pK$ is indeed a cone, that is, invariant under scalar multiplication. Moreover, it is independent of the choice of $\Psi.$ If $\Psi$ is smooth at $q,$ then $TC_pK$ is smoothly embedded and recovers the usual notion of tangent space. For $h : V \to W$ a homogeneous map of real vector spaces that does not vanish on $V \setminus\{0\},$  we denote by $
\orientedprojective(h) : \orientedprojective(V) \to \orientedprojective(W)$ the oriented projectivization.
We define the \emph{projective tangent cone} of $K$ at $p$ by
	\[
	\orientedprojective(TC_pK):=\left[\orientedprojective(d\Psi_c) :\orientedprojective(T_cM)\to \orientedprojective(T_{\im(p)}N)\right].
	\]
	This is a smooth immersed sphere in $\orientedprojective(T_{\im(p)}N).$

A \emph{function} on $K$ is an equivalence class of pairs $((\Psi,S),f),$ where $(\Psi,S)$ is a representative of $K$ and $f$ is a function on $M.$ We say the function $h=[((\Psi,S),f)]$ is \emph{cone-smooth} at the point $p=[((\Psi,S),q)]$ if $f$ is cone-smooth at $q.$ In this case $h$ has a well-defined \emph{cone-derivative} $dh_p : TC_pK \to \R,$ which is a degree-1 homogeneous function.
\end{dfn}

Let $(\Psi : M \to N, S)$ be a cone-smooth map and let $p \in S.$ Recall that the cone derivative $d\Psi_p : T_pM \to T_{\Psi(p)}N$ is homogeneous of degree $1$ and $d\Psi_p|_{T_pM \setminus\{0\}}$ is smooth. It follows that for $0 \neq v \in T_pM$ and $\lambda > 0,$ we have $d(d\Psi_p)_{v} = d(d\Psi_p)_{\lambda v}$ under the canonical identification $T_v T_pM \simeq T_pM \simeq T_{\lambda v}T_pM.$ Thus, for $\widetilde p = [v] \in \orientedprojective(T_pM),$ we define
\[
d(d\Psi_p)_{\widetilde p} : = d(d\Psi_p)_v : T_pM \to T_{\Psi(p)}N.
\]
The following is Lemma~3.16 from~\cite{cylinders}.
\begin{lemma}\label{lemma:blowupdifferential}
Let $(\Psi : M \to N, S)$ be a cone-smooth map. Consider the map $d\left(\Psi|_{M\setminus S}\right) : TM|_{M \setminus S} \to \Psi^* TN|_{M \setminus S}.$ Pulling back by $\pi$ gives a map
\[
\pi^* d\left(\Psi|_{M\setminus S}\right) : \pi^* TM|_{\widetilde M_S^\circ} \to \pi^* \Psi^* TN|_{\widetilde M_S^\circ}.
\]
This map extends uniquely to a map of bundles $\widetilde{d\Psi}:\pi^* TM \to \pi^* \Psi^* TN.$ Moreover, for $p \in S$ and $\widetilde p \in E_p,$ we have
\begin{equation}\label{equation:dderivatives}
\widetilde{d\Psi}_{\widetilde p} = d(d\Psi_p)_{\widetilde p}.
\end{equation}
In particular, if $\Psi$ is a cone-immersion, then $\widetilde{d\Psi}$ is an injective map of vector bundles.
\end{lemma}

\begin{dfn}\label{definition: many tangent spaces}
Let $S \subset M$ be a finite subset. The \emph{blowup tangent bundle} of $(M,S)$ is the bundle $\widetilde{TM}_S : = \pi^*TM \to \widetilde{M}_S.$ When clear from the context, we may omit the subscript $S.$ Let $(\Psi : M \to N, S)$ be a cone-smooth map. The \emph{blowup differential} of $\Psi$ is the map
\[
\widetilde{d \Psi} : \widetilde{TM}_S \to \pi^* \Psi^* TN
\]
given by Lemma~\ref{lemma:blowupdifferential}.
	\end{dfn}

Let $M$ be a smooth manifold and let $S \subset M$ be a finite subset and let $\pi : \widetilde{M}_S \to M$ denote the blowup projection. Given a differential form $\alpha$ on $M$ we can pull-back $\alpha$ as a section of $\Lambda^*(T^*M)$ to obtain a section of $\pi^*\Lambda^*\left(T^*M\right) \simeq \Lambda^*\left(\widetilde{TM}^*_S\right).$ We denote this pull-back by $\pi^{-1}\alpha.$ Observe that this pull-back is different from the pull-back of $\alpha$ as a differential form, $\pi^*\alpha,$ which would be a section of $\Lambda^*(T^*\widetilde{M}_S).$
The following are taken from Definition~3.18, Remark~3.19, Remark~3.27, Definition~3.28 and Definition~3.29 of~\cite{cylinders}.
\begin{dfn}\label{definition: cone-smooth differential form etc}
A \emph{cone-smooth differential form} on $(M,S)$ is a smooth differential form $\alpha$ on $M \setminus S$ such that $\pi^{-1} \alpha$ extends to a smooth section $\widetilde{\alpha}$ of $\Lambda^*\left(\widetilde{TM}_S^*\right).$ We call $\widetilde{\alpha}$ the \emph{blowup form}. We say that $\alpha$ and $\widetilde{\alpha}$ are \emph{closed} if $\alpha$ is closed as a differential form on $M\setminus S.$
Let $(\Psi:M \to N,S)$ be a cone-smooth map. A \emph{cone-smooth vector field along} $\Psi$ is a smooth section $\xi$ of $\left(\Psi|_{M\setminus S}\right)^*TN$ such that $\left(\pi|_{\widetilde{M}_S^\circ}\right)^*\xi$ extends to a smooth section $\widetilde \xi$ of $(\Psi \circ \pi)^*TN.$ We call $\widetilde \xi$ the \emph{blowup vector field.}
\end{dfn}

\begin{rem}\label{remark:differential of cone-smooth function}
Let $S \subset M$ be a finite subset and let $(f : M \to \R,S)$ be a cone-smooth function. Then, the blowup differential $\widetilde{df}$ is a smooth section of the blowup cotangent bundle $\widetilde{TM}_S^* \to \widetilde{M}_S,$ so $df$ is a cone-smooth 1-form. For $p \in S$ and $\widetilde p \in E_p,$ we use the notation $df_{\widetilde p} = \widetilde{df}|_{\widetilde p}.$
\end{rem}

\begin{rem}\label{remark: blowup differential lifting}
Let $\Psi \in \diff(M,S).$ Then Remark~3.10 of~\cite{cylinders} and the fact that
\[
\pi^*\Psi^* TM = \widetilde \Psi^*\pi^*TM = \widetilde \Psi^*\widetilde{TM}_S
\]
imply that the blowup differential gives an isomorphism of vector bundles
\[
\widetilde{d\Psi} : \widetilde{TM}_S \overset{\sim}{\longrightarrow} \widetilde{\Psi}^*\widetilde{TM}_S.
\]
In particular, cone-smooth diffeomorphisms act by pull-back on cone-smooth differential forms.
\end{rem}

\begin{dfn}\label{definition:immersed cone smooth differential form}
Let $K = [(\Psi: M \to N,S)]$ be a cone-immersed submanifold of type $(M,S).$ A \emph{cone-smooth differential form} on $K$ is an equivalence class $\tau = [((\chi,S),\alpha)]$ where $(\chi,S)$ represents $K$ and $\alpha$ is a cone-smooth differential form on $(M,S).$ Two pairs are equivalent if they belong to the same orbit of the $\diff(M,S)$ action given by Remark~\ref{remark: blowup differential lifting}. We may write $\alpha = \Psi^*\tau.$ Given a smooth form $\eta$ on $N,$ the \emph{restriction} to $K$ is the cone-smooth form given by
\[
\eta|_K : = [((\Psi,S),\Psi^*\eta)].
\]
We say that $\tau$ is \emph{closed} if $\alpha$ is. We say that $\tau$ \emph{vanishes at the cone locus} if $\widetilde{\alpha}$ vanishes on $\partial \widetilde{M}_S.$
\end{dfn}

\begin{dfn}
Let $K= [(\Psi:M \to N,S)]$ be a cone-immersed submanifold, let $p = [((\Psi,S),q)]$ be a cone point, and let $\widetilde{p} = \left[\left(\orientedprojective(d\Psi_q),\widetilde q\right)\right]\in \orientedprojective(TC_pK)$. The \emph{tangent space of $K$ at $\widetilde{p}$} is defined by
\[
T_{\widetilde p} K := \widetilde{d\Psi}_{\widetilde q}\left(\widetilde{TM}_{\widetilde q}\right)\subset T_{p_0}N,
\]
which is independent of the choice of representatives. At a smooth point $p = [((\Psi,S),q)]$ of $K,$ we define the tangent space $T_pK:=d\Psi_q(T_qM)\subset T_{\im(p)}N.$
\end{dfn}

The following are Definition~3.31, Remark~3.32 and Definition~3.33 from~\cite{cylinders}.
\begin{dfn}
		\label{definition: critical cone point}\mbox{}
\begin{enumerate}
\item\label{item:critical cone point intrinsic}
Let $(f : M \to \R,S)$ be a cone-smooth function, and let $p \in S.$ The point $p$ is said to be a \emph{critical point} of $f$ if the cone-derivative $df_p : T_pM \to \R$ vanishes identically.
\item\label{item:critical cone point submanifold}
Let $K = [(\Psi : M \to N,S)]$ be a cone-immersed submanifold, let $h = [((\Psi,S),f)]$ be a cone-smooth function on $K,$ and let $p  = [((\Psi,S),q)]$ be a cone point. The point $p$ is a \emph{critical point} of $h$ if $q$ is a critical point of $f.$
\end{enumerate}
	\end{dfn}

\begin{rem}\label{remark: critical cone point}
Let $S \subset M$ be a finite subset and let $(f : M \to \R,S)$ be a cone-smooth function.
It follows from Lemma~\ref{lemma:blowupdifferential} that in the situation of part~(\ref{item:critical cone point intrinsic}) of the preceding definition, if $p$ is a critical point of $f,$ then $df_{\widetilde p} = 0$ for all $\widetilde p \in E_p.$ The analogous statement holds in the situation of part~(\ref{item:critical cone point submanifold}).
\end{rem}

\begin{dfn}\label{dfn:cone-Hessian}
Let $(f : M \to \R,S)$ be a cone-smooth function, and let $p \in S$ be a critical point of $f.$
\begin{enumerate}
\item
The \emph{cone-Hessian} of $f$ at $p$ is the map
\[
\nabla df : T_pM \to T_pM^*,
\]
smooth away from $0$ and homogeneous of degree $1$, defined as follows.
By Remark~\ref{remark: critical cone point}, the blowup differential $\widetilde{df}$ vanishes on the exceptional sphere $E_p \subset \widetilde{M}_S.$ So, the restriction of the second covariant derivative $\nabla \widetilde{df} \in Hom\left(T\widetilde{M}_S, \widetilde{TM}_S^*\right)$ to $E_p$ is independent of the choice of connection. Moreover, $\nabla \widetilde{df}$ vanishes on $TE_p \subset T\widetilde{M}_S|_{E_p}.$ Recall that a vector $0 \neq v \in T_pM$ gives rise to a point $[v] \in \orientedprojective(T_pM) \simeq E_p.$ For $v \in T_pM,$ and $\widetilde v \in T_{[v]} \widetilde{M}_S$ such that $d \pi_{[v]}(\widetilde v) = v,$  we define
\[
\nabla_v df := \nabla_{\widetilde v} \widetilde{df} \in \left(\widetilde{TM}^*_S\right)_{[v]} = T_p^*M.
\]
\item
\label{definition: degenerate critical cone point}
The critical point $p$ is said to be \emph{degenerate} if there exists a tangent vector $0\ne v\in T_pM$ with $\nabla_vdf=0.$
\end{enumerate}
\end{dfn}
The following is Definition~3.3 from~\cite{cylinders}.
\begin{dfn}\label{definition:polar coordinates}
Let $M$ be a smooth manifold, let $p \in M.$ Let $\sigma$ be a smooth section of the quotient map $T_pM\setminus\{0\} \to \orientedprojective(T_pM) \cong S^{n-1}.$ A \emph{polar coordinate map} centered at $p$ associated with $\sigma$ is a smooth map
\[
\kappa : S^{n-1} \times [0,\epsilon) \to M
\]
satisfying the following.
\begin{enumerate}
\item
$\kappa|_{S^{n-1}\times\{0\}}$ is the constant map to $p.$
\item
$\kappa|_{S^{n-1}\times (0,\epsilon)}$ is an open embedding.
\item\label{item:sigma}
For $r \in S^{n-1},$
\[
\pderiv[\kappa]{s}(r,0) = \sigma(r).
\]
\end{enumerate}
An elementary argument shows that for every $\sigma$ there exist many polar coordinate maps. We may sometimes speak of a polar-coordinate map $\kappa$ without mentioning $\sigma$ explicitly. In this case, we refer to the section $\sigma$ determined by condition~(\ref{item:sigma}) as the section associated with $\kappa.$
\end{dfn}
The following is Lemma~3.37 from~\cite{cylinders}.
	\begin{lemma}
		\label{lemma: nice level sets}
		Let $M$ be a smooth manifold of dimension $m+1$ and let $p\in M.$ Let $h:M\to\mathbb{R}$ be cone-smooth at $p$ such that $p$ is a non-degenerate critical point and an extremum point of $h.$ Then there exist a positive $\epsilon$ and a polar coordinate map $\kappa:S^m\times[0,\epsilon)\to M$ centered at $p$ such that
for each $s\in(0,\epsilon)$ the restricted map $\kappa|_{S^m\times\{s\}}$ parameterizes a level set of $h.$
	\end{lemma}

	\subsection{Lagrangians with cone points}
	\label{subsection:Lagrangians with cone points}
	
	Let $(X,\omega)$ be a symplectic manifold and let $L$ be a smooth manifold. A cone-immersed submanifold $\Lambda=[(\Psi:L\to X,S)]$ is said to be Lagrangian if it is Lagrangian in the smooth locus. It follows in this case that, for a cone point $p\in\Lambda$ and $\widetilde{p}\in\orientedprojective\left(TC_p\Lambda\right),$ the tangent space $T_{\widetilde{p}}\Lambda$ is also Lagrangian.

	We wish to study paths of cone-immersed Lagrangians with static cone locus. For a finite subset $C_0\subset X,$ we let $\mathcal{L}(X,L;S,C_0)$ denote the space of oriented cone-immersed Lagrangians in $X$ of type $(L,S)$ with cone locus image equal to $C_0.$ Let $\Theta$ be a manifold with corners. For a family $(\Lambda_t)_{t\in\Theta}$ in $\mathcal{L}(X,L;S,C_0),$ a lifting of $(\Lambda_t)_t$ is a family of cone-immersions, $((\Psi_t:L\to X,S))_{t\in\Theta},$ such that $(\Psi_t,S)$ represents $\Lambda_t$ for $t\in\Theta.$ The family $(\Lambda_t)_{t \in \Theta}$ is \emph{smooth} if it admits a smooth lifting, that is, if the family of maps $\Psi_t\circ\pi$ is smooth, where $\pi:\widetilde{L}_S\to L$ is the blowup projection.

For a path $(\Lambda_t)_t$ in $\mathcal{L}(X,L;S,C_0),$ the time derivative $\deriv{t}\Lambda_t$ is a closed cone-smooth 1-form on $\Lambda_t$ that vanishes at the cone locus. See Section~3.2 of~\cite{cylinders} for details of how the argument in the smooth case generalizes to the cone-smooth case. For a cone-smooth immersed submanifold $\Lambda$ of $X,$ let $C^\infty(\Lambda)$ denote the space of cone-smooth functions on $\Lambda.$ A path $(\Lambda_t)_{t\in[0,1]}$ in $\mathcal{L}(X,L;S,C_0)$ is said to be exact if, for $t\in[0,1],$ we have
	\[
	\deriv{t}\Lambda_t=dh_t
	\]
	for some $h_t\in C^\infty(\Lambda_t).$
	
	Let $(X,\omega,J,\Omega)$ be Calabi-Yau and let $L$ be a smooth manifold. A cone-immersed Lagrangian $\Lambda=[(\Psi:L\to X,S)]$ is said to be positive Lagrangian if it is positive Lagrangian in the smooth locus and, in addition, the tangent space $T_{\widetilde{p}}\Lambda\subset T_{\im(p)}X$ is positive Lagrangian for $p\in\Lambda^C$ and $\widetilde{p}\in\orientedprojective\left(TC_p\Lambda\right).$ Below is the definition of geodesics in its general form, which is taken from Definition~3.40 of~\cite{cylinders}.

	\begin{dfn}
		\label{definition:geodesic}
		Let $(X,\omega,J,\Omega)$ be a Calabi-Yau manifold, let $L$ be a connected smooth manifold, not necessarily closed, and let $S \subset L$ be a finite subset. Let $C_0\subset X$ be finite, let $\mathcal{O}\subset\mathcal{L}(X,L;S,C_0)$ be an exact isotopy class of cone-immersed Lagrangians, and let $(\Lambda_t)_{t\in[0,1]}$ be a path in $\mathcal{O}.$ A lifting $(\Psi_t:L\to X)_t$ of $(\Lambda_t)_t$ is said to be \emph{horizontal} if it satisfies
		\[
		i_{\deriv{t}\Psi_t}\real\Omega=0,\quad t\in[0,1],
		\]
in the smooth locus. The path $(\Lambda_t)_t$ is a \emph{geodesic} if it admits a horizontal lifting $((\Psi_t,S))_{t\in[0,1]}$ and a family of functions $h_t\in C^\infty(\Lambda_t)$ satisfying
		\[
		\deriv{t}\Lambda_t=dh_t,\quad\deriv{t}(h_t\circ\Psi_t)=0.
		\]
		We call the family $(h_t)_{t\in[0,1]}$ the \emph{Hamiltonian} or the \emph{derivative} of the geodesic. We also call the time independent function $h = h_t \circ \Psi_t : L \to \R$ the \emph{Hamiltonian} with respect to the horizontal lifting $(\Psi_t)_t.$ Observe that $h_t = [(\Psi_t,h)].$  If $L$ is not compact, the Hamiltonian is only well-defined up to a time-independent constant. If $C_0$ is empty, we say $(\Lambda_t)_t$ is a \emph{smooth geodesic} or \emph{geodesic of smooth Lagrangians}.

Recalling that $\im(\Lambda_t)$ denotes the image of $\Lambda_t$ in $X,$ we define the \emph{critical locus} of the geodesic $(\Lambda_t)_t$ by
\[
\crit((\Lambda_t)_t) = \bigcap_{t\in[0,1]}\im(\Lambda_t) \subset X.
\]
Observe that $C_0 \subset \crit((\Lambda_t)_t)$ but the opposite inclusion need not hold.
We say the cone-smooth geodesic $(\Lambda_t)_t$ is of \emph{class} $C^{k,\alpha}$ if it admits a lifting $(\Psi_t : L \to X)_t$ that belongs to the space $C^{k,\alpha}(L \times [0,1],X).$
	\end{dfn}

\begin{ntn}\label{ntn:critical point}
Suppose $(\Lambda_t)_{t \in [0,1]}$ is a geodesic with Hamiltonian $(h_t)_t$ and $\Lambda_0,\Lambda_1,$ are embedded. Let $(\Psi_t : L \to X,S)_t$ be a horizontal lifting of $(\Lambda_t)_t.$ Then, for each $q \in \crit((\Lambda_t)_t)$ there exists a unique $p \in L$ such that $\Psi_t(p) = q$ for $t \in [0,1].$ Thus, we let $\hat q_t$ denote the point of $\Lambda_t$ given by
\[
\hat q_t = [(\Psi_t,p)].
\]
\end{ntn}
\begin{rem}\label{rem:critical point}
In the setting of Notation~\ref{ntn:critical point}, it follows that $\hat q_t$ is a critical point of the Hamiltonian function $h_t.$
\end{rem}

The following is Lemma~5.11 of~\cite{cylinders}.
	\begin{lemma}
		\label{lemma: critical points are non-degenerate}
		Suppose $\Lambda_0$ and $\Lambda_1$ intersect transversally at $q \in C_0.$  Then, $\hat q_t$ is a non-degenerate critical point of $h_t,\, t \in [0,1].$
	\end{lemma}

	\subsection{Special Lagrangian cylinders and geodesics}
	\label{subsection:special Lagrangian cylinders and their relation to geodesics}

	In this section we recall results from~\cite{cylinders} concerning special Lagrangian cylinders, families thereof and their relation to geodesics of positive Lagrangians.
	
	\begin{dfn}
		\label{definition:Lagrangian cylinders}
		Let $(X,\omega,J,\Omega)$ be a Calabi-Yau manifold of complex dimension~$n,$ let $\Lambda_0,\Lambda_1\subset X$ be smoothly embedded positive Lagrangians, and let $N$ be a connected smooth manifold of dimension $n-1.$ Write $L:=N\times[0,1].$ A \emph{Lagrangian cylinder} of type $L$ between $\Lambda_0$ and $\Lambda_1$ is an immersed Lagrangian $Z=[\chi:L\to X]$ satisfying the following conditions.
		\begin{enumerate}
			\item The restricted immersion $\chi|_{N\times\{i\}}$ is an embedding of $N\times\{i\}$ into $\Lambda_i$ for $i = 0,1.$
			\item For $i=0,1$ and $p\in N\times\{i\}$ we have $d\chi_p(T_pL)\ne T_{\chi(p)}\Lambda_i.$
		\end{enumerate}
		We let $\mathcal{LC}(N;\Lambda_0,\Lambda_1)$ denote the space of Lagrangian cylinders of type $L$ between $\Lambda_0$ and $\Lambda_1$ and $\mathcal{SLC}(N;\Lambda_0,\Lambda_1)\subset\mathcal{LC}(N;\Lambda_0,\Lambda_1)$ the subspace consisting of imaginary special Lagrangian cylinders. We denote by $\mathcal{SLC}(\Lambda_0,\Lambda_1)$ the union of the spaces $\mathcal{SLC}(N;\Lambda_0,\Lambda_1)$ as $N$ varies.
	\end{dfn}

The requirement that $\chi|_{N\times\{i\}}$ is an embedding guarantees that $\chi$ is a free immersion by Lemma~2.5 of~\cite{cylinders}. Thus, $\mathcal{LC}(N;\Lambda_0,\Lambda_1)$ is a Fr\'echet manifold by Theorem~2.12 of~\cite{cylinders}. The requirement that $\chi|_{N\times\{i\}}$ is an embedding is automatically satisfied for the cylinders of $c$ level sets of Definition~\ref{dfn:cylinder} below because the boundaries are level sets of a function. These are the cylinders that give the cylindrical transform.

	When $N$ is closed, we describe the Fr\'echet structure of $\mathcal{LC}(N;\Lambda_0,\Lambda_1)$ via Lemmas~\ref{lemma:Weinstein neighborhood of cylinder} and~\ref{lemma:closed and annihilating is exact} below. Lemma~\ref{lemma:Weinstein neighborhood of cylinder} is an adaptation of the Weinstein neighborhood theorem to Lagrangian cylinders, which is a special case of Lemma~2.10 of~\cite{cylinders}. For a smooth manifold $M$ and a submanifold $Q\subset M,$ we let $\nu_Q\subset T^*M$ denote the conormal bundle of $Q.$

	\begin{lemma}
		\label{lemma:Weinstein neighborhood of cylinder}
		Let $X$ be Calabi-Yau and let $\Lambda_0,\Lambda_1\subset X$ be smoothly embedded positive Lagrangians intersecting transversally. Let $N$ be a closed connected smooth manifold, write $L:=N\times[0,1],$ and let $Z=[\chi:L\to X]\in\mathcal{LC}(N;\Lambda_0,\Lambda_1).$ Identify $L$ with the zero section in $T^*L.$ Then there exist an open neighborhood $L\subset V\subset T^*L$ and a local symplectomorphism $\varphi:V\to X$ with the following properties.
		\begin{enumerate}[label=(\alph*)]
			\item $\varphi|_L=\chi.$
			\item For $i=0,1,$ a point $p\in N\times\{i\}$ and a covector $\xi\in T_p^*L,$ we have
			\[
			\varphi(p,\xi)\in\Lambda_i\Leftrightarrow\xi\in\nu_{N\times\{i\}}.
			\]
		\end{enumerate}
		A pair $(V,\varphi)$ with the above properties is called a \emph{Weinstein neighborhood of $Z$ compatible with $\Lambda_0$ and $\Lambda_1$}.
	\end{lemma}
The following is Lemma 4.3 of~\cite{cylinders}.
	\begin{lemma}
		\label{lemma:closed and annihilating is exact}
		Let $N$ be a smooth manifold and let $\alpha$ be a closed 1-form on the cylinder $N\times[0,1]$ with pull-back to the boundary component $N\times\{0\}$ zero. Then $\alpha$ is exact.
	\end{lemma}

	We continue with the above notation. For a smooth manifold $M$ with boundary, let $\Omega_B^1(M)$ denote the Fr\'echet space consisting of closed 1-forms on $M$ that pull-back to zero on the boundary. Given a compact Lagrangian cylinder $Z\in\mathcal{LC}(N;\Lambda_0,\Lambda_1),$ Theorem~2.12 of~\cite{cylinders} asserts that any Weinstein neighborhood of $Z$ compatible with $\Lambda_0$ and $\Lambda_1$ identifies an open neighborhood $Z\in\mathcal{U}\subset\mathcal{LC}(N;\Lambda_0,\Lambda_1)$ with an open neighborhood $0\in\mathcal{V}\subset\Omega_B^1(Z).$ The space $\mathcal{LC}(N;\Lambda_0,\Lambda_1)$ is thus a smooth Fr\'echet manifold. For a cylinder $Z$ as above, let $C_i,\;i=0,1,$ denote the boundary components corresponding to $N \times \{i\}.$ Let $\cob{\infty}(Z)$ denote the space of smooth functions on $Z$ which vanish on $C_0$ and are constant on $C_1.$ By Lemma~\ref{lemma:closed and annihilating is exact}, an open neighborhood of $Z$ in $\mathcal{LC}(N;\Lambda_0,\Lambda_1)$ is diffeomorphic to a neighborhood of 0 in $\cob{\infty}(Z).$
The following is part of Lemma 4.4 of~\cite{cylinders}. We give the proof for the reader's convenience.
\begin{lemma}\label{lemma:TLC}
There is a canonical isomorphism
\begin{equation*}
T_Z\mathcal{LC}(N;\Lambda_0,\Lambda_1)\cong\cob{\infty}(Z)
\end{equation*}
\end{lemma}
\begin{proof}
Let $(Z_t)_{t\in(-\epsilon,\epsilon)}$ be a smooth path in $\mathcal{LC}(N;\Lambda_0,\Lambda_1)$ with $Z_0 = Z.$ Let $\Psi_t: N \times [0,1]\to X$ for $t\in(-\epsilon,\epsilon)$ be a smooth family of immersions such that $\Psi_t$ represents $Z_t.$ Then, recalling Definition~\ref{definition: immersed submanifold} concerning differential forms on immersed submanifolds, we identify the tangent vector $\left.\deriv{t}Z_t\right|_{t = 0}$ with the 1-form
$
\left[\left(\Psi_0,i_{\left.\deriv{t}\Psi_t\right|_{t = 0}}\omega\right)\right]
$
on $Z.$
Since $\Psi_t(N \times \{i\}) \subset \Lambda_i$ for $i =0,1,$ it follows that $\deriv{t}\Psi_t|_{N \times \{i\}}$ is tangent to $\Lambda_i.$ Since $\Lambda_i$ is Lagrangian and $\Psi_t(N \times \{i\}) \subset \Lambda_i$, it follows that
$
\left[\left(\Psi_0,i_{\left.\deriv{t}\Psi_t\right|_{t = 0}}\omega\right)\right] \in \Omega_B^1(Z).
$
Finally, by Lemma~\ref{lemma:closed and annihilating is exact}, we have $\Omega_B^1(Z) \cong \cob{\infty}(Z).$
\end{proof}
	
For an imaginary special Lagrangian cylinder $Z,$ define a differential operator
	\begin{equation}
	\label{equation:Laplacian}
	\Delta_\rho:C^\infty(Z)\to C^\infty(Z),\quad u\mapsto*d(\rho*du),
	\end{equation}
	where $*$ denotes the Hodge star operator of the K\"ahler metric and $\rho$ is the function given by~\eqref{equation:rho}. The following lemma is a restatement of Lemmas~4.5 and~4.6 of~\cite{cylinders}. Let $C^{\infty}(L;\partial L)$ denote the space of smooth functions on $L$ which vanish on the boundary.

	\begin{lemma}
		\label{lemma:Laplacian}
		Let $N$ be a closed connected manifold, let $Z = [f : L \to X]$ be an immersed imaginary special Lagrangian of type $L:=N\times[0,1],$ and let $\Delta_\rho$ be as in~\eqref{equation:Laplacian}.
		\begin{enumerate}[label=(\alph*)]
			\item\label{linearization is Laplacian} Let $f_t:L\to X,\;t\in(-\epsilon,\epsilon),$ be a smooth family of Lagrangian immersions with $f_0 = f.$ Write $v:=\left.\deriv{t}\right|_{t=0}f_t$ and suppose we have $i_v\omega=du$ for some $u\in C^\infty(L).$ Then
			\[
			\left.\deriv{t}\right|_{t=0}\left(*f_t^*\real\Omega\right)=\Delta_\rho u.
			\]
\item \label{elliptic}
The linear map
			\[
			\left.\Delta_\rho\right|_{C^\infty(L;\partial L)}:C^\infty(L;\partial L)\to C^\infty(L)
			\]
			is an isomorphism.
In particular,
the intersection $\ker\Delta_\rho\cap\cob{\infty}(L)$ is 1-dimensional.
		\end{enumerate}
	\end{lemma}

\begin{dfn}\label{dfn:fundamental harmonic}
Let $Z = [f : N \times [0,1] \to X]$ be an immersed imaginary special Lagrangian cylinder. For $i = 0,1,$ let $C_i$ denote the boundary component of $Z$ corresponding to $N \times \{i\}.$
\begin{enumerate}[label = (\alph*)]
\item The \emph{fundamental harmonic} of $Z$ is the unique function $\sigma \in \cob{\infty}(Z)$ such that
    \[
    \Delta_\rho\sigma = 0, \qquad \sigma|_{C_1} \equiv 1.
    \]
\item We say that $Z$ has \emph{regular harmonics} if the fundamental harmonic of $Z$ has no critical points.
\item The immersion $f$ representing $Z$ is said to be \emph{adapted to the harmonics} of~$Z$ if
    \[
    \sigma \circ f(p,t) = t, \qquad (p,t) \in N\times [0,1].
    \]
In this case, it follows immediately that $Z$ has regular harmonics.
\end{enumerate}
\end{dfn}
	
	The following is Proposition~4.7 of~\cite{cylinders}, which is an adaptation of McLean's result on the deformation theory of closed special Lagrangians~\cite{mclean}.
	
	\begin{prop}
		\label{proposition:space of special Lagrangian cylinders}
		Let $\Lambda_0,\Lambda_1\subset X$ be smoothly embedded positive Lagrangians intersecting transversally, and let $N$ be a closed connected smooth manifold. Then the space $\mathcal{SLC}(N;\Lambda_0,\Lambda_1)$ is a smooth manifold of dimension 1. For an imaginary special Lagrangian cylinder $Z\in\mathcal{SLC}(N;\Lambda_0,\Lambda_1)$ we have
		\[
		T_Z\mathcal{SLC}(N;\Lambda_0,\Lambda_1)=\ker\Delta_\rho\cap\cob{\infty}(Z).
		\]
	\end{prop}

	By Proposition~\ref{proposition:space of special Lagrangian cylinders}, imaginary special Lagrangian cylinders naturally appear in 1-parameter families. The following proposition shows that $1$-parameter families of imaginary special Lagrangian cylinders arise naturally from geodesics. The proposition is an immediate corollary of Lemmas 5.1 and 5.4 of~\cite{cylinders}.
	\begin{prop}
		\label{proposition:relevant results from previous paper}
		Let $(X,\omega,J,\Omega)$ be a Calabi-Yau manifold.
		\begin{enumerate}[label=(\alph*)]
			\item \label{item:cylinder} Let $(\Lambda_t)_{t\in[0,1]}$ be a geodesic of positive Lagrangians in $X$ with Hamiltonian~$(h_t)_t.$ Let $(\Psi_t:L\to X)_{t\in[0,1]}$ be a horizontal lifting of $(\Lambda_t)_t,$ and let $h:L\to\mathbb{R}$ denote the Hamiltonian of the geodesic with respect to $(\Psi_t)_t,$ that is, $h:=h_t\circ\Psi_t.$ For $c\in\mathbb{R},$ let
\[
L_c := (h^{-1}(c)\setminus \crit(h))\times[0,1].
\]
Then, the mapping
			\[
			\Phi_c:L_c\to X,\qquad(p,t)\mapsto\Psi_t(p),
			\]
			is an imaginary special Lagrangian immersion.
\item\label{item:harmonic} The mapping $\Phi_c$ is adapted to the harmonics of the cylinder $Z = [\Phi_c]$. In particular, $Z$ has regular harmonics.
		\end{enumerate}
	\end{prop}
The following is an amalgamation of Definitions~5.2 and~1.4 of~\cite{cylinders}.
\begin{dfn}\label{dfn:cylinder}
Let $(X,\omega,J,\Omega)$ be a Calabi-Yau manifold and let $(\Lambda_t)_{t \in [0,1]}$ be a geodesic of positive Lagrangians in $X$ with derivative $(h_t)_t.$ Let $(\Psi_t:L\to X)_{t\in[0,1]}$ be a horizontal lifting of $(\Lambda_t)_t,$ and let $h:L\to\mathbb{R}$ denote the Hamiltonian of the geodesic with respect to $(\Psi_t)_t.$
\begin{enumerate}
\item\label{item: cylinder of level sets}
For $c \in \R,$ the \emph{cylinder of $c$ level sets} associated to the geodesic is the immersed imaginary special Lagrangian cylinder represented by the map $\Phi_c$ of Proposition~\ref{proposition:relevant results from previous paper}~\ref{item:cylinder}.
\item
\label{item: cylindrical transform}
		The \emph{cylindrical transform} of the geodesic $(\Lambda_t)_{t \in [0,1]}$ is the subset of the space of imaginary special Lagrangian cylinders $\mathcal{SLC}(\Lambda_0,\Lambda_1)$ parameterized by the family of imaginary special Lagrangian immersions $\Phi_c : L_c \to X$ from Proposition~\ref{proposition:relevant results from previous paper}~\ref{item:cylinder} for $c \in \R$ such that $L_c \neq \emptyset$.
\end{enumerate}
\end{dfn}

Let $\Lambda_0,\Lambda_1 \subset X$ be Lagrangian submanifolds and let $N$ be a connected closed manifold of dimension $n-1.$ Let $(Z_s)_{s \in [s_0,s_1]}$ be a family of Lagrangian cylinders in $\mathcal{LC}(N;\Lambda_0,\Lambda_1).$ Recalling Lemma~\ref{lemma:TLC}, for $s\in[s_0,s_1],$ write
\[
	h_s:=\deriv{s}Z_s\in\cob{\infty}(Z_s).
\]
	For $i = 0,1,$ let $C_{i,s}$ denote the boundary component of the cylinder $Z_s$ corresponding to $N \times \{i\}$ and let $A_s \in \R$ be the unique constant such that
	\[
	h_s|_{C_{1,s}} \equiv A_s.
	\]
The following is Definition~4.11 of~\cite{cylinders}.
\begin{dfn}\label{dfn: flux}
The \emph{relative Lagrangian flux} of the family $(Z_s)_{s \in [s_0,s_1]}$ is given by
\[
\operatorname{RelFlux}\left((Z_s)_{s \in [s_0,s_1]}\right) := -\int_{s_0}^{s_1}A_sds.
\]
More generally, if $I \subset \R$ is an interval, possibly open or half open, with endpoints $a < b,$ and $(Z_s)_{s \in I}$ is a path in $\mathcal{LC}(N;\Lambda_0,\Lambda_1),$ we write
\[
\operatorname{RelFlux}\left((Z_s)_{s\in I}\right) := \lim_{s_0 \to a}\lim_{s_1 \to b}\operatorname{RelFlux}\left((Z_s)_{s\in[s_0,s_1]}\right)
\]
whenever the limit exists.
\end{dfn}
\begin{rem}\label{rem: flux}\mbox{}
\begin{enumerate}
\item \label{it: flux symplectic invariant}
It is clear from the definition that $\operatorname{RelFlux}((Z_s)_{s \in I})$ is a symplectic invariant.
\item \label{it: flux homotopy invariant}
The relative Lagrangian flux $\operatorname{RelFlux}((Z_s)_{s \in [s_0,s_1]})$ depends only on the homotopy class of the path $(Z_s)_s$ relative to its endpoints. See Remark~4.12 of~\cite{cylinders}.
\end{enumerate}
\end{rem}
The following is a reformulation of Lemma~5.6 of~\cite{cylinders}. It asserts that the Hamiltonian of a geodesic can be recovered from the relative Lagrangian flux of the associated cylinders of c-level sets.
\begin{lemma}
		\label{lemma: flux as Hamiltonian}
		Let $(\Lambda_t)_t$ be a geodesic with Hamiltonian $(h_t)_t$ and assume the functions $h_t$ are proper. For $c$ in the image of $h_t,$ let $Z_c$ denote the associated cylinder of $c$ level sets. Since the functions $h_t$ are proper, the cylinder $Z_c$ is compact when $c$ is a regular value of $(h_t)_t.$ Let $c_0<c_1\in\mathbb{R}$ be such that the interval $(c_0,c_1)$ consists of regular values of $(h_t)_t.$
For $b_0<b_1\in(c_0,c_1)$ we have
			\[
			\operatorname{RelFlux}\left((Z_c)_{c\in[b_0,b_1]}\right)=b_1-b_0.
			\]
\end{lemma}

Recall that the \emph{Euler vector field} on a real vector space is the radial vector field that integrates to rescaling by $e^t.$  The following is Lemma~4.15 from~\cite{cylinders}.
	\begin{lemma}
		\label{lemma: cylinder with regular harmonics in Euclidean space}
		Equip $\mathbb{C}^n$ with the standard Calabi-Yau structure, let $\Lambda_0,\Lambda_1\subset\mathbb{C}^n$ be positive Lagrangian linear subspaces, and let $Z\in\mathcal{SLC}\left(S^{n-1};\Lambda_0,\Lambda_1\right).$ Then $Z$ has regular harmonics if and only if $Z$ is nowhere tangent to the Euler vector field.
	\end{lemma}

The following simplified notion of ends for $1$-dimensional manifolds is given in Section 4.3 of~\cite{cylinders}. Let $C$ be a connected non-compact $1$-dimensional manifold. That is, $C$ is a curve diffeomorphic to the real line. A \emph{ray} in $C$ is a connected open proper subset $U\subsetneq C$ with non-compact closure. Two rays, $U,V\subset C$ are said to be equivalent if $U\subset V$ or $V\subset U.$ Finally, an \emph{end} is an equivalence class of rays. Every curve $C$ as above has exactly two ends.
The following is taken from Definition~4.17 of~\cite{cylinders}.
	\begin{dfn}
		\label{definition: interior regularity}
		$\;$
		\begin{enumerate}
			\item\label{interior regularity first part} Let $U\subset\mathcal{SLC}\left(N;\Lambda_0,\Lambda_1\right)$ be open and connected. An \emph{interior-regular parameterization} of $U$ is a smooth immersion $\Phi:N\times[0,1]\times(a,b)\to X$ satisfying the following conditions:
			\begin{enumerate}[ref=(\alph*)]
				\item\label{item: boundary embedding} The restriction of $\Phi$ to a boundary component $\Phi|_{N\times\{i\}\times(a,b)}$ is an embedding for $i = 0,1$.
				\item\label{item: rep} For $s\in(a,b),$ the restricted immersion $\Phi_s:=\Phi|_{N\times[0,1]\times\{s\}}$ represents an element of $U$.
				\item\label{item: diffeo s} The map $\chi:(a,b)\to U,\quad s\mapsto[\Phi_s],$ is a diffeomorphism.
			\end{enumerate}
			The subset $U$ is said to be \emph{interior-regular} if it admits an interior-regular parameterization.
			\item\label{regular convergence to intersection point} Let $\mathcal{Z}\subset\mathcal{SLC}\left(S^{n-1};\Lambda_0,\Lambda_1\right)$ be a connected component, let $E$ be an end of $\mathcal{Z},$ and let $q\in\Lambda_0\cap\Lambda_1.$ A \emph{regular parameterization} of $E$ about $q$ is a smooth map $\Phi:S^{n-1}\times[0,1]\times[0,\epsilon)\to X$ satisfying the following conditions.
			\begin{enumerate}[ref=(\alph*)]
				\item\label{critical point} For $(p,t)\in S^{n-1}\times[0,1]$ we have $\Phi(p,t,0)=q.$
				\item\label{item:interior regular} The restricted map $\Phi|_{S^{n-1}\times[0,1]\times(0,\epsilon)}$ is an interior-regular parameterization of $U,$ for some ray $U\subset\mathcal{Z}$ representing $E.$
				\item \label{item:derivative Phi immersion} The derivative
				\[
				\left.\pderiv{s}\right|_{s=0}\Phi(\cdot,\cdot,s):S^{n-1}\times[0,1]\to T_qX
				\]
				is an immersion and the restriction $\left.\left.\pderiv{s}\right|_{s=0}\Phi(\cdot,\cdot,s)\right|_{S^{n-1}\times\{i\}}$ is an embedding for $i = 0,1.$
				\item \label{item:nowhere tangent Euler} The Euler vector field on $T_qX$ is nowhere tangent to the immersion $\left.\pderiv{s}\right|_{s=0}\Phi(\cdot,\cdot,s).$
			\end{enumerate}
In this case, we also say that $\Phi$ is a regular parameterization of $U$ about~$q.$ We say the end $E$ or the ray $U$ \emph{converges regularly} to the intersection point~$q$ if it admits a regular parameterization about $q.$ We may use a half-open interval with arbitrary endpoints, open either from below or above, in place of the half-open interval $[0,\epsilon).$
		\end{enumerate}
	\end{dfn}
The following is Corollary~4.20 from~\cite{cylinders}.
\begin{lemma}
\label{rem: easy regularity}
If $\Phi:S^{n-1}\times[0,1]\times[0,\epsilon)\to X$ satisfies conditions~\eqref{regular convergence to intersection point}\ref{critical point}, \eqref{regular convergence to intersection point}\ref{item:derivative Phi immersion}, \eqref{regular convergence to intersection point}\ref{item:nowhere tangent Euler} and~\eqref{interior regularity first part}\ref{item: rep} of Definition~\ref{definition: interior regularity}, then possibly after diminishing $\epsilon,$ it also satisfies condition~\eqref{regular convergence to intersection point}\ref{item:interior regular} and thus it is a regular parameterization.
\end{lemma}

The following are Definition~5.7 and Lemma~5.9 from~\cite{cylinders}.
	\begin{dfn}
		\label{definition: swept submanifold}
		Let $\Lambda_0,\Lambda_1\subset X$ be smooth Lagrangians and let $N$ be a closed connected smooth manifold of dimension $n-1.$ Let $U\subset\mathcal{SLC}(N;\Lambda_0,\Lambda_1)$ be open, connected and interior-regular. Let $\Phi:N\times[0,1]\times(0,1)\to X$ be an interior regular parameterization of $U.$ For $i=0,1,$ the submanifold of $\Lambda_i$ \emph{swept by $U$} is the image of the embedding $\Phi|_{N\times\{i\}\times(0,1)}.$ This is independent of $\Phi.$ Similarly, suppose $U\subset\mathcal{SLC}\left(S^{n-1};\Lambda_0,\Lambda_1\right)$ is a ray and $\Phi : S^{n-1}\times [0,1]\times [0,\epsilon) \to X$ is a regular parameterization of $U$ about an intersection point $q \in \Lambda_0 \cap \Lambda_1.$ For $i = 0,1,$ the \emph{unpunctured} submanifold of $\Lambda_i$ swept by $U$ is the image of the restricted map~$\Phi|_{S^{n-1}\times\{i\}\times[0,\epsilon)}.$
	\end{dfn}
	\begin{lemma}
		\label{lemma: geodesic of small Lagrangians}
		Let $\Lambda_0,\Lambda_1\subset X$ be smoothly embedded positive Lagrangians intersecting transversally at a point $q.$ Suppose there exists a connected component $\mathcal{Z}\subset\mathcal{SLC}\left(S^{n-1};\Lambda_0,\Lambda_1\right)$ with an end $E$ converging regularly to $q.$ Let $U \subset \mathcal{Z}$ be a ray representing $E$ admitting a regular parameterization about $q.$ For $i = 0,1,$ let $\Lambda_i^U$ denote the unpunctured submanifold of $\Lambda_i$ swept by $U$. Then, there exists a geodesic of positive Lagrangians between $\Lambda_0^U$ and $\Lambda_1^U$ with cylindrical transform~$U.$ This geodesic is unique up to reparameterization and has critical locus $\{q\}.$
	\end{lemma}
Recall Notation~\ref{ntn:critical point}.
\begin{set}\label{set:rocking}
Let $(\Lambda_t)_{t\in[0,1]}$ be a geodesic of positive Lagrangians in $X$ with Hamiltonian $(h_t)_t.$ Suppose the endpoints $\Lambda_0$ and $\Lambda_1$ are smoothly embedded and let $q\in\crit((\Lambda_t)_t)$ be a transverse intersection point of $\Lambda_0$ and $\Lambda_1.$ Moreover, for $t \in [0,1]$ assume $\hat q_t$ is an absolute minimum or maximum of $h_t.$ Let $(\Psi_t : L \to X,S)_t$ be a horizontal lifting. Let $h = h_t \circ \Psi_t$ denote the Hamiltonian with respect to $(\Psi_t).$ Let $p \in L$ be such that $\Psi_t(p) = q$ for $t \in [0,1].$ By Lemma~\ref{lemma: critical points are non-degenerate}, $p$ is a non-degenerate critical point of $h.$ By Lemma~\ref{lemma: nice level sets},
choose a positive $\epsilon$ and a polar coordinate map $\kappa:S^{n-1}\times[0,\epsilon)\to L$ centered at $p$ such that
for each $s\in(0,\epsilon)$ the restricted map $\kappa|_{S^{n-1}\times\{s\}}$ parameterizes a level set of $h.$
\end{set}
The following is Lemma~5.12 from~\cite{cylinders}.
\begin{lemma}\label{lemma:rocking lemma}
In Setting~\ref{set:rocking}, let $\mathcal{Z}$ denote the cylindrical transform of $(\Lambda_t)_t.$ Then, one end of $\mathcal{Z}$ converges regularly to $q.$ In fact, the map
		\[
		\Phi:S^{n-1}\times[0,1]\times[0,\epsilon)\to X,\qquad(c,t,s)\mapsto\Psi_t(\kappa(c,s))
		\]
is a regular parameterization of $\mathcal{Z}$ about $q.$
\end{lemma}

The following are Definitions~5.14 and~6.6 from~\cite{cylinders}.
	
\begin{dfn}
	\label{definition: regularity}
Let $\Lambda_0,\Lambda_1 \subset X$ be smooth positive Lagrangian submanifolds intersecting at two points $q_0$ and $q_1.$ Let $\mathcal{Z}\subset\mathcal{SLC}\left(S^{n-1};\Lambda_0,\Lambda_1\right)$ be a connected component. A \emph{regular parameterization} of $\mathcal{Z}$ is a smooth map $\Phi:S^{n-1}\times[0,1]\times[0,1]\to X$ satisfying the following conditions:
\begin{enumerate}
\item The restricted map $\Phi|_{S^{n-1}\times[0,1]\times(0,1)}$ is an interior-regular parameterization of $Z.$
\item The restricted maps $\Phi|_{S^{n-1}\times[0,1]\times[0,1/2)}$ and $\Phi|_{S^{n-1}\times[0,1]\times(1/2,1]}$ are regular parameterizations of the two ends of $\mathcal{Z}$ about the intersection points $q_0$ and $q_1,$ respectively.
\end{enumerate}
We say $\mathcal{Z}$ is \emph{regular} if it admits a regular parameterization.
\end{dfn}

\begin{dfn}\label{dfn:topology on geodesics}
Let $\mathcal{O}$ be a Hamiltonian isotopy class of positive Lagrangian spheres. For $\Lambda_0,\Lambda_1 \in \mathcal{O},$ we write $\Lambda_0 \pitchfork_2 \Lambda_1$ if $\Lambda_0$ and $\Lambda_1$ intersect transversally at exactly two points.  Let
\[
\mathfrak{Z}_\mathcal{O} := \left\{(\Lambda_0,\Lambda_1,\mathcal{Z})\left |
\begin{matrix}
\Lambda_i \in \mathcal{O},\; i = 0,1,\quad \Lambda_0 \pitchfork_2 \Lambda_1, \; \\
\mathcal{Z} \subset \mathcal{SLC}(\Lambda_0,\Lambda_1) \text{ a regular component}
\end{matrix}
\right.\right\}.
\]
We define the strong and weak $C^{k,\alpha}$ topologies on $\mathfrak{Z}_\mathcal{O}$ as follows.
For
\[
\mathcal{V} \subset C^\infty(S^{n-1}\times [0,1],X), \qquad \mathcal{U} \subset C^\infty(S^{n-1}\times [0,1],TX),
\]
open subsets in the $C^{k,\alpha}$ topology, write
\[
\mathcal{T}_{\mathcal{U},\mathcal{V}} := \left \{(\Lambda_0,\Lambda_1,\mathcal{Z}) \in \mathfrak{Z}_\mathcal{O} \left|
\begin{matrix}
\forall Z \in \mathcal{Z}, \; \exists f : S^{n-1}\times [0,1] \to X \text{ representing } Z \\
\text{such that } f \in \mathcal{V},\\
\forall E \text{ an end of }\mathcal{Z},\; \exists \Phi : [0,\epsilon) \to X \text{ a regular}\\
\text{parameterization of $E$ such that }\left.\pderiv[\Phi]{s}\right|_{s = 0} \in \mathcal{U}
\end{matrix}
\right.\right\}
\]
and
\[
\mathcal{X}_{\mathcal{V}} = \left \{(\Lambda_0,\Lambda_1,\mathcal{Z}) \in \mathfrak{Z}_\mathcal{O} \left|
\begin{matrix}
\exists Z \in \mathcal{Z}, \; \exists f : S^{n-1}\times [0,1] \to X \text{ representing } Z \\
\text{such that } f \in \mathcal{V}
\end{matrix}
\right.\right\}.
\]
Then, a basis for the strong $C^{k,\alpha}$ topology on $\mathfrak{Z}_\mathcal{O}$ is given by sets of the form $\mathcal{T}_{\mathcal{U},\mathcal{V}}$ and a sub-basis for the weak $C^{k,\alpha}$ topology on $\mathfrak{Z}_\mathcal{O}$ is given by sets of the form $\mathcal{X}_\mathcal{V}.$  Let
\[
\mathfrak{G}_\mathcal{O} : = \{(\Lambda_t)_{t \in [0,1]} | (\Lambda_t)_{t \in [0,1]} \text{ is a geodesic with } \Lambda_0,\Lambda_1 \in \mathcal{O}, \quad \Lambda_0 \pitchfork_2 \Lambda_1\}
\]
denote the space of geodesics with endpoints in $\mathcal{O}$ intersecting transversally at two points.
By Theorem~1.5 of~\cite{cylinders}, the cylindrical transform gives a bijection
\[
\mathfrak{G}_\mathcal{O} \simeq \mathfrak{Z}_\mathcal{O}.
\]
So, the strong and weak $C^{k,\alpha}$ topologies on $\mathfrak{Z}_\mathcal{O}$ give rise to topologies on $\mathfrak{G}_\mathcal{O},$ which we also call the strong and weak $C^{k,\alpha}$ topologies respectively.
\end{dfn}

	\section{The positive Lagrangian Grassmannian}
	\label{section:the positive Lagrangian Grassmannian}
\subsection{The linear positive Lagrangian connection}	
	Let $\left(\mathbb{C}^n,\omega,J,\Omega\right)$ denote the standard Calabi-Yau structure on $\mathbb{C}^n.$ Let $\laggrass(n)$ denote the Lagrangian Grassmannian,
	\[
	\laggrass(n):=\left\{\Lambda\in\Gr_\mathbb{R}\left(n,\mathbb{C}^n\right)\;|\;\omega|_\Lambda=0\right\},
	\]
	where $\Gr_\mathbb{R}(n,\mathbb{C}^n)$ denotes the space of real $n$-dimensional linear subspaces of $\mathbb{C}^n.$ Then $\laggrass(n)$ is a smooth manifold of dimension $\frac{n(n+1)}{2}.$ The following well-known lemma provides a convenient description of the tangent bundle $T\laggrass(n).$
	
	\begin{lemma}
		\label{lemma:tangent space of Lagrangian Grassmannian}
		For $\Lambda\in\laggrass(n)$ there is a canonical isomorphism $T_\Lambda\laggrass(n)\cong Q(\Lambda),$ the space of quadratic forms on $\Lambda.$
	\end{lemma}
	
	\begin{proof}
		As always when working with spaces of Lagrangian submanifolds, the desired isomorphism is obtained via contraction with $\omega.$ Let $(\Lambda_t)_{t\in(-\epsilon,\epsilon)}$ be a smooth path in $\laggrass(n)$ and let $\Psi_t:\mathbb{R}^n\to\mathbb{C}^n,\;t\in(-\epsilon,\epsilon),$ be a smooth linear lifting. As a 1-form, the time-derivative of $(\Lambda_t)_t$ is given by
		\[
		\sigma=\left.\deriv{t}\right|_{t=0}\Lambda_t:=(\Psi_0)_*\left(i_{\left.\deriv{t}\Psi_t\right|_{t = 0}}\omega\right)\in\Omega^1(\Lambda_0).
		\]
		As the path $(\Lambda_t)_t$ is Lagrangian, $\sigma$ is closed and thus equal to the derivative of a unique function, $h\in C^\infty(\Lambda_0),$ satisfying $h(0)=0.$ As the lifting $(\Psi_t)_t$ is linear, the linear functional $\sigma_x: T_x \Lambda_t \to \R$ depends linearly on $x$ under the canonical isomorphism $T_x \Lambda_t \simeq \Lambda_t.$ It follows that $h$ is a quadratic form.
	\end{proof}

	As every Lagrangian $\Lambda\in\laggrass(n)$ is in particular totally real and $\Omega$ is of type $(n,0),$ we have $\Omega|_\Lambda\ne0$ (see~\cite{harvey-lawson}). Let $\phase:\laggrass(n)\to\mathbb{R}/\pi\mathbb{Z}$ denote the Lagrangian phase,
	\[
	\phase(\Lambda):=\arg\left(\Omega|_\Lambda\right),\quad\Lambda\in\laggrass(n).
	\]
The phase is only well-defined in $\mathbb{R}/\pi\mathbb{Z},$ rather than $\mathbb{R}/2\pi\mathbb{Z},$ as the elements of $\laggrass(n)$ are not oriented. We define the \emph{positive Lagrangian Grassmannian} to be the set of positive Lagrangian linear subspaces,
	\[
	\laggrass^+(n):=\left\{\Lambda\in\laggrass(n)\;\left|\;\phase(\Lambda)\not \equiv \frac{\pi}{2} \pmod{\pi \mathbb{Z}}\right.\right\}.
	\]
	We equip each $\Lambda\in\laggrass^+(n)$ with the orientation making $\real\Omega|_\Lambda$ a positive volume form. So, the phase lifts to a well-defined real-valued function on $\laggrass^+(n),$ which by abuse of notation we continue to denote by
	\[
	\phase:\laggrass^+(n)\to\left(-\frac{\pi}{2},\frac{\pi}{2}\right).
	\]

	As shown in~\cite{solomon}, every path of \emph{closed} positive Lagrangians in a Calabi-Yau admits horizontal liftings. We show the same holds true for paths in $\laggrass^+(n).$

	\begin{lemma}
		\label{lemma:there is always a horizontal lifting}
		Let $(\Lambda_t)_{t\in[0,1]}$ be a smooth path in $\laggrass^+(n)$ and let $\Psi_0:\mathbb{R}^n\to \Lambda_0$ be an isomorphism of real vector spaces. Then $\Psi_0$ extends uniquely to a horizontal lifting, $\Psi_t:\mathbb{R}^n\to\mathbb{C}^n,\;t\in[0,1],$ of the path $(\Lambda_t)_t.$ Moreover, the unique lifting $(\Psi_t)_t$ is linear.
	\end{lemma}

	\begin{proof}
		The argument is similar to the one presented in~\cite{solomon}. Let $\left(\widetilde{\Psi}_t\right)_{t\in[0,1]}$ be a linear lifting of $(\Lambda_t)_t$ with $\widetilde{\Psi}_0=\Psi_0.$ For $t\in[0,1],$ let $w_t$ be the vector field on $\mathbb{R}^n$ given by
		\[
		i_{w_t}\left(\widetilde{\Psi}_t\right)^*\real\Omega=-i_{\deriv{t}\widetilde{\Psi}_t}\real\Omega.
		\]
		Then $w_t$ is a linear time-dependent vector field and hence integrable. Its flow, denoted by $\varphi_t,$ is also linear. The desired lifting is given by
		\[
		\Psi_t:=\widetilde{\Psi}_t\circ\varphi_t.
		\]
	\end{proof}

	\begin{rem}
		\label{remark:every horizontal lifting is assumed linear}
		By virtue of Lemma~\ref{lemma:there is always a horizontal lifting}, all horizontal liftings of paths in $\laggrass^+(n)$ are henceforth assumed to be linear.
	\end{rem}

	\begin{lemma}
		\label{lemma:covariant derivative well-defined}
		Let $(\Lambda_t)_{t\in[0,1]}$ be a smooth path in $\laggrass^+(n)$ and let $(\Psi_t)_{t\in[0,1]}$ be a linear lifting. Let $(h_t)_{t\in[0,1]}$ be a vector field along the path $(\Lambda_t)_t.$ That is, $h_t$ is a quadratic form on $\Lambda_t$ for $t\in[0,1].$ Then for $t_0\in[0,1],$ the function
		\[
\left.\deriv{t}\right|_{t=t_0}(h_t\circ\Psi_t) : \R^n \to \R
		\]
		is a quadratic form.
	\end{lemma}
	
	\begin{proof}
		For $t\in[0,1],$ the composition $h_t\circ\Psi_t$ is a quadratic form on $\mathbb{R}^n.$ The space $Q\left(\mathbb{R}^n\right)\subset\C^\infty\left(\mathbb{R}^n\right)$ is a closed linear subspace. The lemma follows.
	\end{proof}
	
	By Lemmas~\ref{lemma:there is always a horizontal lifting} and~\ref{lemma:covariant derivative well-defined}, we can define the positive Lagrangian connection on $\laggrass^+(n)$ just as in Definition~\ref{definition:liftings and horizontal liftings}. On the other hand, since $\mathbb{R}^n$ is not compact, the Riemannian metric~\eqref{equation:Jake's Riemannian metric} makes no sense in $\laggrass^+(n).$ In fact, we have the following lemma.
\begin{lemma}
The positive Lagrangian connection on $\laggrass^+(n)$ is not a metric connection.
\end{lemma}
\begin{proof}
We show that the curvature of the connection can have a non-zero real eigenvalue and thus cannot be anti-self-adjoint with respect to any metric. We use the formula for the curvature of the positive Lagrangian connection derived in Theorem~6.1 of~\cite{solomon-curv} for general Lagrangian submanifolds. Although~\cite{solomon-curv} considers compactly supported Hamiltonians, the proof of Theorem 6.1 goes through without change for non-compactly supported Hamiltonians as well and thus it applies to the Lagrangian Grassmannian. Let $\Lambda \in \laggrass^+(n)$ and let $\Psi : \R^n \to \Lambda$ be a linear parameterization. By Lemma~\ref{lemma:tangent space of Lagrangian Grassmannian} the tangent space $T_\Lambda\laggrass^+(n)$ is canonically identified with the space of quadratic forms on $\Lambda.$ Composing such a quadratic form with $\Psi$ gives a quadratic form on $\R^n,$ which we represent by a symmetric $n \times n$ matrix. Let $A,B,C$ be symmetric $n \times n$ matrices. Abbreviate $\{A,B\} = AB + BA$ and $[A,B] = AB - BA.$ The formula of Theorem~6.1 in~\cite{solomon-curv} for the Riemannian curvature of the positive Lagrangian connection simplifies to
\[
R(A,B)C = 4 \sec^2(\phase(\Lambda)) \left(\{(\tr A)B - (\tr B)A,C\} + [[A,B],C] \right).
\]
Consider the matrices
\[
A = \frac{1}{n}\id,
\qquad
B = \begin{pmatrix}
0 & 1 & 1 &\cdots & 1 \\
1 & 0 & 1 & \cdots & 1 \\
\vdots &\ddots &\ddots &\ddots & \vdots \\
1 & \cdots & 1 & 0 & 1 \\
1 & \cdots & 1 & 1 & 0
\end{pmatrix},
\qquad
C =
\begin{pmatrix}
1 & \cdots & 1 \\
\vdots & & \vdots \\
1 & \cdots & 1
\end{pmatrix}.
\]
Then, $(\tr A)B - (\tr B)A = B$ and $[A,B] = 0,$ so
\[
R(A,B)C = 4 \sec^2(\phase(\Lambda))\{B,C\} = 4(n-1) \sec^2(\phase(\Lambda)) C.
\]
So, $R(A,B)$ has a non-zero real eigenvalue and thus cannot be anti-self-adjoint with respect to any inner product. On the other hand, if the positive Lagrangian connection were a metric connection, $R(A,B)$ would be anti-self-adjoint with respect to the metric.
\end{proof}
In what follows, all covariant derivatives, geodesics, and exponential maps in $\laggrass^+(n)$ are taken with respect to the positive Lagrangian connection.

\subsection{Linear geodesics}	
	
	In this section we study geodesics in $\laggrass^+(n)$ and prove Theorem~\ref{theorem:linear geodesic}. We equip $\mathbb{R}^n$ with the standard inner product and $\mathbb{C}^n$ with the standard \emph{real} inner product. Every element in $\laggrass^+(n)$ thus has an induced inner product.  We start by stating the following well-known observation.
	
	\begin{lemma}
		\label{lemma:spectral decomposition}
		Let $\Lambda_0,\Lambda_1\in\laggrass(n).$ Then, for some $k\le n,$ there exist distinct elements $\alpha_1,\ldots,\alpha_k\in\mathbb{R}/\pi\mathbb{Z}$ and orthogonal decompositions,
		\[
		\Lambda_i=\bigoplus_{j=1}^kV_{i,j},\quad i=0,1,
		\]
		satisfying
		\[
		e^{\alpha_j\ii}V_{0,j}=V_{1,j},\quad j=1,\ldots,k.
		\]
		The arguments $\alpha_j,\;j=1,\ldots,k,$ and decompositions are unique up to order.
	\end{lemma}

	Proposition~\ref{proposition:geodesic preserves orthogonal frame} below is key in the understanding of geodesics in $\laggrass^+(n).$ We use the following notion.
	
	\begin{dfn}
		\label{definition:compatible lifting}
		Let $(\Lambda_t)_{t\in[0,1]}$ be a geodesic in $\laggrass^+(n)$ with derivative $h_t\in Q(\Lambda_t),\;t\in[0,1].$ A horizontal lifting $(\Psi_t)_{t\in[0,1]}$ of $(\Lambda_t)_t$ is said to be \emph{compatible with $h_0$} if $\Psi_0:\mathbb{R}^n\to\Lambda_0$ is an inner product preserving isomorphism such that for some $a_1,\ldots,a_n\in\mathbb{R},$ we have
		\begin{equation}
		\label{equation:h does not change}
		h_t\circ\Psi_t(x)=\sum_ja_jx_j^2,\quad x\in\mathbb{R}^n,\;t\in[0,1].
		\end{equation}
		The real numbers $a_1,\ldots,a_n,$ which are independent of $(\Psi_t)_t$ up to order, are called the \emph{corresponding coefficients.}
	\end{dfn}
Below, we do not use the Einstein summation convention to avoid ambiguity.
	\begin{prop}
		\label{proposition:geodesic preserves orthogonal frame}
		Let $(\Lambda_t)_{t\in[0,1]}$ be a geodesic in $\laggrass^+(n)$ with derivative $(h_t)_t.$ Let $e_1,\ldots,e_n\in\mathbb{R}^n$ denote the standard basis. Let $(\Psi_t)_{t\in[0,1]}$ be a horizontal lifting of $(\Lambda_t)_t$ compatible with $h_0,$ and let $a_1,\ldots,a_n,$ denote the corresponding coefficients. For $j=1,\ldots,n$ and $t\in[0,1],$ write $g_j(t):=\langle\Psi_t(e_j),\Psi_t(e_j)\rangle$ and $g^j(t):=\left(g_j(t)\right)^{-1}.$ Then we have the following.
		\begin{enumerate}[label=(\alph*)]
			\item \label{expression for metric}
			\[
			\deriv{t}g_j(t)=-4a_j\tan\phase(\Lambda_t),\quad j=1,\ldots,n,\;t\in[0,1].
			\]
			\item \label{metric is diagonal}
			\[
			\langle\Psi_t(e_j),\Psi_t(e_k)\rangle=0,\quad j,k=1,\ldots,n,\;j\ne k,\;t\in[0,1].
			\]
			\item \label{basis element remains in complex line} For $j=1,\ldots,n,$ there exists a unique smooth function $\theta_j:[0,1]\to\mathbb{R}$ such that
			\[
			\theta_j(0)=0,\quad\Psi_t(e_j)\in e^{\theta_j(t)\ii}\mathbb{R}\langle\Psi_0(e_j)\rangle,\;t\in[0,1].
			\]
			The function $\theta_j$ satisfies
			\[
			\deriv{t}\theta_j(t)=-2a_jg^j(t).
			\]
		\end{enumerate}
	\end{prop}

	\begin{proof}
		Recall the following identity, which is proved in~\cite[Remark 5.6]{solomon}:
		\begin{equation}
		\label{equation:time-derivative of horizontal lifting}
		\deriv{t}\Psi_t(x)=-J\nabla h_t(\Psi_t(x))-\tan\phase(\Lambda_t)\nabla h_t(\Psi_t(x)),\quad x\in\mathbb{R}^n.
		\end{equation}
		Set $g_{jk}(t):=\langle\Psi_t(e_j),\Psi_t(e_k)\rangle,\;j,k=1,\ldots,n,\;t\in[0,1],$ and let $\left(g^{jk}\right)_{j,k}(t)$ denote the inverse matrix of $\left(g_{jk}\right)_{j,k}(t).$ Let $\pderiv{x^1},\ldots,\pderiv{x^n}$ denote the tangent frames of $\Lambda_t,\;t\in[0,1],$ induced by the parameterizations $\Psi_t.$ By~\eqref{equation:h does not change} we have
		\[
		\nabla h_t(\Psi_t(e_j))=2a_j \sum_{k} g^{jk}(t)\pderiv{x^k},\quad j=1,\ldots,n.
		\]
		Equation~\eqref{equation:time-derivative of horizontal lifting} can thus be rewritten as
		\begin{equation}
		\label{equation:another time-derivative of horizontal lifting}
		\deriv{t}\Psi_t(e_j)=-2a_j\sum_l g^{jl}(t)J\pderiv{x^l}-2a_j\tan\phase(\Lambda_t)\sum_{l}g^{jl}(t)\pderiv{x^l},
		\end{equation}
		for $j=1,\ldots,n.$ From the Leibniz rule we deduce
		\begin{align}
		\label{equation:time-derivative of induced metric}
		\deriv{t}g_{jk}(t)
		&=-2\tan\phase(\Lambda_t)\left(a_k \sum_l g_{jl}g^{kl}+a_j \sum_l g_{kl}g^{jl}\right)\\
		&=-4a_j\tan\phase(\Lambda_t)\delta_{jk},\nonumber
		\end{align}
		establishing parts~\ref{expression for metric} and~\ref{metric is diagonal}. Equation~\eqref{equation:another time-derivative of horizontal lifting} can now be rewritten as
		\begin{equation*}
		\label{equation:short time-derivative of horizontal lifting}
		\deriv{t}\Psi_t(e_j)=-2a_jg^j(t)J\pderiv{x^j}-2a_j\tan\phase(\Lambda_t)g^j(t)\pderiv{x^j},\quad j=1,\ldots,n.
		\end{equation*}
		Part~\ref{basis element remains in complex line} follows.
	\end{proof}

	\begin{dfn}
		\label{definition:partial phase differences}
		For a geodesic $(\Lambda_t)_t$ in $\laggrass^+(n),$ we call the functions $\theta_j$ of Proposition~\ref{proposition:geodesic preserves orthogonal frame} \ref{basis element remains in complex line} the \emph{partial phase differences}.
	\end{dfn}
\begin{rem}\label{remark: sum partial phases}
It follows from the complex linearity of $\Omega$ that if $(\Lambda_t)_t$ is a geodesic in $\laggrass^+(n)$ and $\theta_j$ are the partial phase differences, then
\[
\deriv{t}\phase(\Lambda_t) = \sum_{j = 1}^n \deriv[\theta_j]{t}(t).
\]
In particular,
\[
\phase(\Lambda_1) - \phase(\Lambda_0) = \sum_{j = 1}^n \theta_j(1).
\]
\end{rem}
The following is an immediate consequence of Proposition~\ref{proposition:geodesic preserves orthogonal frame}.
	\begin{cor}
		\label{corollary:geodesics and spectral decompositions}
		Let $(\Lambda_t)_{t\in[0,1]}$ be a geodesic in $\laggrass^+(n)$ with derivative $(h_t)_{t\in[0,1]}.$ Then the orthogonal decompositions of $\Lambda_i,\;i=0,1,$ associated to the quadratic forms $h_i$ and the induced inner product coincide with the decompositions of Lem\-ma~\ref{lemma:spectral decomposition}. Moreover, letting $(\Psi_t)_t$ be a horizontal lifting of $(\Lambda_t)_t,$ the isomorphism $\Psi_1\circ\Psi_0^{-1}:\Lambda_0\to\Lambda_1$ preserves these decompositions.
	\end{cor}

In light of Remark~\ref{remark: sum partial phases},	Proposition~\ref{proposition:geodesic preserves orthogonal frame} \ref{basis element remains in complex line} also provides a simple proof of the following statement, which is proved in greater generality in~\cite{solomon-curv} and~\cite{thomas-yau}.

	\begin{cor}
		\label{corollary:time-derivative of phase}
		Let $(\Lambda_t)_t$ be a geodesic in $\laggrass^+(n)$ with derivative $(h_t)_t.$ Then we have
		\[
		\deriv{t}\phase(\Lambda_t)= \Delta h_t,
		\]
		where $\Delta$ denotes the geometer's Laplacian with respect to the induced metric on $\Lambda_t.$ In particular, if the quadratic forms $h_t$ are positive/negative semi-definite, then $\phase(\Lambda_t)$ is a monotone function of $t.$
	\end{cor}

	Let $\Lambda_0,\Lambda_1\in\laggrass^+(n)$ and suppose we wish to find a geodesic between $\Lambda_0$ and $\Lambda_1.$ In other words, we wish to find a tangent vector $h_0\in T_{\Lambda_0}\laggrass(n)\cong Q(\Lambda_0)$ with $\exp_{\Lambda_0}(h_0)=\Lambda_1.$ In view of Proposition~\ref{proposition:geodesic preserves orthogonal frame} and Corollary~\ref{corollary:geodesics and spectral decompositions}, the following is a natural approach. Identify $\Lambda_0$ with $\mathbb{R}^n$ via an orthonormal basis compatible with the orthogonal decomposition
	\[
	\Lambda_0=\bigoplus_jV_{0,j}
	\]
	of Lemma~\ref{lemma:spectral decomposition}. With respect to this identification, the desired $h_0,$ if it exists, has to be given by
	\[
	h_0(x)=\sum_ja_jx_j^2
	\]
	for some $a_1,\ldots,a_n\in\mathbb{R}.$ The coefficients $a_j,\;j=1,\ldots,n,$ determine the time-evolution of the partial phase differences, $\theta_1,\ldots,\theta_n,$ and thus need to be chosen so that $\theta_j(1)$ represents the argument $\alpha_j$ of Lemma~\ref{lemma:spectral decomposition} for $j=1,\ldots,n.$ In particular, each of the functions $\theta_j$ is monotone, where the sign of $a_j$ determines whether it is increasing or decreasing. It thus makes sense that the Maslov index defined below is related to geodesics in $\laggrass^+(n).$ We recall the definition given in~\cite{solomon-verbitsky}.
	
	\begin{dfn}
		\label{definition:Maslov index}
		Let $\Lambda_0,\Lambda_1\in\laggrass^+(n).$ Let $\alpha_1,\ldots,\alpha_n\in\mathbb{R}/\pi\mathbb{Z}$ be the arguments of Lemma~\ref{lemma:spectral decomposition} with the right multiplicities. For $j=1,\ldots,n,$ let $\beta_j\in[0,\pi)$ be the unique representative of $\alpha_j.$ The \emph{Maslov index} of the pair $(\Lambda_0,\Lambda_1)$ is the integer given by
		\[
		m(\Lambda_0,\Lambda_1):=\frac{1}{\pi}\left(\sum_j\beta_j+\phase(\Lambda_0)-\phase(\Lambda_1)\right).
		\]
		For a given $\Lambda_0\in\laggrass^+(n),$ we set
		\[
		\overline{\zeromas(\Lambda_0)}:=\left\{\left.\Lambda_1\in\laggrass^+(n)\;\right|\;m(\Lambda_0,\Lambda_1)=0\right\}
		\]
		and
		\[
		\zeromas(\Lambda_0):=\left\{\left.\Lambda_1\in\overline{\zeromas(\Lambda_0)}\;\right|\;\Lambda_0\cap\Lambda_1=\{0\}\right\}.
		\]
		As implied by the notation, $\overline{\zeromas(\Lambda_0)}$ is indeed the closure of $\zeromas(\Lambda_0).$ We remark that both spaces are connected. We let $\geo(\Lambda_0)\subset\zeromas(\Lambda_0)$ and $\overline{\geo(\Lambda_0)}\subset\overline{\zeromas(\Lambda_0)}$ denote the subsets consisting of Lagrangians that can be connected with $\Lambda_0$ by a geodesic with negative semi-definite derivative.

More generally, if $(X,\omega,J,\Omega)$ is a Calabi-Yau manifold and $\Lambda_0,\Lambda_1 \subset X$ are positive Lagrangian submanifolds, the Maslov index of an intersection point $q \in \Lambda_0,\Lambda_1$ is given by
\[
m(q;\Lambda_0,\Lambda_1): = m(T_q\Lambda_0,T_q\Lambda_1).
\]
	\end{dfn}

	We divide the proof of Theorem~\ref{theorem:linear geodesic} into three steps. Fix $\Lambda_0\in\laggrass^+(n).$ In Corollary~\ref{corollary:space of good Lags is non-empty} we show that $\geo(\Lambda_0)$ is non-empty. In Proposition~\ref{proposition:space of good Lags is closed} we show that $\overline{\geo(\Lambda_0)}$ is a closed subset of $\overline{\zeromas(\Lambda_0)}.$ In Proposition~\ref{proposition:space of good Lags is open} we show that $\geo(\Lambda_0)$ is an open subset of $\zeromas(\Lambda_0).$
	
	\begin{lemma}
		\label{lemma:negative derivative implies Maslov zero}
		Let $(\Lambda_t)_{t\in[0,1]}$ be a geodesic in $\laggrass^+(n)$ with derivative $(h_t)_t.$ Suppose the quadratic forms $h_t$ are negative semi-definite.
\begin{enumerate}[label = (\alph*)]
\item \label{item:Maslov}
We have $m(\Lambda_0,\Lambda_1)=0.$
\item \label{item:negative definite}
If $\Lambda_0 \cap \Lambda_1 = \{0\},$ then the quadratic forms $h_t$ are negative definite.
\end{enumerate}
	\end{lemma}

	\begin{proof}
		Choose a horizontal lifting of $(\Lambda_t)_t$ compatible with $h_0,$ let $a_1,\ldots,a_n$ denote the corresponding coefficients and let $\theta_1,\ldots,\theta_n$ denote the partial phase differences. By assumption, the coefficients $a_j,\;j=1,\ldots,n,$ are all non-positive. It thus follows from Proposition~\ref{proposition:geodesic preserves orthogonal frame} \ref{basis element remains in complex line} that
		\begin{equation}\label{equation:theta positive}
		\theta_j(1)\ge0,\quad j=1,\ldots,n,
		\end{equation}
with equality only when $a_j = 0.$ If $\theta_j = 0$ for some $j,$ then $\Lambda_0 \cap \Lambda_1 \neq \{0\}.$ So, part~\ref{item:negative definite} of the lemma follows.

		As $\phase(\Lambda_0),\phase(\Lambda_1)\in\left(-\frac{\pi}{2},\frac{\pi}{2}\right),$
Remark~\ref{remark: sum partial phases} and inequality~\eqref{equation:theta positive} imply that
		\[
		\theta_j(1)<\pi,\quad j=1,\ldots,n.
		\]
		The partial phase differences $\theta_j(1),\;j=1,\ldots,n,$ thus coincide with the representatives $\beta_j,\;j=1,\ldots,n,$ of Definition~\ref{definition:Maslov index}, which implies part~\ref{item:Maslov} of the lemma.
	\end{proof}
	
	\begin{cor}
		\label{corollary:space of good Lags is non-empty}
		For $\Lambda_0\in\laggrass^+(n),$ the space $\geo(\Lambda_0)$ is non-empty.
	\end{cor}

	\begin{proof}
		As $\laggrass^+(n)$ is finite-dimensional, the exponential map $\exp_{\Lambda_0}$ is well-defined on an open neighborhood $0\in U\subset T_{\Lambda_0}\laggrass(n)=Q(\Lambda_0).$ Let $h\in U$ be negative definite. By Lemma~\ref{lemma:negative derivative implies Maslov zero}~\ref{item:Maslov} we have $\exp_{\Lambda_0}(h)\in\overline{\zeromas(\Lambda_0)}.$ By Proposition~\ref{proposition:geodesic preserves orthogonal frame} \ref{basis element remains in complex line}, the partial phase differences are strictly increasing, implying that
\[
\Lambda_0\cap\exp_{\Lambda_0}(h)=\{0\}.
\]
It follows that $\exp_{\Lambda_0}(h)\in\geo(\Lambda_0).$
	\end{proof}

	The main ingredient in the proof of Proposition~\ref{proposition:space of good Lags is closed} is the following a priori estimate.
	
	\begin{lemma}
		\label{lemma:negative definite geodesic has bounded derivative}
		Let $-\frac{\pi}{2}<\alpha_0<\alpha_1<\frac{\pi}{2}.$
		\begin{enumerate}[label=(\alph*)]
			\item\label{metric is uniformly bounded} There exists a positive constant $N=N(\alpha_1)$ with the following property. Let $(\Lambda_t)_{t\in[0,1]}$ be a geodesic with negative semi-definite derivative $(h_t)_t,$ such that $\phase(\Lambda_0),\phase(\Lambda_1)\in[\alpha_0,\alpha_1].$ Let $(\Psi_t)_t$ be a horizontal lifting of $(\Lambda_t)_t$ compatible with $h_0$ and let $a_1,\ldots,a_n$ denote the corresponding coefficients. Let $e_1,\ldots,e_n$ denote the standard basis of $\mathbb{R}^n$ and write
			\[
			g_j(t):=\langle\Psi_t(e_j),\Psi_t(e_j)\rangle,\;g^j(t):=\left(g_j(t)\right)^{-1},\quad j=1,\ldots,n,\quad t\in[0,1].
			\]
			Then we have
			\[
			g_j(t)\le N,\quad j=1,\ldots,n,\;t\in[0,1].
			\]
			\item\label{coefficients are uniformly bounded} There exists a positive constant $M=M(\alpha_0,\alpha_1)$ such that, in the setting of part~\ref{metric is uniformly bounded}, we have
			\[
			a_j>-M,\quad j=1,\ldots,n.
			\]
		\end{enumerate}
	\end{lemma}

	\begin{proof}
		Let $(\Lambda_t)_t,(h_t)_t,(\Psi_t)_t,$ $a_j,j=1,\ldots,n,$ and $g_j,j=1,\ldots,n,$ be as in~\ref{metric is uniformly bounded}. We assume $h_t\ne0,$ otherwise there is nothing to prove. By Corollary~\ref{corollary:time-derivative of phase}, $\phase(\Lambda_t)$ is strictly increasing with time-derivative
		\begin{equation}
		\label{equation:explicit time-derivative of phase}
		\deriv{t}\phase(\Lambda_t)=\Delta h_t=-2\sum_j g^j(t)a_j.
		\end{equation}
		In particular, we have
		\[
		\phase(\Lambda_t)\in[\alpha_0,\alpha_1],\quad t\in[0,1].
		\]
		Switching the roles of $\phase(\Lambda_t)$ and $t,$ we write
		\begin{equation}
		\label{equation:derivative of time by phase}
		\deriv[t]{\phase(\Lambda_t)}=-\frac{1}{2\sum_j g^j(t)a_j}.
		\end{equation}
		Let $j\in\{1,\ldots,n\}.$ If we have $a_j=0,$ it follows from Proposition~\ref{proposition:geodesic preserves orthogonal frame} \ref{expression for metric} that $g_j(t)\equiv1.$ Hence, we assume $a_j<0.$ By Proposition~\ref{proposition:geodesic preserves orthogonal frame} \ref{expression for metric} and~\eqref{equation:derivative of time by phase} we have
		\begin{equation}
		\label{equation:derivative of g by phase}
		\deriv[g_j]{\phase(\Lambda_t)}=2\frac{a_j\tan\phase(\Lambda_t)}{\sum_k g^k a_k}.
		\end{equation}
		If $\phase(\Lambda_t)\le0,$ it follows that
		\begin{equation}
		\label{equation:g is decreasing}
		\deriv[g_j]{\phase(\Lambda_t)}\le0.
		\end{equation}
		If $\phase(\Lambda_t)>0,$ equation~\eqref{equation:derivative of g by phase} implies
		\begin{align}
		\label{equation:derivative of g by phase is bounded}
		\deriv[g_j]{\phase(\Lambda_t)}
		&\le2\frac{a_j\tan\phase(\Lambda_t)}{g^j a_j}\\
		&=2\tan\phase(\Lambda_t)g_j\nonumber\\
		&\le2\tan\alpha_1g_j.\nonumber
		\end{align}
		The differential inequalities~\eqref{equation:g is decreasing} and~\eqref{equation:derivative of g by phase is bounded} imply that, if $\alpha_1\le0,$ we have
		\[
		g_j(t)\le1,\quad t\in[0,1],
		\]
		whereas if $\alpha_1>0,$ we have
		\[
		g_j(t)\le e^{\pi\tan\alpha_1},\quad t\in[0,1],
		\]
		establishing part~\ref{metric is uniformly bounded}.
		
		By~\ref{metric is uniformly bounded} we have
		\[
		g^j(t)\ge\frac{1}{N},\quad j=1,\ldots,n,\;t\in[0,1].
		\]
		Integrating equation~\eqref{equation:explicit time-derivative of phase}, we obtain
		\begin{align*}
		\alpha_1-\alpha_0
		&=-2\sum_j a_j\int_0^1g^j(t)dt\\
		&\ge-\frac{2}{N}\sum_ja_j.
		\end{align*}
		Part~\ref{coefficients are uniformly bounded} follows.
	\end{proof}

	\begin{prop}
		\label{proposition:space of good Lags is closed}
		Let $\Lambda_0\in\laggrass^+(n).$ Then $\overline{\geo(\Lambda_0)}$ is a closed subset of $\overline{\zeromas(\Lambda_0)}.$
	\end{prop}

	\begin{proof}
		Let $(\Lambda_{1,k})_{k\in\mathbb{N}}$ be a sequence in $\overline{\geo(\Lambda_0)}$ converging to $\Lambda_{1,\infty}\in\overline{\zeromas(\Lambda_0)}.$ For $k\in\mathbb{N},$ let $(\Lambda_{t,k})_{t\in[0,1]}$ be a geodesic between $\Lambda_0$ and $\Lambda_{1,k}$ with negative semi-definite derivative. Write $\phase(\Lambda_0)=:\alpha_0.$ As the sequence $(\Lambda_{1,k})_k$ is convergent, there exists some $\alpha_1\in\left(-\frac{\pi}{2},\frac{\pi}{2}\right)$ with
		\[
		\phase(\Lambda_{1,k})<\alpha_1,\quad k\in\mathbb{N}.
		\]
		It follows that all the geodesics $(\Lambda_{t,k})_t$ have images contained in the compact set
		\[
		\laggrass_{\alpha_0,\alpha_1}^+(n):=\left\{\left.\Lambda\in\laggrass^+(n)\;\right|\;\phase(\Lambda)\in[\alpha_0,\alpha_1]\right\}.
		\]
		Also, by Lemma~\ref{lemma:negative definite geodesic has bounded derivative} \ref{coefficients are uniformly bounded}, there exists a compact set $K\subset T_{\Lambda_0}\laggrass(n)$ containing all the time-derivatives $\left.\deriv{t}\right|_{t=0}\Lambda_{t,k}.$ It follows that for $l\in\mathbb{N},$ the derivatives of all geodesics in question up to order $l$ are uniformly bounded. By the Arzel\`a-Ascoli theorem, there exists a subsequence of $((\Lambda_{t,k})_t)_k$ converging in the $C^2$-sense to a path $(\Lambda_{t,\infty})_{t\in[0,1]},$ which is a geodesic between $\Lambda_0$ and $\Lambda_{1,\infty}.$ As the space of negative semi-definite quadratic forms is a closed subset in the space of general quadratic forms, the geodesic $(\Lambda_{t,\infty})_{t\in[0,1]}$ has negative semi-definite derivative.
	\end{proof}

	The proof of Proposition~\ref{proposition:space of good Lags is open} below is based on the relation between geodesics and imaginary special Lagrangian cylinders, as summarized in Section~\ref{subsection:special Lagrangian cylinders and their relation to geodesics}. In addition, it relies on the following elementary observation concerning geodesics and Jacobi fields with respect to arbitrary connections.
	
	\begin{lemma}
		\label{lemma:tangent Jacobi field}
		Let $M$ be a smooth finite-dimensional manifold equipped with an affine connection. Let $\gamma:[0,1]\to M$ be a geodesic, and let $J\in\Gamma\left([0,1],\gamma^*TM\right)$ be a Jacobi field tangent to $\gamma.$ That is, we have $J=f\dot{\gamma}$ for some $f:[0,1]\to\mathbb{R}.$ If we have $J(0)=0$ and $J(1)=0,$ then $J$ is the trivial Jacobi field.
	\end{lemma}

For $\Lambda_0\in\laggrass^+(n),$ let $T_{\Lambda_0}^-\laggrass^+(n) \subset T_{\Lambda_0}\laggrass^+(n)$ denote the cone consisting of negative definite quadratic forms.

	\begin{prop}
		\label{proposition:space of good Lags is open}
Let $(\Lambda_t)_{t\in [0,1]}$ be a geodesic segment in $\laggrass^+(n)$ with negative semi-definite derivative and $\Lambda_0 \cap \Lambda_1 = \{0\}.$ Then, there exist open neighborhoods
\[
\Lambda_i \in U_i \subset \laggrass^+(n), \qquad  i = 0,1,
\]
and a smooth map
\[
\Gamma : [0,1] \times U_0 \times U_1 \to \laggrass^+(n),
\]
such that for $(u_0,u_1) \in U_0\times U_1$ the path
\[
t \mapsto \Gamma(t,u_0,u_1)
\]
is a geodesic from $u_0$ to $u_1$ and
\[
 \Gamma(t,\Lambda_0,\Lambda_1) = \Lambda_t,\qquad t \in [0,1].
\]
In particular, the space $\geo(\Lambda_0)$ is an open subset of $\zeromas(\Lambda_0).$
	\end{prop}

	\begin{proof}
		It suffices to show that $\Lambda_1$ is not a conjugate point of $\Lambda_0$ along the geodesic~$(\Lambda_t)_t.$ Thus, keeping in mind Lemma~\ref{lemma:negative derivative implies Maslov zero}~\ref{item:negative definite}, it suffices to prove the following.

\vspace{6pt}
\noindent\textbf{Claim:}
Let $(\Lambda_t)_{t\in[0,1]}$ be a geodesic with negative definite derivative, and let $J=(J_t)_{t\in[0,1]}$ be a Jacobi field along $(\Lambda_t)_t$ satisfying
		\[
		J_0=0,\quad J_1=0.
		\]
		Then we have $J_t=0$ for $t\in[0,1].$
\vspace{6pt}
	
		In view of Lemma~\ref{lemma:tangent Jacobi field}, it suffices to prove that, under the assumptions of the above claim, the Jacobi field $J$ is tangent to the geodesic $(\Lambda_t)_t.$ Let the two parameter family $(\Lambda_{t,s})_{(t,s)\in[0,1]\times(-\epsilon,\epsilon)}$ be a variation realizing the Jacobi field $J.$ In other words, the path $(\Lambda_{t,s})_{t\in[0,1]}$ is a geodesic with $\Lambda_{0,s}=\Lambda_0$ for $s\in(-\epsilon,\epsilon),$ and we have
		\[
		\left.\pderiv{s}\right|_{s=0}\Lambda_{t,s}=J_t,\quad t\in[0,1].
		\]
		For $(t,s)\in[0,1]\times(-\epsilon,\epsilon),$ write
		\[
		h_{t,s}:=\pderiv{t}\Lambda_{t,s}\in Q(\Lambda_{t,s}).
		\]
		By assumption, diminishing $\epsilon$ if necessary, all quadratic forms $h_{t,s}$ are negative definite and thus have regular level sets diffeomorphic to $S^{n-1}.$
Recalling Definition~\ref{dfn:cylinder}, for $s\in(-\epsilon,\epsilon),$ let $Z_s \in\mathcal{SLC}\left(S^{n-1};\Lambda_0,\Lambda_{1,s}\right)$ be the cylinder of $-1$-level sets associated to the geodesic $(\Lambda_{t,s})_t.$ Abbreviate $L = S^{n-1} \times [0,1].$

We construct a smooth map
\[
\Upsilon: L\times(-\epsilon,\epsilon)\to\mathbb{C}^n
\]
such that for $s\in(-\epsilon,\epsilon)$ the restricted map $\Upsilon_s:=\Upsilon|_{L\times\{s\}}$ is an immersion representing the cylinder $Z_s$ and adapted to its harmonics. The construction is carried out as follows.
Choose a linear parameterization $\Psi_0: \R^n \to \Lambda_0$ and for $s \in (-\epsilon,\epsilon),$ let $(\Psi_{t,s})_{t \in [0,1]}$ be the horizontal lifting of the geodesic $(\Lambda_{t,s})_{t}$ with $\Psi_{0,s} = \Psi_0.$ Let
\[
h^{s} := h_{0,s}\circ \Psi_0 : \R^n \to \R
\]
and let
\[
\Theta : S^{n-1}\times (-\epsilon,\epsilon) \to \R^n
\]
be a smooth map such that $\Theta_s = \Theta|_{S^{n-1}\times\{s\}}$ carries the sphere $S^{n-1} \times \{s\}$ diffeomorphically onto the level set $(h^{s})^{-1}(-1)$. Take
\begin{equation}\label{equation:Upsilon}
\Upsilon(p,t,s) := \Psi_{t,s}(\Theta_s(p)).
\end{equation}
By definition, $\Upsilon_s$ represents the cylinder of $-1$ level sets $Z_s.$
Proposition~\ref{proposition:relevant results from previous paper}~\ref{item:harmonic} asserts that $\Upsilon _s$ is adapted to the harmonics of $Z_s.$

Define a map $v:L\to\mathbb{C}^n$ by
		\[
		v(p,t):=\left.\deriv{s}\right|_{s=0}\Upsilon(p,t,s),\qquad(p,t)\in S^{n-1}\times[0,1] = L.
		\]
As all the cylinders $Z_s$ are Lagrangian, the 1-form $\tau:=i_v\omega\in\Omega^1(L)$ is closed. For $i = 0,1,$ since $J_i=0,$ the vector $v(q)$ is tangent to $\Lambda_i$ for $q\in S^{n-1}\times \{i\}.$ It follows that the 1-form $\tau$ annihilates the boundary component $S^{n-1}\times \{i\} \subset L.$ By Lemma~\ref{lemma:closed and annihilating is exact} the $1$-form $\tau$ is exact. Let $u\in \cob{\infty}(L)$ be the unique function satisfying
		\[
		du=\tau.
		\]
As the cylinders $Z_s$ are imaginary special Lagrangian, Lemma~\ref{lemma:Laplacian}~\ref{linearization is Laplacian} implies that $\Delta u=0.$
Since $\Upsilon_0$ is adapted to the harmonics of $Z_0,$  Lemma~\ref{lemma:Laplacian}~\ref{elliptic} implies that $u|_{S^{n-1}\times \{t\}}$ is constant for $t \in [0,1].$

		For $t\in[0,1],$ the tangent vector $J_t$ is a quadratic form on $\Lambda_t.$  It follows from equation~\eqref{equation:Upsilon} that $\Upsilon_0$ carries $S^{n-1}\times \{t\}$ into $\Lambda_{t,0}.$
		By construction, for $t\in[0,1],$ a point $q\in S^{n-1}\times\{t\}$ and a tangent vector
\[
\xi\in T_q(S^{n-1} \times \{t\}),
\]
we have
		\begin{align*}
		0
		&=du_q(\xi)\\
		&=\omega(v(q),d(\Upsilon_0)_{q}(\xi))\\
		&=d(J_t)_{\Upsilon_0(q)}(d(\Upsilon_0)_{q}(\xi))\\
		&=d\left(J_t \circ \Upsilon_0|_{S^{n-1}\times \{t\}}\right)_q(\xi).
		\end{align*}
Thus, for $t \in [0,1]$	the function $J_t \circ \Upsilon_0|_{S^{n-1}\times \{t\}}$ is constant. By the definition of the cylinder of $-1$ level sets,
\[
h_{t,0} \circ \Upsilon_0|_{S^{n-1}\times \{t\}} = -1.
\]
Since $J_t$ and $h_{t,0}$ are both quadratic forms on $\Lambda_{t,0}$, and they share the level set $\Upsilon_0(S^{n-1}\times \{t\}),$ it follows that they agree up to a constant multiple possibly depending on $t.$		
Thus, the Jacobi field $J$ is tangent to the geodesic $(\Lambda_t)_t.$ The claim follows.
	\end{proof}

	\begin{proof}[Proof of Theorem~\ref{theorem:linear geodesic}]
Since $m_0(\Lambda_0)$ is connected, the theorem follows from Corollary~\ref{corollary:space of good Lags is non-empty}, Proposition~\ref{proposition:space of good Lags is closed} and Proposition~\ref{proposition:space of good Lags is open}.
	\end{proof}

\begin{rem}
Let $(\Lambda_t)_{t \in [0,1]}$ be a geodesic in $\laggrass^+(n)$ with $m(\Lambda_0,\Lambda_1) = 0.$ For $c < 0,$ the cylinder of $c$-level sets associated to $(\Lambda_t)_t$ coincides with one of the special Lagrangian cylinders constructed by Lawlor~\cite[Prop. 9]{Law89}. See also~\cite[Ch.~7]{Har90} and~\cite[Sec. 5]{Joy01}. Lawlor used these cylinders to prove the angle criterion, which determines when the union of two $n$-planes in $\R^{2n}$ is area  minimizing. To deduce the coincidence of the cylinders of $c$-level sets with the cylinders of Lawlor, observe that both contain an ellipsoid in a Lagrangian linear subspace, which is a real analytic isotropic submanifold of $\C^n$ of dimension $n-1,$ so they must coincide by Theorem III.5.5 of~\cite{harvey-lawson}.

In fact, one can deduce the result of~\cite[Prop. 9]{Law89} from Theorem~\ref{theorem:linear geodesic}. The angle criterion translates to the Maslov zero condition. Conversely, an alternative route to the proof of Theorem~\ref{theorem:linear geodesic} starts with the special Lagrangian cylinders of~\cite[Prop.~9]{Law89} and applies the inverse cylindrical transform. To apply the inverse cylindrical transform, one must analyze harmonic functions on the special Lagrangian cylinders. We believe the independent proof of Theorem~\ref{theorem:linear geodesic} given here is nonetheless valuable as a simple example of the positive Lagrangian geodesic equation that can be solved explicitly and to provide a different perspective on the work of~\cite{Har90, Joy01,Law89}
\end{rem}

	\section{Geodesics of Lagrangian cones}
	\label{section:geodesics of Lagrangian cones}
	
A \emph{smooth immersed Lagrangian cone} in $\C^n$ is a cone-smooth Lagrangian submanifold of $\C^n$ that can be represented by a cone-immersion $(f : \R^n \to \C^n, 0)$ that is  $1$-homogeneous. Let $\lagcones(n)$ denote the space of smooth immersed Lagrangian cones in $\mathbb{C}^n$ and $\lagcones^+(n)\subset\lagcones(n)$ the open subspace consisting of positive Lagrangian cones. Then $\lagcones^+(n)$ is a smooth Fr\'echet manifold containing $\laggrass^+(n).$ The goal of this section, accomplished in Proposition~\ref{proposition:geodesic is in fact linear}, is to show that, under some assumptions, geodesics of positive Lagrangian cones are in fact linear.
The following description of the tangent bundle $T\lagcones(n)$ is analogous to that of Lemma~\ref{lemma:tangent space of Lagrangian Grassmannian}.
	
	\begin{lemma}
		\label{lemma:tangent space at a Lag cone}
		For $\Lambda\in\lagcones(n)$ there is a canonical isomorphism $T_\Lambda\lagcones(n)\cong\mathcal{H}(\Lambda),$ the space of cone-smooth degree-2 homogeneous functions on $\Lambda.$
	\end{lemma}
	
The positive Lagrangian connection on $\laggrass^+(n)$ extends naturally to $\lagcones^+(n).$  Indeed, we have the following analog of Lemma~\ref{lemma:there is always a horizontal lifting}.
	\begin{lemma}
		\label{lemma:cone horizontal lifting}
		Let $(\Lambda_t)_{t\in[0,1]}$ be a smooth path in $\lagcones^+(n)$. Let $(\Psi_0:\mathbb{R}^n\to \C^n,0)$ be a cone-smooth $1$-homogeneous immersion representing $\Lambda_0.$ Then $\Psi_0$ extends uniquely to a cone-smooth horizontal lifting, $(\Psi_t:\mathbb{R}^n\to\mathbb{C}^n,0)_{t\in[0,1]},$ of the path $(\Lambda_t)_t.$ Moreover, the unique lifting $(\Psi_t)_t$ is $1$-homogeneous.
	\end{lemma}

This lemma enables us to parallel transport homogeneous functions along paths in $\lagcones^+(n).$ Hence, the notion of geodesics in $\lagcones^+(n)$ makes sense. The following lemma describes the relation between geodesics in $\lagcones^+(n)$ and special Lagrangian cylinders. For the purposes of this work, we may restrict the discussion to geodesics with linear endpoints. For a Lagrangian cone $\Lambda$ and a homogeneous function $h\in\mathcal{H}(\Lambda),$ we say $h$ is \emph{positive} (\emph{negative}) if we have $h(p)>0\;(h(p)<0)$ for $p\in\Lambda\setminus\{0\}.$ For $\Lambda_0, \Lambda_1 \in \laggrass^+(n),$ the positive real numbers $\R_{>0}$ act on the space of Lagrangian cylinders $\mathcal{SLC}\left(S^{n-1};\Lambda_0,\Lambda_1\right)$ by rescaling $[\chi] \mapsto [s\chi], \; s \in \R_{>0}.$
	
	\begin{lemma}
		\label{lemma:geodesic of cones}
		Let $\Lambda_0,\Lambda_1\in\laggrass^+(n).$ Then the cylindrical transform gives a bijection from geodesics in $\lagcones^+(n)$ between $\Lambda_0$ and $\Lambda_1$ with positive/negative derivative to $\R_{>0}$ orbits in $\mathcal{SLC}\left(S^{n-1};\Lambda_0,\Lambda_1\right)$ consisting of imaginary special Lagrangian cylinders nowhere tangent to the Euler vector field.
	\end{lemma}
	
	\begin{proof}
		Let $(\Lambda_t)_{t\in[0,1]}$ be a geodesic in $\lagcones^+(n)$ with positive derivative, denoted by $(h_t)_t,$ between $\Lambda_0$ and $\Lambda_1.$ As the homogeneous functions $h_t$ are positive, they admit the regular value $c$ for each $c > 0.$ Moreover, the level set $h^{-1}(c)$ is diffeomorphic to $S^{n-1}.$ Let $Z$ be the cylinder of $c$-level sets. By Proposition~\ref{proposition:relevant results from previous paper}, the cylinder $Z$ is imaginary special Lagrangian with regular harmonics. By Lemma~\ref{lemma: cylinder with regular harmonics in Euclidean space}, the cylinder $Z$ is nowhere tangent to the Euler vector field. By the homogeneity of $(h_t)_t,$ the cylinders of $c$-level sets for different values of $c$ are related by rescaling.
		
		Conversely, let $Z\in\mathcal{SLC}\left(S^{n-1};\Lambda_0,\Lambda_1\right)$ be a cylinder nowhere tangent to the Euler vector field. Let $\chi:S^{n-1}\times[0,1]\to\mathbb{C}^n$ be an immersion representing $Z$ adapted to the harmonics of $Z.$ For $s \in (0,\infty),$ let
\[
Z_s := [s\chi]
\]
be the rescaling of $Z$ by $s.$ Let
\[
U = \{Z_s\}_{s \in (0,\infty)} \subset \mathcal{SLC}\left(S^{n-1};\Lambda_0,\Lambda_1\right).
\]
Define
\[
\Phi : S^{n-1} \times [0,1] \times [0,\infty) \to \C^n
\]
by
\[
\Phi(p,t,s) = s\chi(p,t).
\]
Recalling Definition~\ref{definition: interior regularity}, since $Z$ is nowhere tangent to the Euler vector field, it follows that $\Phi$ is a regular parameterization of $U$ about $0 \in \C^n.$ So, Lemma~\ref{lemma: geodesic of small Lagrangians} gives us a unique geodesic $(\Lambda_t)_t$ of positive Lagrangians from $\Lambda_0$ to $\Lambda_1$ with cylindrical transform~$U.$ By uniqueness, this geodesic must be invariant under rescaling by $s \in (0,\infty),$ so it lies in $\lagcones^+(n).$ Let $h_t : \Lambda_t \to \R$ be given by $h_t = \deriv[\Lambda_t]{t}.$ By definition of the cylindrical transform, the level sets of $h_t$ are spheres of the form
$\Phi(S^{n-1}\times \{t\}\times \{s\})$ for $s \in (0,\infty),$ and in particular, compact. Hence, $h_t$ is either positive or negative.
	\end{proof}

The following definition is a simplified version of Definition~\ref{dfn:topology on geodesics} for the space of geodesics in $\lagcones^+(n)$ with linear endpoints.
\begin{dfn}
Let
\[
\mathfrak{Z}_{\laggrass^+(n)} := \left\{(\Lambda_0,\Lambda_1,\mathcal{Z})\left |
\begin{matrix}
\Lambda_i \in \laggrass^+(n),\; i = 0,1,\quad \Lambda_0 \cap \Lambda_1 = \{0\}, \; \\
\mathcal{Z} \subset \mathcal{SLC}(\Lambda_0,\Lambda_1) \text{ an $\R_{>0}$ orbit}
\end{matrix}
\right.\right\}.
\]
We define the $C^{k,\alpha}$ topology on $\mathfrak{Z}_{\laggrass^+(n)}$ as follows.
For
\[
\mathcal{V} \subset C^\infty(S^{n-1}\times [0,1],\C^n),
\]
an open subset in the $C^{k,\alpha}$ topology, write
\[
\mathcal{X}_{\mathcal{V}} = \left \{(\Lambda_0,\Lambda_1,\mathcal{Z}) \in \mathfrak{Z}_{\laggrass^+(n)} \left|
\begin{matrix}
\exists Z \in \mathcal{Z}, \; \exists f : S^{n-1}\times [0,1] \to \C^n \text{ representing } Z \\
\text{such that } f \in \mathcal{V}
\end{matrix}
\right.\right\}.
\]
Then, a basis for the $C^{k,\alpha}$ topology on $\mathfrak{Z}_{\laggrass^+(n)}$ is given by sets of the form $\mathcal{X}_\mathcal{V}.$  Let
\[
\mathfrak{G}_{\laggrass^+(n)} : = \left\{(\Lambda_t)_{t \in [0,1]} \left|
\begin{matrix}
(\Lambda_t)_{t \in [0,1]} \text{ is a geodesic in } \lagcones^+(n)\\ \text{ with positive/negative derivative},\\
\Lambda_0, \Lambda_1 \in \laggrass^+(n), \quad \Lambda_0 \cap \Lambda_1 = \{0\}
\end{matrix}
\right. \right\}
\]
denote the space of geodesics in $\lagcones^+(n)$ with linear endpoints intersecting transversally.
By Lemma~\ref{lemma:geodesic of cones}, the cylindrical transform gives a bijection
\[
\mathfrak{G}_{\laggrass^+(n)} \simeq \mathfrak{Z}_{\laggrass^+(n)}.
\]
So, the $C^{k,\alpha}$ topology on $\mathfrak{Z}_{\laggrass^+(n)}$ gives rise to a topology on $\mathfrak{G}_{\laggrass^+(n)},$ which we also call the $C^{k,\alpha}$ topology.
\end{dfn}

The following setting will be referred to in Lemmas~\ref{lemma: Dirichlet} and~\ref{lemma:perturbed geodesic of cones} below.
\begin{set}\label{set: one}
		Let $(\Lambda_{0,r})_{r\in[0,1]}$ and $(\Lambda_{1,r})_{r\in[0,1]}$ be smooth paths in $\laggrass^+(n)$ with
\[
\Lambda_{0,r}\cap\Lambda_{1,r}=\{0\},\;r\in[0,1].
\]
Let $(\Lambda_{t,0})_{t\in[0,1]}$ be a geodesic in $\lagcones^+(n)$ with negative derivative between $\Lambda_{0,0}$ and $\Lambda_{1,0}.$
\end{set}
\begin{ntn}\label{ntn: Weinstein}
In Setting~\ref{set: one}, abbreviate $L : = S^{n-1}\times [0,1].$	Let
\[
Z_0 = [f : L \to \C^n]\in\mathcal{SLC}\left(S^{n-1};\Lambda_{0,0},\Lambda_{1,0}\right)
\]
denote the cylinder of $-1$-level sets associated to $(\Lambda_{t,0})_t.$ Recalling Lemma~\ref{lemma:Weinstein neighborhood of cylinder}, choose a Weinstein neighborhood $(V,\psi)$ of $Z_0$ compatible with $\Lambda_{0,0}$ and $\Lambda_{1,0},$ where $V \subset T^*L$ and $\psi:V \to \C^n$ with $\psi|_L = f.$ Let $\pi_L : T^*L \to L$ denote the projection.
Let $\alpha\in(0,1).$ For $u\in\cob{2,\alpha}(L),$ let $\operatorname{Graph}(du) \subset T^*L$ denote the graph. For $u$ small enough that $\operatorname{Graph}(du) \subset V,$ let $j_u : L \to \C^n$ be given by
\[
j_u = \psi \circ \left( \pi_L|_{\operatorname{Graph}(du)}\right)^{-1}.
\]
Let $\left(\varphi_r:\mathbb{C}^n\to\mathbb{C}^n\right)_{r\in[0,1]}$ be a smooth Hamiltonian isotopy such that $\varphi_r$ carries $\Lambda_{0,0}$ to $\Lambda_{0,r}$ and $\Lambda_{1,0}$ to $\Lambda_{1,r}.$
\end{ntn}

\begin{lemma}\label{lemma: Dirichlet}
In Setting~\ref{set: one}, suppose the geodesic of positive Lagrangian cones $(\Lambda_{t,0})_{t \in [0,1]}$ extends to a family  $(\Lambda_{t,r})_{(t,r)\in[0,1]\times[0,\epsilon)}$ satisfying the following:
\begin{itemize}
\item
For $r\in[0,\epsilon),$ the path $(\Lambda_{t,r})_{t\in[0,1]}$ is a geodesic in $\lagcones^+(n)$ with negative derivative between $\Lambda_{0,r}$ and $\Lambda_{1,r}.$
\item
The geodesic $(\Lambda_{t,r})_{t\in[0,1]}$ depends continuously on $r$ with respect to the $C^{1,\alpha}$ topology on $\mathfrak{G}_{\laggrass^+(n)}.$
\end{itemize}
For $r \in [0,\epsilon)$ let $Z_r$ denote the cylinder of $-1$-level sets associated to the geodesic $(\Lambda_{t,r})_{t \in [0,1]}.$ Then, after possibly diminishing $\epsilon,$ for each $r \in [0,\epsilon),$ there exists a unique function $u_r \in C^{\infty}(L;\partial L)$ such that $\operatorname{Graph}(du) \subset V$ and
\[
\varphi_r \circ j_{u_r} : L \to \C^n
\]
represents the imaginary special Lagrangian cylinder $Z_r.$
\end{lemma}
\begin{proof}
We use Notation~\ref{ntn: Weinstein}. After possibly shrinking $V,$ we may assume that each point of $L$ has a neighborhood $U$ such that $\psi|_{V \cap \pi_L^{-1}(U)}$ is a diffeomorphism onto its image.

Let $(f_r : L \to \C^n)_{r \in [0,\epsilon)}$ be a $C^{1,\alpha}$ continuous family of smooth imaginary special Lagrangian immersions representing $Z_r.$ After possibly shrinking $\epsilon,$ we claim that
\[
\varphi_r^{-1}\circ f_r(L) \subset \psi(V).
\]
Here, the set $\psi(V)$ is not open in $\C^n$ because $V$ is a manifold with boundary. Indeed, $V$ is an open neighborhood of the zero section of $T^*L$ and $L$ is a manifold with boundary. Nonetheless, for $i = 0,1,$ we have $\varphi_r^{-1}\circ f_r(S^{n-1}\times \{i\}) \subset \Lambda_{i,0}.$ Moreover, $\varphi_r^{-1} \circ f_r$ is close in the $C^{1}$ topology to $f_0$ and $f_0(L) \subset \psi(V)$ and is not tangent to the boundary. So, the claim follows.

For $U \subset L,$ abbreviate $\widetilde U = V \cap \pi_L^{-1}(U).$ Since each point of $L$ has a neighborhood $U$ such that $\psi|_{\widetilde U}$ is a diffeomorphism onto its image, and $L$ is compact, we can choose $\{U_j\}_{j = 1}^N$ an open cover of $L$ such that $\psi|_{\widetilde U_j}$ is a diffeomorphism onto its image. Let $\{T_j\}_j$ be an open cover of $L$ such that $\overline T_j \subset U_j.$ After possibly shrinking $\epsilon,$ we may assume that $f_r(\overline T_j) \subset \psi(\widetilde U_j).$
For $r \in [0,\epsilon),$ define $\kappa_r : L \to L$ by
\[
\kappa_r|_{T_j} = \pi_L \circ (\psi|_{\widetilde U_j})^{-1} \circ \varphi_r^{-1} \circ f_r|_{T_j}.
\]
Since $\kappa_0 = \id_L,$ after possibly shrinking $\epsilon,$ we may assume that $\kappa_r$ is a diffeomorphism for $r \in [0,\epsilon).$ By Lemma~\ref{lemma:closed and annihilating is exact}, there exists a unique $u_r \in \cob{\infty}(L)$ such that
\[
j_{u_r} = \varphi_r^{-1} \circ f_r \circ \kappa_r^{-1}.
\]

We claim that for $r \in [0,\epsilon)$ we have $u_r \in C^\infty(L;\partial L).$ Indeed, let $f_{s,r}: L \to \C^n$ be given by
\[
f_{s,r}(p) := s f_r(p), \qquad s \in [0,1],\quad r \in [0,\epsilon),\quad p \in L.
\]
Since for $r \in [0,\epsilon)$ the Hamiltonian of the geodesic $(\Lambda_{t,r})_t$ is $2$-homogeneous, it follows that for $s \in (0,1],$
\[
Z_{s,r} := [f_{s,r}] \in\mathcal{SLC}\left(S^{n-1};\Lambda_{0,r},\Lambda_{1,r}\right)
\]
is the cylinder of $-s^2$ level sets of the geodesic $(\Lambda_{t,r})_t.$ By Lemma~\ref{lemma: flux as Hamiltonian}, we have
\begin{equation}\label{eq:flux-1}
\operatorname{RelFlux}\left((Z_{s,r})_{s \in (0,1]}\right) = -1.
\end{equation}
Consider the two-parameter family of Lagrangian cylinders,
\[
Z'_{s,r} := [\varphi_r^{-1}\circ f_{s,r}] \in \mathcal{LC}\left(S^{n-1};\Lambda_{0,0},\Lambda_{1,0}\right), \qquad (s,r) \in (0,1]\times [0,\epsilon).
\]
By Remark~\ref{rem: flux} \eqref{it: flux symplectic invariant} and equation~\eqref{eq:flux-1}, we have
\[
\operatorname{RelFlux}\left((Z'_{s,r})_{s \in (0,1]}\right) = -1, \qquad r \in [0,\epsilon).
\]
Since $f_{0,r} = 0$ for $r \in [0,\epsilon),$ it follows from Remark~\ref{rem: flux} \eqref{it: flux homotopy invariant} that
\[
\operatorname{RelFlux}\left((Z'_{1,r})_{r \in [0,r_1]}\right) = 0, \qquad r_1 \in [0,\epsilon).
\]
Let $A_r \in \R$ be the unique constant such that $\deriv[u_r]{r} |_{S^{n-1}\times\{1\}} \equiv A_r.$ By Definition~\ref{dfn: flux}, we have
\[
\operatorname{RelFlux}\left((Z'_{1,r})_{r \in [0,r_1]}\right) =  \int_0^{r_1}A_r.
\]
Combining the preceding two equations, we obtain
\[
\int_0^{r_1} A_r = 0, \qquad r_1 \in [0,\epsilon).
\]
Differentiating with respect to $r_1,$ it follows from the fundamental theorem of calculus that $A_r = 0$ for all $r \in [0,\epsilon).$ Since $u_0 = 0,$ we obtain $u_r \in C^{\infty}(L;\partial L)$ as desired.
\end{proof}
	
The following lemma is a simple variant of~\cite[Theorem~1.6]{cylinders}.
	
	\begin{lemma}
		\label{lemma:perturbed geodesic of cones}
In Setting~\ref{set: one}, there exists an $\epsilon>0$ such that the geodesic $(\Lambda_{t,0})_t$ extends to a smooth family $(\Lambda_{t,r})_{(t,r)\in[0,1]\times[0,\epsilon)}$ such that for $r\in[0,\epsilon),$ the path $(\Lambda_{t,r})_{t\in[0,1]}$ is a geodesic in $\lagcones^+(n)$ with negative derivative between $\Lambda_{0,r}$ and $\Lambda_{1,r}.$ Moreover, if $(\Lambda'_{t,r})_{(t,r) \in [0,1]\times [0,\epsilon)}$ is another such family of geodesics, but with the dependence on $r$ being a priori only continuous with respect to the $C^{1,\alpha}$ topology on $\mathfrak{G}_{\laggrass^+(n)},$ there exists $\epsilon' > 0$ such that $\Lambda'_{t,r} = \Lambda_{t,r}$ for $(t,r) \in [0,1]\times [0,\epsilon').$
	\end{lemma}
	
	\begin{proof}
Recall Notation~\ref{ntn: Weinstein}. Let $0\in\mathcal{W}\subset C^{2,\alpha}(L;\partial L),$ be an open neighborhood such that for $u \in \mathcal{W}$ we have $\operatorname{Graph}(du) \subset V.$ Define the differential operator
		\[
		\mathcal{F}:\mathcal{W}\times[0,1]\to C^\alpha(L),\qquad (u,r)\mapsto*j_u^*\varphi_r^*\real\Omega.
		\]
		Then $\mathcal{F}$ is smooth with $\mathcal{F}(0,0)=0$ and $\mathcal{F}(u,r) = 0$ if and only if the immersion $\varphi_r \circ j_u : L \to \C^n$ is imaginary special Lagrangian. By Lemma~\ref{lemma:Laplacian}, the linearization of $\mathcal{F}$ at $(0,0)$ in the directions of $\mathcal{W}$ is equal to the Riemannian Laplacian, which is an isomorphism $C^{2,\alpha}(L;\partial L)\to C^\alpha(L).$ By the implicit function theorem, shrinking $\mathcal{W}$ if necessary, for some $\epsilon>0,$ there exists a smooth map $k:[0,\epsilon)\to\mathcal{W}$ such that, for $r\in[0,\epsilon),$ the function $k(r)$ is the unique element in $\mathcal{W}$ satisfying $\mathcal{F}(k(r),r)=0.$ By elliptic regularity (e.g.\ \cite[Chapter 17]{gilbarg-trudinger}), all the functions $k(r)$ are smooth. For $r\in[0,\epsilon),$ we have $Z_r:=[\varphi_r \circ j_{k(r)} : L \to \C^n]\in\mathcal{SLC}\left(S^{n-1};\Lambda_{0,r},\Lambda_{1,r}\right).$ Diminishing $\epsilon$ if necessary, the cylinders $(Z_r)_{r \in [0,\epsilon)}$ are nowhere tangent to the Euler vector field. For $r\in[0,\epsilon),$ let $(\Lambda_{t,r})_{t\in[0,1]}$ denote the geodesic associated to the cylinder $Z_r$ by Lemma~\ref{lemma:geodesic of cones}. Then $(\Lambda_{t,r})_{(t,r)\in[0,1]\times[0,\epsilon)}$ is the desired family.

To prove the uniqueness claim, let $u_r \in C^{\infty}(L;\partial L), r \in [0,\epsilon'),$ be the family of functions associated $\Lambda'_{t,r}$ as in Lemma~\ref{lemma: Dirichlet}. Then $\mathcal{F}(u_r,r) = 0,$ so by the uniqueness of $k(r),$ we have $u_r = k(r).$
	\end{proof}
	
	\begin{prop}
		\label{proposition:geodesic is in fact linear}
		Let $(\Lambda_{0,r})_{r\in[0,1]}$ and $(\Lambda_{1,r})_{r\in[0,1]}$ be smooth paths in $\laggrass^+(n)$ with $\Lambda_{0,r}\cap\Lambda_{1,r}=\{0\},\;r\in[0,1].$ Let $(\Lambda_{t,r})_{(t,r)\in[0,1]\times[0,1]}$ be a family in $\lagcones^+(n)$ such that for $r\in[0,1]$ the path $(\Lambda_{t,r})_{t\in[0,1]}$ is a geodesic with negative derivative between $\Lambda_{0,r}$ and $\Lambda_{1,r}$ depending continuously on $r$ with respect to the $C^{1,\alpha}$ topology on $\mathfrak{G}_{\laggrass^+(n)}.$ Suppose the initial geodesic $(\Lambda_{t,0})_t$ lies in $\laggrass^+(n).$ Then the entire given family lies in $\laggrass^+(n).$
	\end{prop}
	
	\begin{proof}
		Let $A\subset[0,1]$ denote the set consisting of values of $r$ such that the geodesic $(\Lambda_{t,r})_t$ lies in $\laggrass^+(n).$ By assumption, $A$ is non-empty. As $\laggrass^+(n)\subset\lagcones^+(n)$ is a closed subset, $A$ is closed in $[0,1].$ By Lemma~\ref{lemma:negative derivative implies Maslov zero}~\ref{item:Maslov} we have $m(\Lambda_{0,0},\Lambda_{1,0})=0,$ and it follows that $m(\Lambda_{0,r},\Lambda_{1,r})=0$ for $r\in[0,1].$ By Proposition~\ref{proposition:space of good Lags is open} and the uniqueness part of Lemma~\ref{lemma:perturbed geodesic of cones}, the set $A$ is open in $[0,1].$
	\end{proof}

\section{Blowing up at a critical point of a geodesic}\label{sec:blowup}
Throughout this section, let $(X,\omega,J,\Omega)$ denote a Calabi-Yau manifold as in Definition~\ref{definition:Calabi-Yau}. We give a blowup procedure that is used to establish $C^{1,1}$ regularity of geodesics of positive Lagrangians in $X$.

\subsection{Geodesic of tangent cones}
	The following lemma, needed for the proof of Theorem~\ref{theorem:perturbed C1 geodesic}, is the original observation behind this article. We make use of Notation~\ref{ntn:critical point}.
	
	\begin{lemma}
		\label{lemma:geodesic of tangent cones}
		Let $(\Lambda_t)_{t\in[0,1]}$ be a geodesic of positive Lagrangians in $X$ with derivative $(h_t)_t.$ Suppose the endpoints $\Lambda_0$ and $\Lambda_1$ are smoothly embedded and let $q\in\crit((\Lambda_t)_t).$ Then, the path $(TC_{\hat q_t}\Lambda_t)_{t\in[0,1]}$ is a geodesic of positive Lagrangian cones in $T_qX$ with derivative given by
		\[
		\left(\deriv{t}TC_{\hat q_t}\Lambda_t\right)(v)=\frac{1}{2}\nabla_vdh_t(v),\quad t\in[0,1],\;v\in TC_{\hat q_t}\Lambda_t.
		\]
Moreover, if $(\Psi_t : L \to X,S)_t$ is a horizontal lifting of $(\Lambda_t)_t,$ and $p \in L$ with $\Psi_t(p) = q,$ then the family of cone derivatives
\[
(d \Psi_t)_p : T_p L \to T_qX, \qquad t \in [0,1],
\]
is a horizontal lifting of $(TC_{\hat q_t}\Lambda_t)_{t\in[0,1]}.$
	\end{lemma}

\begin{dfn}
In the situation of Lemma~\ref{lemma:geodesic of tangent cones}, we call $(TC_{\hat q_t}\Lambda_t)_{t\in[0,1]}$ the \emph{geodesic of tangent cones} of $(\Lambda_t)_t$ at $q.$
\end{dfn}

We now introduce the notion of rescaling data, which is used in the proof of Lemma~\ref{lemma:geodesic of tangent cones} and will appear in other proofs as well.
\begin{dfn}\label{dfn:rescaling data}
Let $\Lambda_0,\Lambda_1 \subset X$ be smoothly embedded Lagrangian submanifolds intersecting at a point $q.$ \emph{Rescaling data} for $(\Lambda_0,\Lambda_1,q)$ consists of the following.
\begin{itemize}
\item
A neighborhood $q \in U \subset X.$
\item
A ball centered at zero $V \subset \C^n.$
\item
A symplectomorphism $\mathbf{X} : V \to U$ such that $\mathbf{X}(0) =q$ and $\mathbf{X}^{-1}(\Lambda_0)$ and $\mathbf{X}^{-1}(\Lambda_1)$ are contained in Lagrangian linear subspaces of $\C^n.$
\end{itemize}
For $s \geq 0,$ let $M_s^*: \C^n \to \C^n$ denote multiplication by $s,$ and abbreviate $V_s : = M_s^{-1}(V).$ In particular, $V_0 = \C^n.$ Rescaling data determines the following additional data.
\begin{itemize}
\item
Lagrangian linear subspaces $\widehat{\Lambda}_0,\widehat{\Lambda}_1\subset \C^n$ containing $\mathbf{X}^{-1}(\Lambda_0)$ and $\mathbf{X}^{-1}(\Lambda_1)$ respectively.
\item
A family of complex structures $J_s$ on $V_s$ given by $J_s := M_s^*\mathbf{X}^*J$ for $s > 0$ and $J_0 = \lim_{s \to 0} J_s$ in the sense of $C^\infty$ convergence on compact subsets.
\item
A family of $n$-forms $\Omega_s$ on $V_s$ given by $\Omega_s : = M_s^*\mathbf{X}^*\Omega$ for $s > 0$ and $\Omega_0 = \lim_{s \to 0} \Omega_s$ in sense of $C^\infty$ convergence on compact subsets.
\end{itemize}
In fact, $J_0$ is obtained by extending the linear complex structure $(\mathbf{X^*}J)_0$ on $T_0V_0$ to a constant coefficient complex structure on $V_0.$ Similarly, $\Omega_0$ is obtained by extending the alternating multi-linear form $(\mathbf{X}^*\Omega)_0$ on $T_0V_0$ to a constant coefficient differential form on $V_0.$
\end{dfn}

	The proof of Lemma~\ref{lemma:geodesic of tangent cones} relies on the following elementary observation (compare with~\cite[Lemma~2.1]{milnor}).
	
	\begin{lemma}
		\label{lemma: milnor lemma}
		Let $f:\mathbb{R}^k\times\mathbb{R}\to\mathbb{R}$ be smooth with $f(x,0)=0,\;x\in\mathbb{R}^k.$ Then we have $f(x,s)=s\cdot g(x,s)$ for $(x,s)\in\mathbb{R}^k\times\mathbb{R},$ where $g:\mathbb{R}^k\times\mathbb{R}\to\mathbb{R}$ is smooth and satisfies
		\[
		g(x,0)=\pderiv[f]{s}(x,0),\qquad \pderiv[g]{s}(x,0)=\frac{1}{2}\frac{\partial^2f}{\partial s^2}(x,0),\qquad x\in\mathbb{R}^k.
		\]
	\end{lemma}

	\begin{proof}[Proof of Lemma~\ref{lemma:geodesic of tangent cones}]
Let $U,V,\mathbf{X}$ be rescaling data for $(\Lambda_0,\Lambda_1,q)$ and let $\Omega_s$ be the associated family of $n$-forms as in Definition~\ref{dfn:rescaling data}. Let $M_s : \C^n \to \C^n$ denote multiplication by $s.$

Choose a horizontal lifting $(\Psi_t : L \to X,S)_t$ of the geodesic $(\Lambda_t)_t.$  Let $p \in L$ satisfy $\Psi_t(p) = q.$ Let $p \in W \subset L$ be a neighborhood such that
\[
\Psi_t(W) \subset U, \qquad t \in [0,1].
\]
After perhaps shrinking $W,$ identify $W$ with a ball $B \subset \R^n$ centered at zero by a diffeomorphism $\mathbf{Y} : B \to W$ such that $\mathbf{Y}(0) = p.$ For $s \geq 0,$ let $\mu_s : \R^n \to \R^n$ denote multiplication by $s$ and write \[
B_s : = \mu_s^{-1}(B).
\]
For $s > 0,$ define a family of Lagrangian cone-immersions
\[
\chi_{s,t} : B_s \to V_s
\]
by
\[
\chi_{s,t} = M_s^{-1}\circ \mathbf{X}^{-1} \circ \Psi_t \circ \mathbf{Y}\circ \mu_s.
\]
By Lemma~\ref{lemma: milnor lemma}, since
\[
\mathbf{X}^{-1} \circ \Psi_t \circ \mathbf{Y} \circ \mu_0(x) = 0, \qquad x \in B_0,
\]
it follows that $\chi_{s,t}$ extends smoothly to $s = 0$ with
\[
\chi_{0,t}(x) = \left.\pderiv{s}\mathbf{X}^{-1} \circ \Psi_t \circ \mathbf{Y} \circ \mu_s(x)\right|_{s = 0} = \left.\pderiv{s}\mathbf{X}^{-1} \circ \Psi_t \circ \mathbf{Y}(sx)\right|_{s = 0}.
\]
Thus, identifying $\C^n,\R^n,$ with their respective tangent spaces at zero, we have
\begin{equation}\label{eq:chiconederiv}
d\mathbf{X}_0(\chi_{0,t}(x)) = (d\Psi_t)_p(d\mathbf{Y}_0(x)).
\end{equation}
In particular,
\[
d\mathbf{X}_0 \circ \chi_{0,t} : \R^n \to TC_{\hat q_t}\Lambda_t
\]
is a cone-smooth parameterization.

We proceed to compute the Hamiltonian $k_{s,t}$ of the family of Lagrangian cone-immersions $\chi_{s,t}.$ Without loss of generality, we may assume that $h_t(\hat q_t) = 0.$ Since $\mathbf{X}$ is a symplectomorphism, and $M_s^{-1}$ rescales the symplectic form by a factor of $s^{-2},$ it follows that $k_{s,t}$ is given by
\[
k_{s,t} \circ \chi_{s,t} =  s^{-2} h_t \circ \Psi_t \circ \mathbf{Y} \circ \mu_s, \qquad s > 0.
\]
Remark~\ref{rem:critical point} asserts that $\hat q_t = [(\Psi_t,p)]$ is a critical point of $h_t,$ so Lemma~\ref{lemma: milnor lemma} implies that $k_{s,t} \circ \chi_{s,t}$ extends smoothly to $s = 0.$ Moreover,
\[
k_{0,t} \circ \chi_{0,t}(x) = \left.\frac{1}{2}\frac{\partial^2}{\partial s^2} h_t\circ \Psi_t\circ \mathbf{Y}(sx)\right|_{s = 0} = \frac{1}{2} \nabla_v dh_t(v),
\]
where
\[
v = d(\Psi_t \circ \mathbf{Y})_0(x) \in TC_{\hat q_t}\Lambda_t.
\]
Since $(\Psi_t,S)_t$ is a horizontal lifting of $(\Lambda_t)_t$, it follows that $(\chi_{s,t},0)_t$ is a horizontal lifting of the path of cone-immersed Lagrangians $([(\chi_{s,t},0)])_t$ with respect to the $n$-form $\Omega_s$ for $s > 0.$ By continuity, $(\chi_{0,t},0)_t$ is a horizontal lifting of the path of cone-immersed Lagrangians $([(\chi_{0,t},0)])_t$ with respect to the $n$-form $\Omega_0.$ Since $\Lambda_t$ is a geodesic, $(k_{s,t})_t$ is constant in $t$ and thus $([(\chi_{s,t},0)])_t$ is a geodesic for $s \geq 0$. Since
$(\mathbf{X}^*\Omega)_0 = (\Omega_0)_0,$
it follows that $d\mathbf{X}_0 \circ \chi_{0,t}$ is a horizontal lifting of the path of Lagrangian cones $(TC_{\hat q_t} \Lambda_t)_t.$ Moreover, $(TC_{\hat q_t} \Lambda_t)_t$ is a geodesic with Hamiltonian as claimed. Finally, it follows from equation~\eqref{eq:chiconederiv} that $(d\Psi_t)_p : T_pL \to T_qX,\, t \in [0,1],$ is a horizontal lifting of $(TC_{\hat q_t}\Lambda_t)_{t\in[0,1]}.$
	\end{proof}

\subsection{Cylindrical transform commutes with blowup}
\begin{lemma}\label{lemma: blowup cylinder}
Let $\Lambda_0,\Lambda_1,$ be smoothly embedded positive Lagrangians in $X$ that intersect transversally at a point $q$ of Maslov index $0.$ Let $\mathcal{Z} \subset \mathcal{SLC}(S^{n-1};\Lambda_0,\Lambda_1)$ be a family of imaginary special Lagrangian cylinders converging regularly to $q.$ Let $\Phi : S^{n-1}\times [0,1]\times [0,1) \to X$ be a regular parameterization of $\mathcal{Z}.$  Then,
\[
\left.\frac{\partial \Phi}{\partial s}\right|_{s = 0}: S^{n-1}\times [0,1] \to T_qX
\]
is an imaginary special Lagrangian immersion with respect to the induced Calabi-Yau structure on $T_qX.$
\end{lemma}
\begin{proof}
Let $U,V,\mathbf{X}$ be rescaling data for $(\Lambda_0,\Lambda_1,q)$ and let $\Omega_s$ be the associated family of $n$-forms as in Definition~\ref{dfn:rescaling data}. Let $M_s : \C^n \to \C^n$ denote multiplication by~$s.$

Choose $\delta > 0$ small enough that $\Phi(p,t,s) \in U$ for all $(p,t) \in S^{n-1}\times[0,1]$ and $s < \delta.$ By Lemma~\ref{lemma: milnor lemma}, since $\Phi(t,p,0) = q$ for all $(t,p) \in S^{n-1}\times[0,1],$ we have
		\[
		\mathbf{X}^{-1}\circ\Phi(p,t,s) = s \cdot \Psi(p,t,s), \quad(p,t,s) \in S^{n-1} \times[0,1] \times [0,\delta),
		\]
		where $\Psi : S^{n-1} \times [0,1] \times [0,\delta)\to \C^n$ is smooth with
		\[
		\Psi(p,t,0) = \pderiv[(\mathbf{X}^{-1}\circ\Phi)]{s}(p,t,0), \quad (p,t) \in S^{n-1} \times [0,1].
		\]
For $s \in [0,\delta),$ write
\[
\Psi_s := \Psi|_{S^{n-1}\times[0,1]\times\{s\}}.
\]
For $s \in (0,\delta),$ the map $\Psi_s$ is an immersion representing an $\Omega_{s}$-imaginary special Lagrangian cylinder. As $\Phi$ is regular, it follows from Definition~\ref{definition: interior regularity} \eqref{regular convergence to intersection point} that the map
$
\Psi_0
$
is an immersion. By continuity
$
\Psi_0
$
is an $\Omega_{0}$-imaginary special Lagrangian immersion. Thus,
\[
\left.\frac{\partial \Phi}{\partial s}\right|_{s = 0} = d\mathbf{X}_0\circ \left.\pderiv[(\mathbf{X}^{-1}\circ\Phi)]{s}\right|_{s = 0} = d\mathbf{X}_0\circ \Psi_0
\]
is an imaginary special Lagrangian immersion with respect to the induced Calabi-Yau structure on $T_qX.$
\end{proof}
\begin{dfn}\label{dfn:tangent family}
In the setting of Lemma~\ref{lemma: blowup cylinder}, let $\mathcal{Z}^q_\Phi \in \mathcal{SLC}(S^{n-1};T_q \Lambda_0,T_q\Lambda_1)$ be given by
\[
\mathcal{Z}^q_\Phi : = \left[\left.\frac{\partial \Phi}{\partial s}\right|_{s = 0}: S^{n-1}\times [0,1] \to T_qX \right],
\]
and let $\mathcal{Z}_q$ denote the $\R_{>0}$ orbit of $\mathcal{Z}^q_\Phi$ in $\mathcal{SLC}(S^{n-1};T_q \Lambda_0,T_q\Lambda_1).$ We call $\mathcal{Z}_q$ the \emph{tangent family} of $\mathcal{Z}$ at $q.$ One verifies that $\mathcal{Z}_q$ is independent of $\Phi.$ By Definition~\ref{definition: interior regularity}~\eqref{regular convergence to intersection point} the cylinders in $\mathcal{Z}_q$ are nowhere tangent to the Euler vector field.
\end{dfn}
In the following, we use Notation~\ref{ntn:critical point}.
\begin{lemma}\label{lemma:commute}
		Let $(\Lambda_t)_{t\in[0,1]}$ be a geodesic of positive Lagrangians in $X$ with derivative $(h_t)_t.$ Suppose the endpoints $\Lambda_0$ and $\Lambda_1$ are smoothly embedded and let $q\in\crit((\Lambda_t)_t)$ be a transverse intersection point of $\Lambda_0$ and $\Lambda_1.$ Moreover, for $t \in [0,1]$ assume $\hat q_t$ is an absolute minimum or maximum of $h_t.$ Let $\mathcal{Z}$ denote the cylindrical transform of $(\Lambda_t)_t,$ which converges regularly to $q$ by Lemma~\ref{lemma:rocking lemma}. Then, the tangent family $\mathcal{Z}_q \subset \mathcal{SLC}(S^{n-1};T_q \Lambda_0,T_q\Lambda_1)$ coincides with the cylindrical transform of the geodesic of tangent cones $(TC_{\hat q_t}\Lambda_t)_t.$
\end{lemma}
\begin{proof}
By definition, the tangent family $\mathcal{Z}_q$ is an $\R_{>0}$ orbit in the space of imaginary special Lagrangian cylinders $\mathcal{SLC}(S^{n-1};T_q \Lambda_0,T_q\Lambda_1).$ By Lemma~\ref{lemma:geodesic of cones}, so is the cylindrical transform of the geodesic of tangent cones $(TC_{\hat q_t}\Lambda_t)_t.$ Thus, to prove the lemma, it suffices to show that the two orbits share a common point.

Recall Setting~\ref{set:rocking} and let $\sigma: \orientedprojective(T_pL) \cong S^{n-1} \to T_pL\setminus\{0\}$ be the section associated with $\kappa.$ By Lemma~\ref{lemma:rocking lemma} the map
\[
		\Phi:S^{n-1}\times[0,1]\times[0,\epsilon)\to X,\qquad(c,t,s)\mapsto\Psi_t(\kappa(c,s)),
\]
is a regular parameterization of $\mathcal{Z}$ about $q.$ We calculate
\begin{equation}\label{equation:dPhis}
\pderiv[\Phi]{s}(c,t,0) = (d\Psi_t)_p \left( \pderiv[\kappa]{s}(c,0)\right) = (d\Psi_t)_p \left( \sigma(c) \right).
\end{equation}
Lemma~\ref{lemma:geodesic of tangent cones} asserts that $((d\Psi_t)_p)_t$ is a horizontal lifting of the geodesic of tangent cones $(TC_{\hat q_t}\Lambda_t)_t.$
Let
\[
h^T_t = \deriv{t}TC_{\hat q_t}\Lambda_t
\]
denote the Hamiltonian of the geodesic of tangent cones and let $h^T = h^T_t \circ (d\Psi_t)_p$ denote the Hamiltonian with respect to the horizontal lifting $((d\Psi_t)_p)_t.$ By Lemma~\ref{lemma:geodesic of tangent cones} we have
\begin{equation}\label{eq:hT}
h^T(v) = \frac{1}{2}\nabla_{v}dh(v), \qquad v\in T_pL.
\end{equation}

Keeping in mind that $p$ is a critical point of $h,$ we have
\[
\nabla_{\sigma(c)}dh(\sigma(c)) = \frac{\partial^2(h \circ \kappa)}{\partial s^2}(c, 0).
\]
Thus, since $h\circ \kappa(c,s)$ is independent of $c,$ also $a: = \nabla_{\sigma(c)}dh(\sigma(c))$ is independent of~$c.$ So, it follows from equation~\eqref{eq:hT} that $\sigma$ parameterizes the $\frac{a}{2}$ level set of $h^T.$ By equation~\eqref{equation:dPhis}, the cylinder $\mathcal{Z}_\Phi^q$ from Definition~\ref{dfn:tangent family} coincides with the cylinder of $\frac{a}{2}$ level sets associated to $(TC_{\hat q_t}\Lambda_t)_t$ as desired.
\end{proof}

\subsection{\texorpdfstring{$C^{1,1}$}{C11} regularity from blowing up}
The following lemmas are used in the proof of Theorem~\ref{theorem:perturbed C1 geodesic} as well as Theorem~\ref{theorem:spider web}~\ref{geodesic of open C1 Lags}.
	\begin{lemma}
		\label{lemma:cone-smooth function is C1}
		Let $\Theta$ be a compact manifold with corners and let $((f_t:M\to N,S))_{t\in \Theta}$ be a cone-smooth family of maps. Let $p \in S$ such that the cone-derivative $(df_t)_p$ is linear for all $t.$ Then, $f_t : M \to N$ is $C^{1,1}$ in a neighborhood of $p$ uniformly in $t.$
	\end{lemma}

	\begin{proof}
		It suffices to prove the lemma when $N = \R.$ It follows from the definition of the cone derivative that $f_t$ is differentiable at $p$ and its ordinary derivative coincides with its cone derivative. After subtracting a smooth family of functions, we may assume that $(df_t)_p = 0.$

Recall Definition~\ref{definition: many tangent spaces}. We think of the blowup differential $\widetilde{df}_t$ as a section of the dual of the blowup tangent bundle $\widetilde{TM}_S^*.$ It follows from equation~\eqref{equation:dderivatives} that $(\widetilde{df_t})_{\widetilde p}$ vanishes for all $\widetilde p \in E_p.$ Choose an open neighborhood $E_p \subset U \subset \widetilde M_S$ and a diffeomorphism $\chi : S^{n-1} \times [0,\epsilon) \to U.$ Let $s : U \to \R$ be given by the composition of $\chi^{-1}$ and the projection to $[0,\epsilon).$ By Lemma~\ref{lemma: milnor lemma}, there exists a smooth family of sections $\xi_t$ of $\widetilde{TM}_S^*|_U$ such that $\widetilde{df_t} = s\xi_t.$ Let $d : M \times M \to \R$ denote the distance function of a smooth Riemannian metric on $M$ and let $\delta : \widetilde M_S \times \widetilde M_S \to \R$ denote the distance function of a smooth Riemannian metric on $\widetilde M_S.$ Then, there exists a constant $C > 0$ such that for
\[
\widetilde x,\widetilde y \in U, \qquad x = \pi(\widetilde x), \;y = \pi(\widetilde y),
\]
we have
\[
\min(s(\widetilde x),s(\widetilde y)) \delta(\widetilde x,\widetilde y) \leq C d(x,y), \qquad
|s(\widetilde x) - s(\widetilde y)| \leq C d(x,y).
\]
Choose a local trivialization of $T^*M$ near $p$ and pull it back to obtain a local trivialization of $\widetilde{TM}^*_S$ near $E_p.$ Working in these trivializations, and assuming without loss of generality that $\min(s(\widetilde x),s(\widetilde y)) = s(\widetilde x),$ we have
\begin{align*}
\frac{|(df_t)_x-(df_t)_y|}{d(x,y)} & =  \frac{\left|(\widetilde{df_t})_{\widetilde x} - (\widetilde{df_t})_{\widetilde y}\right|}{d(x,y)} \\
& = \frac{|s(\widetilde x)\xi_t(\widetilde x) - s(\widetilde y)\xi_t(\widetilde y)|}{d(x,y)}\\
&\leq \frac{|s(\widetilde x)\xi_t(\widetilde x) - s(\widetilde x)\xi_t(\widetilde y)| + |s(\widetilde x) - s(\widetilde y)||\xi_t(\widetilde y)|}{d(x,y)}  \\
&\leq \frac{C|s(\widetilde x)\xi_t(\widetilde x) - s(\widetilde x)\xi_t(\widetilde y)|}{\min(s(\widetilde x),s(\widetilde y))\delta(\widetilde x,\widetilde y)} + \frac{|s(\widetilde x) - s(\widetilde y)||\xi_t(\widetilde y)|}{d(x,y)} \\
&\leq C\frac{|\xi_t(\widetilde x) - \xi_t(\widetilde y)|}{\delta(\widetilde x, \widetilde y)} + C|\xi_t(\widetilde y)|,
\end{align*}
which is uniformly bounded because $\xi_t$ is a smooth family.
	\end{proof}

\begin{lemma}\label{lemma:C11 regularity from blowup}
		Let $(\Lambda_t)_{t\in[0,1]}$ be a geodesic of positive Lagrangians in $X$ with derivative $(h_t)_t.$ Suppose the endpoints $\Lambda_0$ and $\Lambda_1$ are smoothly embedded and let $q\in\crit((\Lambda_t)_t).$ If the tangent cones $TC_{\hat q_t}\Lambda_t$ are linear for $t \in [0,1],$ then the geodesic $(\Lambda_t)_t$ is of regularity $C^{1,1}$ in a neighborhood of $q.$
\end{lemma}

\begin{proof}

Let $(\Psi_{t} : S^n \to X,S)_{t \in [0,1]}$ be a horizontal lifting of the geodesic $\Lambda_{t}$ with $\Psi_{0}$ smooth and let $p \in S$ such that $\Psi_{t}(p) = q.$ By Lemma~\ref{lemma:geodesic of tangent cones}, the family of cone derivatives $(d\Psi_{t})_p : T_p S^n \to T_{q}X$ is a horizontal lifting of the geodesic of Lagrangian linear subspaces $TC_{\hat q_{t}}\Lambda_{t} \subset T_{q}X.$ Since $\Psi_{0}$ is smooth, the cone derivative $(d\Psi_{0})_p$ is linear, so by Lemma~\ref{lemma:there is always a horizontal lifting} the cone derivative $(d\Psi_{t})_p$ is linear for $t \in [0,1].$ Thus, by Lemma~\ref{lemma:cone-smooth function is C1}, the cone immersions $(\Psi_{t} : S^n \to X,S)$ are $C^{1,1}$ in a neighborhood of $p$ uniformly in $t.$
Furthermore, the time derivative $\pderiv[\Psi_{t}]{t}$ exists by definition of a cone-smooth family.

Next, we claim that the differential $(d\Psi_{t})_x$ is $C^{0,1}$ as a function of $t$ uniformly in $x\in S^n$ in a neighborhood of $p.$ Indeed, the blowup derivative $(\widetilde{d\Psi_{t}})_{\tilde x}$ depends smoothly on $t$ and $\tilde x.$ So, it is $C^{0,1}$ in $t$ uniformly in $\tilde x.$ On the other hand, for $\tilde x$ such that $\pi(\tilde x) = x,$ we have
\[
(\widetilde{d\Psi_{t}})_{\tilde x} = (d\Psi_{t})_x.
\]

Finally, we show the partial derivative $\pderiv[\Psi_{t}]{t}$ is $C^{0,1}$ in a neighborhood of $p.$ Indeed, since $\hat q_{t}$ is a critical point of $h_{t},$ Lemma~\ref{lemma:cone-smooth function is C1} implies that $h_{t}$ is $C^{1,1}.$ So, $dh_{t}$ is $C^{0,1}.$ Write
\[
h = h_{t}\circ \Psi_{t},
\]
for the Hamiltonian associated to the horizontal lifting $(\Psi_{t})_t,$ which is independent of $t$ because $(\Lambda_{t})_t$ is a geodesic. The Hamiltonian $h$ is $C^{1,1}$ as the composition of $C^{1,1}$ maps. Recall that
\[
\theta_{\Lambda_{t}} : \Lambda_{t} \to \left(-\frac{\pi}{2},\frac{\pi}{2}\right)
\]
denotes the Lagrangian angle, and write
\[
\tilde \theta_{\Lambda_{t}} = \theta_{\Lambda_{t}} \circ \Psi_{t}: L \to \left(-\frac{\pi}{2},\frac{\pi}{2}\right).
\]
Write $\nabla^t$ for the gradient with respect to the pull-back metric $\Psi_{t}^*g.$ Let
\[
J^t : TL \to \Psi_{t}^*TX
\]
denote the bundle map given by
\[
J^t(\xi) = J \circ d\Psi_{t}(\xi).
\]
Since $d\Psi_{t}$ is $C^{0,1}$ jointly in space and time, it follows that the families $\Psi_{t}^*g, \tilde\theta_{\Lambda_{t}}$ and $J^t,$ are $C^{0,1}$ jointly in space and time. By~\cite[Remark 5.6]{solomon} we have for $x\in S^n,$
\begin{align*}
\pderiv[\Psi_{t}]{t}(x)&=-J\nabla h_{t}(\Psi_{t}(x))-\tan\theta_{\Lambda_{t}}(\Psi_{t}(x))\nabla h_{t}(\Psi_{t}(x))\\
&=-J^t\nabla^t h(x)-\tan\tilde \theta_{\Lambda_{t}}(x)\nabla^t h(x),
\end{align*}
so $\pderiv[\Psi_{t}]{t}(x)$ is $C^{0,1}$ as desired.

It follows that $\Psi_{t}$ is $C^{1,1}$ as a map $S^n \times [0,1] \to X$ and the geodesic $(\Lambda_{t})_t$ is $C^{1,1}$ as claimed.
\end{proof}

	\begin{proof}[Proof of Theorem~\ref{theorem:perturbed C1 geodesic}]

Theorem~\ref{theorem:perturbed C1 geodesic} is the same as Theorem~1.6 of~\cite{cylinders} with the exception of the assumption that $(\Lambda_t)_t$ is $C^1$ and the claim that for $\Lambda \in \mathcal{Y}$ the geodesic from $\Lambda_0$ to $\Lambda$ is of regularity $C^{1,1}.$ Possibly shrinking $\mathcal{Y},$ we may assume that all $\Lambda \in \mathcal{Y}$ intersect $\Lambda_0$ transversally at exactly two points. After possibly replacing $\mathcal{Y}$ with its connected component containing $\Lambda_1,$ we may assume that $\mathcal{Y}$ is connected. For $\Lambda \in \mathcal{Y},$ consider a smooth path $(\Lambda_{1,r})_{r \in [0,1]}$ with $\Lambda_{1,0} = \Lambda_1$ and $\Lambda_{1,1} = \Lambda.$ For $r \in [0,1]$ let $(\Lambda_{t,r})_{t \in [0,1]}$ be the unique geodesic in $\mathcal{X}$ from $\Lambda_0$ to $\Lambda_{1,r}.$ We prove that $(\Lambda_{t,r})_{t \in [0,1]}$ is of regularity $C^{1,1}.$

Indeed, for $(t,r)\in[0,1]\times[0,1],$ write
		\[
		h_{t,r}:=\deriv{t}\Lambda_{t,r}.
		\]
Recalling Remark~\ref{rem:critical point} and Notation~\ref{ntn:critical point}, let $\hat q_{r,t}$ be the point of $\Lambda_{t,r}$ where $h_{t,r}$ attains its maximum and let
\[
q_r := \im \hat q_{r,t} \in \crit((\Lambda_{t,r})_t).
\]
By Lemma~\ref{lemma: critical points are non-degenerate}, the point $\hat q_{r,t}$ is a non-degenerate critical point of $h_{t,r}.$ Thus, for $0 \neq v \in TC_{\hat q_{r,t}}\Lambda_{t,r}$ we have
\begin{equation}\label{eq:negative hessian}
\nabla_vdh_{t,r}(v) < 0.
\end{equation}
Let
\[
\left(\varphi_r:\mathbb{C}^n\to T_{q_r}X\right)_{r\in[0,1]}
\]
be a smooth family of symplectic complex-linear isomorphisms such that for $r\in[0,1]$ we have
\[
\varphi_r^*\Omega=\rho(q_r)\Omega_0,
\]
where $\Omega_0$ denotes the standard holomorphic $n$-form on $\C^n.$ For $(t,r)\in[0,1]\times[0,1],$ write
		\[
		C_{t,r}:=\varphi_r^{-1}(TC_{\hat q_{r,t}}\Lambda_{t,r}).
		\]
		By Lemma~\ref{lemma:geodesic of tangent cones} and inequality~\eqref{eq:negative hessian}, for $r \in [0,1]$ the path $(C_{t,r})_{t\in[0,1]}$ is a geodesic in $\lagcones^+(n)$ with negative derivative. By definition of the strong $C^{1,\alpha}$ topology on $\mathfrak{G}_\mathcal{O},$ the geodesic $(C_{t,r})_{t\in[0,1]}$ depends continuously on $r$ with respect to the $C^{1,\alpha}$ topology on $\mathfrak{G}_{\laggrass^+(n)}.$ By assumption, the geodesic $(C_{t,0})_t$ is linear. By Proposition~\ref{proposition:geodesic is in fact linear}, all the cones $C_{t,r}$ are in fact linear subspaces, and so are the tangent cones $TC_{\hat q_{r,t}}\Lambda_{t,r}.$ Similarly, when $\hat q_{r,t}$ is the point of $\Lambda_{t,r}$ where $h_{t,r}$ attains its minimum, the tangent cones $TC_{\hat q_{r,t}}\Lambda_{t,r}$ are linear. So, by Lemma~\ref{lemma:C11 regularity from blowup} the geodesic $\Lambda_{t,r}$ is of regularity $C^{1,1}.$
	\end{proof}

	\section{Geodesics of small open Lagrangians}
	\label{section:geodesics of small open Lagrangians}

In this section we show how to integrate geodesics of positive Lagrangian cones to construct geodesics of small open positive Lagrangians. The section culminates with the proof of Theorem~\ref{theorem:spider web}. In the following we use Notation~\ref{ntn:critical point}.

\begin{lemma}\label{lemma:integration of tangent cone}
		Let $(X,\omega,J,\Omega)$ be Calabi-Yau, and let $\Lambda_0,\Lambda_1\subset X$ be smoothly embedded positive Lagrangians intersecting transversally at a point $q$. Let $(C_t)_{t\in[0,1]}$ be a geodesic of positive Lagrangian cones with positive/negative derivative in $T_qX$ such that $C_i = T_q \Lambda_i$ for $i = 0,1.$
\begin{enumerate}[label=(\alph*)]
			\item\label{general regular family of cylinders} There exists a one-parameter family $(Z_s)_{s\in(0,\epsilon)}\subset\mathcal{SLC}\left(S^{n-1};\Lambda_0,\Lambda_1\right)$ converging regularly to $q$ with tangent family the cylindrical transform of $(C_t)_t.$ This family is unique up to reparameterization.
			\item\label{general geodesic of open C1 Lags} There exist open neighborhoods, $q\in U_i\subset\Lambda_i,\;i=0,1,$ which are connected by a geodesic $(U_t)_{t\in[0,1]}$ of open positive Lagrangians with tangent cones $TC_{\hat q_t} U_t = C_t.$ Given $U_i, i = 0,1,$ such a geodesic is unique up to reparameterization.
		\end{enumerate}
\end{lemma}

	\begin{proof}
Let $U,V,\mathbf{X}$ be rescaling data for $(\Lambda_0,\Lambda_1,q)$ and let $\widehat{\Lambda}_0,\widehat{\Lambda}_1,J_s,\Omega_s$ be the associated additional data as in Definition~\ref{dfn:rescaling data}. Let $M_s : \C^n \to \C^n$ denote multiplication by $s$ and let $V_s = M_s^{-1}(V).$

Let $\omega_0$ denote the standard symplectic form on $V_0 = \C^n.$ After possibly rescaling $\Omega_0$ by a positive constant, the quadruple $(V_0,\omega_0,J_0,\Omega_0)$ is a Calabi-Yau manifold isomorphic to $\mathbb{C}^n$ with the standard structure. Let $g_0$ denote the corresponding K\"ahler metric. The derivative of $\mathbf{X}$ gives an isomorphism $d\mathbf{X} : V_0 \simeq T_0V_0 \to T_qX$ respecting the Calabi-Yau structure $(\omega_0,J_0,\Omega_0)$ and the induced Calabi-Yau structure on the tangent space $T_qX.$ So,
\[
\widehat \Lambda_t : = (d\mathbf{X})^{-1}(C_t),\qquad t \in [0,1],
\]
is an $\Omega_0$-geodesic of positive Lagrangian cones from $\widehat\Lambda_0$ to $\widehat\Lambda_1$ with positive/negative definite derivative, which we denote by $\left(\widehat h_t\right)_t.$ Abbreviate $L = S^{n-1}\times [0,1].$ Let
\[
\widehat Z_0 = [f: L \to V] \in \mathcal{SLC}(S^{n-1};\widehat \Lambda_0,\widehat\Lambda_1)
\]
belong to the $\R_{>0}$ orbit corresponding to the geodesic $\left(\widehat\Lambda_t\right)_{t\in[0,1]}$ by Lemma~\ref{lemma:geodesic of cones}. In particular, $\widehat Z_0$ is nowhere tangent to the Euler vector field.
Recalling Lemma~\ref{lemma:Weinstein neighborhood of cylinder}, choose a Weinstein neighborhood $(W,\psi)$ of $\widehat Z_0$ compatible with $\widehat\Lambda_{0}$ and $\widehat \Lambda_{1},$ where $W \subset T^*L$ and $\psi:W \to V$ with $\psi|_L = f.$ Let $\pi_L : T^*L \to L$ denote the projection.
Let $\alpha\in(0,1).$ For $u\in\cob{2,\alpha}(L),$ let $\operatorname{Graph}(du) \subset T^*L$ denote the graph. Let $0 \in \mathcal{W} \subset C^{2,\alpha}(L;\partial L)$ be an open neighborhood such that for $u \in \mathcal{W}$ we have $\operatorname{Graph}(du) \subset W.$ For $u \in \mathcal{W},$ let $j_u : L \to V$ be given by
\[
j_u = \psi \circ \left( \pi_L|_{\operatorname{Graph}(du)}\right)^{-1}.
\]		

Define a differential operator
		\[
		\mathcal{F}:\mathcal{W}\times[0,1)\to C^\alpha\left(L\right),\qquad(u,s)\mapsto*j_u^*\real\Omega_s,
		\]
		where $*$ denotes the Hodge star operator of the metric $f^*g_0.$ Since $\widehat Z_0$ is imaginary special Lagrangian, we have $\mathcal{F}(0,0)=0.$ The operator $\mathcal{F}$ is smooth and by Lemma~\ref{lemma:Laplacian}, the linearization in the directions of $\mathcal{W}$ is equal to the Riemannian Laplacian,
		\[
		d\mathcal{F}_{(0,0)}(u,0)=\Delta u,\qquad u\in C^{2,\alpha}\left(L;\partial L\right),
		\]
		which is an isomorphism $C^{2,\alpha}\left(L;\partial L\right)\to C^\alpha\left(L\right).$ By the implicit function theorem, for some $\epsilon>0$ and shrinking $\mathcal{W}$ if necessary, there exists a smooth map
\[
k:[0,\epsilon)\to\mathcal{W}
\]
such that for $s\in[0,\epsilon),$ the function $k(s)$ is the unique element of $\mathcal{W}$ satisfying $\mathcal{F}(k(s),s)=0.$ By elliptic-regularity (e.g.~\cite[Chapter~17]{gilbarg-trudinger}), $k$ is in fact a smooth map $[0,\epsilon)\to C^\infty\left(L;\partial L\right).$ The cylinder $\widehat{Z}_s: = [j_{k(s)}: L \to V]$ is $\Omega_s$-imaginary special Lagrangian for $s\in[0,\epsilon).$

The one-parameter family of cylinders of part~\ref{general regular family of cylinders} is given by
		\[
		Z_s:=[\mathbf{X}\circ M_s \circ j_{k(s)}: L \to X],\qquad s\in(0,\epsilon).
		\]
Indeed, all the cylinders $Z_s$ are imaginary special Lagrangian with respect to the Calabi-Yau form $\Omega,$ so it remains to show the regularity of the family $(Z_s)_s$ about $q.$ Define a map
		\[
		\Phi:S^{n-1}\times[0,1]\times[0,\epsilon)\to X,\qquad(p,t,s)\mapsto \mathbf{X}\circ M_s \circ j_{k(s)}(p,t).
		\]
We proceed to verify that $\Phi$ is a regular parameterization of $(Z_s)_s$ about $q.$ By construction, $\Phi_s := \Phi|_{S^{n-1}\times [0,1]\times \{s\}}$ parameterizes $Z_s$ for $s \in (0,\epsilon),$ so $\Phi|_{S^{n-1}\times [0,1]\times (0,\epsilon)}$ satisfies condition~\eqref{interior regularity first part}\ref{item: rep} of Definition~\ref{definition: interior regularity}. Condition~\eqref{regular convergence to intersection point}\ref{critical point} is satisfied because $\mathbf{X}(0) = q.$ Furthermore, since $j_0 = f,$ we have
\begin{equation}\label{eq:tangent family}
\left.\pderiv[\Phi]{s}\right|_{s = 0} = d\mathbf{X}_0 \circ f.
\end{equation}
So, $\Phi$ satisfies conditions~\eqref{regular convergence to intersection point}\ref{item:derivative Phi immersion} and \eqref{regular convergence to intersection point}\ref{item:nowhere tangent Euler} of Definition~\ref{definition: interior regularity}. So, by Lemma~\ref{rem: easy regularity}, possibly after diminishing $\epsilon,$ the map $\Phi$ is a regular parameterization of $(Z_s)_s$ about~$q.$ Recalling Definition~\ref{dfn:tangent family}, it also follows from equation~\eqref{eq:tangent family} that the tangent family of $(Z_s)_s$ is the cylindrical transform of $(C_t)_t.$

Next, we address uniqueness.  Let $(Y_s)_{s\in(0,\epsilon')}\subset\mathcal{SLC}\left(S^{n-1};\Lambda_0,\Lambda_1\right)$ be another family converging regularly to $q$ with tangent family the cylindrical transform of $(C_t)_t.$ We show that after reparameterization in $s$ and possibly shrinking $\epsilon,\epsilon',$ the families $(Y_s)_s$ and $(Z_s)_s$ coincide. Indeed, let $\Phi':S^{n-1}\times[0,1]\times[0,\epsilon')\to X$ be a regular parameterization of $(Y_s)_s$ about $q.$ Since the tangent families of $(Z_s)_s$ and $(Y_s)_s$ both coincide with the cylindrical transform of $(C_t)_t,$ after possibly reparameterizing $(Y_s)_s$ and composing $\Phi'$ with a diffeomorphism of $S^{n-1} \times [0,1] \times [0,\epsilon'),$ we may assume that
\begin{equation}\label{eq:pderphi'=}
\left.\pderiv[\Phi']{s}\right|_{s = 0} = \left.\pderiv[\Phi]{s}\right|_s.
\end{equation}
Possibly shrinking $\epsilon',$ we may assume that $Y_s$ is contained in $U$ for $s \in (0,\epsilon').$ For $\rho : [0,\epsilon'') \to [0,\epsilon')$ a diffeomorphism with $\deriv[\rho]{s}(0) = 1,$ define $l^\rho_s : L \to V_s$ by
\[
l^\rho_s(p,t) := M_s^{-1} \circ \mathbf{X}^{-1}\circ \Phi'(p,t,\rho^{-1}(s)).
\]
By Lemma~\ref{lemma: milnor lemma}, equation~\eqref{eq:tangent family} and equation~\eqref{eq:pderphi'=}, the family $l^\rho_s$ extends smoothly to $s = 0,$ with
\begin{equation}\label{eq:lrho0}
l^\rho_0 = f_0.
\end{equation}
Possibly after composing $\Phi'$ with a diffeomorphism of $S^{n-1} \times [0,1] \times [0,\epsilon'),$ for $s \in [0,\epsilon'')$ there exists $u^\rho_s \in \cob{\infty}(L)$ such that $l^\rho_s = j_{u^\rho_s}.$ It follows from equation~\eqref{eq:lrho0} that $u^\rho_0 = 0.$

We claim that for an appropriate choice of $\rho,$ we have $u^\rho_s \in C^\infty(L;\partial L).$
Indeed, for $r \in [0,1]$ and $s \in [0,\epsilon''),$ define $l^\rho_{r,s} : L \to V$ by
\[
l^\rho_{r,s}(p,t) := r l^\rho_s(p,t).
\]
Let
\[
\widehat Y^\rho_{r,s} := [l^\rho_{r,s} : L \to V] \in \mathcal{LC}(S^{n-1};\widehat \Lambda_0,\widehat\Lambda_1).
\]
For $s \in [0,\epsilon'),$ let
\[
F^\rho_s : = \operatorname{RelFlux}\left(\left(\widehat Y^\rho_{r,s}\right)_{r \in (0,1]}\right).
\]
Since $M_s^*\omega_0 = s^2 \omega_0,$ we have
\[
F^{\rho}_{\rho(s)}/F^{\id}_{s} = s^2/\rho(s)^2.
\]
Choosing
\[
\rho(s) := s\frac{\sqrt{F_s^{\id}}}{\sqrt{F_0^{\id}}},
\]
it follows that $F^\rho_{\rho(s)}$ is constant in $s,$ and thus $F^\rho_s$ is also constant in $s.$
Since $l^\rho_{0,s} = 0$ for $s \in [0,\epsilon''),$ it follows from Remark~\ref{rem: flux} \eqref{it: flux homotopy invariant} that
\[
\operatorname{RelFlux}\left(\left(\widehat Y^\rho_{1,s}\right)_{s \in [0,s_1]}\right) = 0, \qquad s_1 \in [0,\epsilon'').
\]
Let $A_s \in \R$ be the unique constant such that $\left.\deriv[u^\rho_s]{s} \right|_{S^{n-1}\times\{1\}} \equiv A_s.$ By Definition~\ref{dfn: flux}, we have
\[
\operatorname{RelFlux}\left(\left(\widehat Y^\rho_{1,s}\right)_{r \in [0,s_1]}\right) =  \int_0^{s_1}A_s.
\]
Combining the preceding two equations, we obtain
\[
\int_0^{s_1}A_s = 0, \qquad s_1 \in [0,\epsilon'').
\]
Differentiating with respect to $s_1,$ it follows from the fundamental theorem of calculus that $A_s = 0$ for all $s \in [0,\epsilon'').$ Since $u^\rho_0 = 0,$ we obtain $u_r \in C^{\infty}(L;\partial L)$ as desired.

By construction, the cylinder
\[
\widehat Y^\rho_s := [l^\rho_s : L \to V_s]
\]
is $\Omega_s$-imaginary special Lagrangian. So, $\mathcal{F}(u^\rho_s,s) = 0$ and by the uniqueness of $k(s),$ it follows that $u^\rho_s = k(s).$ It follows that $Y_{\rho^{-1}(s)} = Z_s.$
Thus, we have established part~\ref{general regular family of cylinders} of the lemma.

We construct the geodesic $(U_t)_t$ of open positive Lagrangians of part~\ref{general geodesic of open C1 Lags} by combining Lemma~\ref{lemma: geodesic of small Lagrangians} and part~\ref{general regular family of cylinders}. The tangent cones $TC_{\hat q_t} U_t$ coincide with $C_t$ by Lemma~\ref{lemma:commute}. The uniqueness claim follows from Lemma~\ref{lemma:commute}, the uniqueness claim of Lemma~\ref{lemma: geodesic of small Lagrangians} and the uniqueness claim of part~\ref{general regular family of cylinders}.

	\end{proof}
	
\begin{proof}[Proof of Theorem~\ref{theorem:spider web}]
By Theorem~\ref{theorem:linear geodesic} there exists a geodesic of Lagrangian linear subspaces $(\lambda_t)_{t \in [0,1]} \subset T_qX$ with $\lambda_i = T_q\Lambda_i$ for $i = 0,1.$ Applying Lemma~\ref{lemma:integration of tangent cone} with $C_t = \lambda_t,$ we obtain a regular family of cylinders as in part~\ref{regular family of cylinders} of the theorem and a geodesic of open positive Lagrangians $(U_t)_t$ as in part~\ref{geodesic of open C1 Lags} of the theorem. Since $TC_{\hat q_t} U_t = \lambda_t$ is linear, it follows from Lemma~\ref{lemma:C11 regularity from blowup} that $(U_t)_t$ has regularity $C^{1,1}.$
\end{proof}

	\bibliographystyle{amsabbrvcnobysame}
	\bibliography{bibli}

\end{document}